\documentclass[12pt,a4paper]{article}
\usepackage[utf8]{inputenc}
\usepackage[english]{babel}
\usepackage{indentfirst}
\usepackage{amssymb}
\usepackage{misccorr}
\usepackage{amsmath}
\usepackage{amsthm}
\usepackage{amsopn}
\usepackage{mathdots}
\ifx\pdfoutput\undefined
\usepackage{graphicx}
\else
\usepackage[pdftex]{graphicx}
\fi
\usepackage{fancyhdr}
\newcommand\restr[2]{{
  \left.\kern-\nulldelimiterspace 
  #1 
  \right|_{#2} 
}}
\usepackage{amsfonts, amssymb, amsthm, amsmath, euscript, hhline}
\graphicspath{{images/}}
\newtheorem{defn}{Definition}[section]
\theoremstyle{plain}
\newtheorem{theorem}{Theorem}[section]
\newtheorem{prop}[theorem]{Proposition}
\newtheorem{lemma}[theorem]{Lemma}
\newtheorem{corollary}[theorem]{Corollary}

\newtheorem*{exam}{Example}

\begin{document}
	
\title{$\phantom{a}$\\[-1.7in]
\\[0.2in]
{\bf Doctoral Thesis, Draft 04} \\[0.1in]
\bf\textit{Categories $\bggO$ for Dynkin Borel Subalgebras} \\[0.1in]
\bf\textit{of Root-Reductive Lie Algebras}\\[0.1in]
{\rm \normalsize by}\\[0.1in]}

\author{ {\bf Aleksandr Fadeev}  \\[0.1in]
Department of Mathematics and Logistics \\[0.1in]
Focus Area Mobility \\[0.1in]
Jacobs University Bremen\\[0.1in]
{\small \href{mailto:t.nampaisarn@jacobs-university.de}{t.nampaisarn@jacobs-university.de}}\\[0.1in]
{\normalsize A Thesis submitted in partial fulfilment}\\[0.1in]
{\normalsize of the requirements for the degree of}\\[0.1in]
{\bf Doctor of Philosophy}\\[0.1in]
{\bf in Mathematics}
}
\date{\today}

\clearpage

\begin{titlepage}
    
    \begin{center}
        \vspace{0.2in}
        {\textbf{\textit{\Large Classification of  primitive ideals of }}} \\[0.1in]
{\textbf{\textit{\Large $U(\mathfrak{o}(\infty))$ and $U(\mathfrak{sp}(\infty))$}}}\\[0.2in]
{ \normalsize by}\\[0.2in]
        {\textbf{Aleksandr Fadeev}}  \\[0.1in]

        \vspace{0.2in}
        {\normalsize A Thesis Submitted in Partial Fulfilment}\\[0.1in]
{\normalsize of the Requirements for the Degree of}\\[0.2in]
{\textbf{Doctor of Philosophy}}\\[0.1in]
{\textbf{ in Mathematics}}\\[0.3in]
    \end{center}
    
    \begin{flushright}
   {\textbf{ Approved, Thesis Committee}}\\[0.2in]
    
   {Prof. Dr. Ivan Penkov}\\[-0.1in]
    
    \noindent\rule{4,5in}{0.4pt}
    
    Chair, Jacobs University Bremen\\[0.2in]
    
  {Prof. Dr. Alan Huckleberry}\\[-0.1in]
    
    \noindent\rule{4,5in}{0.4pt}
    
    Jacobs University Bremen / Ruhr-Universit\"at Bochum\\[0.2in]
  {Dr. Alexey Petukhov}\\[-0.1in]
    
    \noindent\rule{4,5in}{0.4pt}
    
 Institute for Information Transmission Problems in Moscow\\[0.3in]
    
    \end{flushright}
    \vfill
    \begin{flushleft}
    	\underline{Date of Defense: {\textit{Wednesday, December 04, 2019}\textup{      \:\;\;\;\;\;\;\;\;  \:\;\;\;\;\;\;\;\;  \:\;\;\;\;\;\;\;\;  \:\;\;\;\;\;\;\;\;            \:\;\;\; }}}\\

Department of Mathematics and Logistics \\[0.1in]
    \end{flushleft}

\end{titlepage}
\thispagestyle{empty}

\pagebreak

\newpage

\begin{abstract}\thispagestyle{plain}
 The purpose of  this Ph.D. thesis is to study and classify primitive ideals of the enveloping algebras $U(\mathfrak{o}(\infty))$ and $U(\mathfrak{sp}(\infty))$. Let $\mathfrak{g}(\infty)$ denote any of the Lie algebras  $\mathfrak{o}(\infty)$ or $\mathfrak{sp}(\infty)$. Then\break $\mathfrak{g}(\infty)=\bigcup_{n\geq 2} \mathfrak{g}(2n)$ for  $\mathfrak{g}(2n)=\mathfrak{o}(2n)$ or  $\mathfrak{g}(2n)=\mathfrak{sp}(2n)$, respectively. 
 We show that each primitive ideal $I$ of  $U(\mathfrak{g}(\infty))$ is weakly bounded, i.e.,   $I\cap U(\mathfrak{g}(2n))$ equals the intersection of  annihilators of bounded weight  $\mathfrak{g}(2n)$-modules. To every primitive ideal $I$ of $\mathfrak{g}(\infty)$ we attach  a unique irreducible coherent local system of bounded ideals, which is an analog of a coherent local system of finite-dimensional modules, as introduced earlier by A. Zhilinskii. As a result, primitive ideals of $U(\mathfrak{g}(\infty))$ are parametrized by triples $(x,y,Z)$ where $x$ is a nonnegative integer, $y$ is a nonnegative integer or half-integer, and $Z$ is a Young diagram. In the case of $\mathfrak{o}(\infty)$,  each primitive ideal is integrable, and our classification reduces to a classification of integrable ideals going back to A. Zhilinskii, A. Penkov and I. Petukhov. In the case of $\mathfrak{sp}(\infty)$, only 'half' of the primitive ideals are integrable, and nonintegrable primitive ideals correspond to triples $(x,y,Z)$ where $y$ is  a half-integer.  
\end{abstract}

\mbox{}
\thispagestyle{empty}
\newpage

\begin{center}{\textbf{Acknowledgment}}\end{center}

This Ph.D. research has been carried out under the supervision of Prof. Ivan Penkov, Department of
Mathematics and Logistics, Jacobs University Bremen.
I want to thank Ivan Penkov for introducing me to the topic and for	   great help in editing the text.

I  want to thank Aleksey Petukhov for sharing his ideas with me during several discussions in the course of my work.

I am grateful to Mikhail Ignatiev for  big help in editing the text.

I want to thank Dimitar Grantcharov for suggesting several helpful references.

I wish to acknowledge the DFG (Deutsche Forschungsgemeninschaft) for financial
support during my three-year doctoral study. Without this support, my work would not have been
possible. I would also like to thank the entire Department of Mathematics and Logistics at
Jacobs University Bremen for the opportunity to have nice discussions on various topics in
mathematics.

Finally, I would like acknowledge my wife for her constant support throughout my studies.

\pagebreak

\begin{center}
{\textbf{Declaration}}

\end{center}

I hereby declare that the thesis submitted was created and written solely by myself without 
any external support. Any sources, direct or indirect, are marked as such. I am aware of the 
fact  that  the contents of  the  thesis in digital  form may be  revised with  regard  to  usage of 
unauthorized aid as well as whether the whole or parts of it may be identified as plagiarism. I 
do agree my work to be entered into a database for it to be compared with existing sources, 
where it will remain in order to enable further comparisons with future theses. This does not 
grant any rights of reproduction and usage, however.    
The Thesis has been written independently and has not been submitted at any other university 
for the conferral of a PhD degree; neither has the thesis been previously published in full. 
\\[0.5in]

\noindent Signature: \rule{1.5in}{0.4pt}$\phantom{aaa}$Name: Aleksandr Fadeev$\phantom{aaa}$\\Date: \today

\pagebreak

\newpage

		\newpage
	\tableofcontents

\newpage
	\section{Introduction}
In this  work, we classify the primitive ideals of the universal enveloping algebras of the infinite-dimensional  Lie algebras $\mathfrak{o}(\infty)$ and $\mathfrak{sp}(\infty)$.

\vspace{1em}
 Let $\mathfrak{g}(\infty)$  denote one of the Lie algebras   $\mathfrak{sl}(\infty)$,       $\mathfrak{o}(\infty)$ and $\mathfrak{sp}(\infty)$,         and let              $\mathfrak{g}(n)$ denote the respective finite-dimensional Lie algebra   $\mathfrak{sl}(n)$, $\mathfrak{o}(2n)$,  $\mathfrak{o}(2n+1)$ or $\mathfrak{sp}(2n)$.
Then the Lie algebra $\mathfrak{g}(\infty)$ is isomorphic to the direct limit $\varinjlim\mathfrak{g}(n)$ for  certain natural embeddings $\mathfrak{g}(n)\hookrightarrow \mathfrak{g}(n+1)$.

\vspace{1em}
Classifying the primitive ideals of the universal enveloping algebra\break $U(\mathfrak{g}(\infty))$ is an important structural problem
  for $U(\mathfrak{g}(\infty))$ as an associative algebra, and is also a fundamental problem     in the representation  theory of $\mathfrak{g}(\infty)$. 
Indeed, for any Lie algebra $\mathfrak{g}$, the annihilator in $U(\mathfrak{g})$ of a simple $\mathfrak{g}$-module is a primitive ideal, but experience shows that
even for a simple finite-dimensional Lie algebra $\mathfrak{g}$ the problem of classifying primitive ideals  in $U(\mathfrak{g})$   is  tractable, while the problem  of classifying simple $\mathfrak{g}$-modules when rank of $\mathfrak{g}$ the large enough,  is untractable or 'wild'.   
The simple objects of several important categories of $\mathfrak{g}$-modules have been classified. This concerns category $\mathcal{O}$,  the category of Harish--Chandra modules, the category of weight modules of finite type, and some other categories, but there is no known approach to  a classification of arbitrary simple $\mathfrak{g}$-modules. 
On the other hand, the classification of primitive ideals for a simple Lie algebra $\mathfrak{g}$ is now one of the cornerstones of the representation theory of simple, or semisimple, finite-dimensional Lie algebras.

\vspace{1em}
 Let me discuss this classification in more detail. The starting point  is  Duflo's Theorem, see \cite{D}, which claims that, given a simple $\mathfrak{g}$-module, there exists  a simple highest weight $\mathfrak{g}$-module with the same annihilator.
 Hence, for the classification of primitive ideals, it is enough to  classify the  annihilators of all simple highest weight  $\mathfrak{g}$-modules. After this, it remains to understand when two different highest weight modules  have the same annihilators. A sufficient condition for  this  was given by A. Joseph \cite{J1}, and for $\mathfrak{g}=\mathfrak{sl}(n)$  Joseph was able to solve the problem completely. 
It turned out, that for integral weights $\lambda$ and $\mu$ lying in one Weyl group orbit, the corresponding primitive ideals coincide precisely when the insertion tableaux in the outputs of the Robinson--Schensted algorithm applied to $\lambda$ and $\mu$ coincide. We recall that this is a combinatorial algorithm  which attaches to each permutation two Young tableaux: the insertion tableau and the recording tableau.\break
 For example, let 
$$\delta=\biggl(\begin{array}{ccccccccc}
 1& 2& 3& 4& 5& 6& 7 &8 & 9 \\
6& 5& 1& 2& 8& 3& 7 &4 & 9
\end{array}\biggr)\in S_9$$  Then the output of the Robinson--Schensted algorithm applied to $\delta$ is 
$$Y=
\begin{array}{ccc}
\hline\multicolumn{1}{|c|}{9 }&\multicolumn{1}{|c|}{7 }&\multicolumn{1}{|c|}{4 }\\
\hline\multicolumn{1}{|c|}{8}&\multicolumn{1}{|c|}{5}&\multicolumn{1}{|c|}{2}\\
\hhline{---}\multicolumn{1}{|c|}{6}&&\\
\hhline{-~~}\multicolumn{1}{|c|}{2}&&\\
\hhline{-~~}\multicolumn{1}{|c|}{1}&&\\
\hhline{-~~}
\end{array}, \:
Y'=
\begin{array}{ccc}
\hline\multicolumn{1}{|c|}{9 }&\multicolumn{1}{|c|}{8 }&\multicolumn{1}{|c|}{7 }\\
\hline\multicolumn{1}{|c|}{6}&\multicolumn{1}{|c|}{4}&\multicolumn{1}{|c|}{2}\\
\hhline{---}\multicolumn{1}{|c|}{5}&&\\
\hhline{-~~}\multicolumn{1}{|c|}{3}&&\\
\hhline{-~~}\multicolumn{1}{|c|}{1}&&\\
\hhline{-~~}
\end{array}$$
where $Y$ is called the insertion tableau and $Y'$ is called the recording tableau.
For a detailed description of this algorithm and more examples,  see Subsection \ref{sub:RS}.

\vspace{1em}

Another important result of Joseph, closely related to the classification of primitive ideals, is that the associated variety of a primitive ideal coincides with the closure of a nilpotent coadjoint orbit.
\vspace{1em}

  The classification of primitive ideals of $U(\mathfrak{g})$ for $\mathfrak{g}(n)$ equal to $\mathfrak{o}(2n)$, $\mathfrak{o}(2n+1)$ or  $\mathfrak{sp}(2n)$ was completed by
D. Barbash and D. Vogan in their work \cite{BV}. As we pointed out above, two annihilators of simple $\mathfrak{sl}(n)$-modules $L(\lambda)$ and $L(\mu)$, with  respective highest weights $\lambda$ and $\mu$, coincide if and only if the insertion tableaux in the outputs of Robinson--Schensted algorithm applied $\lambda$ and $\mu$ coincide. This is no longer true for  $\mathfrak{o}(2n)$, $\mathfrak{o}(2n+1)$ or  $\mathfrak{sp}(2n)$. In this case, consider a primitive ideal  $I=\operatorname{Ann}(L(\lambda))$ for a simple module $L(\lambda)$ with highest weight $\lambda$. Then  Barbash and  Vogan find a weight $\gamma$ such that $I=\operatorname{Ann}(L(\gamma))$ and the insertion tableau $Y$ in the output of Robinson--Schensted algorithm applied to  $\gamma$ 
has the property that the lengths of the  rows of $Y$ are equal to the sizes of  Jordan cells for a nilpotent matrix from the  associated variety of $I$. In this way, primitive ideals of $U(\mathfrak{o}(2n))$, $U(\mathfrak{o}(2n+1))$ and $U(\mathfrak{sp}(2n))$ can also be parameterized by (certain)  Young tableaux.

\vspace{1em} Now, let us turn to the infinite-dimensional setting.
 First, we will briefly recall the classification of primitive ideals of the universal enveloping algebra for the infinite-dimensional Lie algebra $\mathfrak{sl}(\infty)$, due  to   I. Penkov  and A. Petukhov. In the series of works \cite{PP1}, \cite{PP2}, \cite{PP3}, \cite{PP4},  they provided such a classification, and in the paper \cite{PP5}  this classification was related to an infinite-dimensional  analogue of the Robinson--Schensted algorithm.
\vspace{1em}
 
We should note that in the case of  $\mathfrak{sl}(\infty)$ there are some important differences to the finite-dimensional case. First, it easy to prove that the center  of the universal enveloping algebra of  the Lie algebra  $\mathfrak{sl}(\infty)$ consists of constants only. Second, a 'generic' simple  $\mathfrak{sl}(\infty)$-module has zero annihilator, and in \cite{PP2} a criterion for a simple highest weight $\mathfrak{sl}(\infty)$-module to have nonzero annihilator is established.
Next,  every primitive ideal for $\mathfrak{sl}(\infty)$ is integrable, which is wrong in the finite-dimensional case. Recall that an integrable  ideal $I\subset U(\mathfrak{g}(\infty))$ is by definition the annihilator of an integrable $\mathfrak{g}(\infty)$-module $M$, i.e., of a $\mathfrak{g}(\infty)$-module $M$ which, when restricted to $\mathfrak{g}(n)$ for any $n$, is  a sum of finite-dimensional $\mathfrak{g}(n)$-modules.
Finally, the integrability of a primitive ideal $I\subset U(\mathfrak{sl}(\infty))$ does not imply  that any simple $\mathfrak{sl}(\infty)$-module $M$, whose annihilator equals $I$, is integrable.
\vspace{1em}

  The integrability of a primitive ideal $I\subset U(\mathfrak{sl}(\infty))$ was proved by
 Penkov and  Petukhov in the work \cite{PP4}. This  reduced  the classification of primitive ideals  to a classification of  annihilators of simple integrable $\mathfrak{sl}(\infty)$-modules. This latter classification had already been carried out in the papers \cite{PP1}, \cite{PP2}, \cite{PP3} (without classifying simple integrable $\mathfrak{sl}(\infty)$-modules!) and relies essentially on work of A. Zhilinskii.

\vspace{1em}
In a series of works \cite{Zh1}, \cite{Zh2}, \cite{Zh3}, Zhilinskii introduces and classifies  certain  combinatorial data  which he called coherent local systems, c.l.s..  In the work \cite{PP3},    Penkov and Petukhov establish a bijection between some irreducible c.l.s.  and  integrable primitive ideals. In the cases of $\mathfrak{o}(\infty)$ and $\mathfrak{sp}(\infty)$,  integrable primitive ideals of $U(\mathfrak{sl}(\infty))$ are in one-to-one correspondence with all irreducible c.l.s.  Next, in the paper \cite{PP4}   Penkov and Petukhov introduce the notion  of a precoherent local system (p.l.s.) and  prove that every primitive ideal of $U(\mathfrak{sl}(\infty))$ is  integrable.
\vspace{1em}

 More precisely, let $V(n)$ be the natural $\mathfrak{sl}(n)$-module, and let $V=\varinjlim V(n)$ be the natural $\mathfrak{sl}(\infty)$-module.  We denote by $S^{\bullet}(V)$ and $\Lambda^{\bullet}(V)$ the symmetric algebra and  the exterior algebra of $V$ respectively. 
It turns out that every primitive ideal of $U(\mathfrak{sl}(\infty))$  has the form $$I(x,y,Y_l,Y_r):=\operatorname{Ann}_{U(\mathfrak{sl}(\infty))}(V_{Y_l}\otimes (S^{\bullet}(V))^{\otimes x}\otimes  (\Lambda^{\bullet}(V))^{\otimes y} \otimes (V_{Y_r})_*),$$
where $x,y\in\mathbb{Z}_{\geq 0}$, $Y_l$ and $Y_r$ are  Young diagrams, and the modules $V_{Y_l}$ and $(V_{Y_r})_*$ are constructed as follows.
Let $Y$ be a Young diagram with row lengths 
$$l_1\geq l_2 \geq \dots \geq l_s >0.$$ 
Then for $n\geq s$ we denote by $V_Y(n)$ the $\mathfrak{sl}(n)$-module  with highest weight
$$\underbrace{(l_1,~l_2,~\ldots,~l_s,~0,~0,~\ldots,~0)}_{n\text{ numbers}},$$ and note that the modules $V_Y(n)$ are nested:
$V_Y(n)\hookrightarrow V_Y(n+1)$.
This allows us to define the $\mathfrak{sl}(\infty)$-module $V_Y$ as the direct limit $\varinjlim V_Y(n)$.  Finally, we put $$(V_Y)_*=\varinjlim (V_Y(n))^*,$$ where $(V_Y(n))^*$ is the module   dual to $V_Y(n)$.
\vspace{1em}

As stated above, the  goal in the present work is to classify the primitive ideals of the universal enveloping algebras  of  $\mathfrak{o}(\infty)$ and $\mathfrak{sp}(\infty)$. For $\mathfrak{o}(\infty)$,  Penkov and  Petukhov conjectured   that every primitive ideal is  integrable. Proving this conjecture, reduces the classification of primitive ideals of $U(\mathfrak{o}(\infty))$  to the known classification \cite{PP3} of integrable primitive ideals of $U(\mathfrak{o}(\infty))$.
On the other hand, in the case of $\mathfrak{sp}(\infty)$ not every primitive ideal is integrable. Indeed, consider the simple $\mathfrak{sp}(2n)$-modules $SW^+(2n)$ and $SW^-(2n)$ (the Shale--Weil modules) with respective highest weights $(-\frac{1}{2}, -\frac{1}{2},\dots,-\frac{1}{2}, -\frac{1}{2})$ and $(-\frac{1}{2}, -\frac{1}{2},\dots,-\frac{1}{2}, -\frac{3}{2})$. One can check that the direct limits $\varinjlim SW^+(2n)$ and  $\varinjlim SW^-(2n)$ are well-defined  $\mathfrak{sp}(\infty)$-modules. In the work \cite{PP3}  Penkov and  Petukhov prove that the annihilators in $U(\mathfrak{sp}(\infty))$ of these modules coincide, and constitute  a nonintegrable primitive ideal. 
\vspace{1em}


For the $\mathfrak{sp}(\infty)$-case  Penkov and  Petukhov provided me with a conjectural  construction of  all primitive ideal of $U(\mathfrak{sp}(\infty))$  by using a generalization of the notion of c.l.s. In this work,  I prove  this conjecture, as well as the conjecture that all primitive ideals of $U(\mathfrak{o}(\infty))$ are integrable.

\vspace{1em}
 In what follows, I describe the contents of  this dissertation.
\vspace{1em}

In  Section \ref{sec:pre} we give most necessary definitions, as well as known statements which  are used later in this work. Section \ref{3} is devoted to the proof of the fact  that every
primitive ideal of $U(\mathfrak{o}(\infty))$ or $U(\mathfrak{sp}(\infty))$ is weakly bounded. In addition, it turns out that all primitive ideals of   $U(\mathfrak{o}(\infty))$ and some primitive ideals  of $U(\mathfrak{sp}(\infty))$  are locally integrable. In Section \ref{4} we prove that every locally integrable ideal of $U(\mathfrak{o}(\infty))$ and $U(\mathfrak{sp}(\infty))$ is integrable.

\vspace{1em}

 In Section \ref{5} we introduce  the notions of a coherent local system of bounded ideals (c.l.s.b.) and a precoherent local system of bounded ideals  (p.l.s.b.), which generalize the notions of a coherent local system and a precoherent local system of finite-dimensional representations, respectively. 
\vspace{1em}

A new combinatorial tool appearing in the case of $\mathfrak{sp}(\infty)$ are the Kazhdan--Lusztig polynomials. Each Kazhdan--Lusztig polynomial is defined by two elements of the Weyl group of $\mathfrak{sp}(2n)$ for some $n$ (in general, of a Coxeter group); for a more precise definition, see \cite{H} or Subsection \ref{sub:KL}.  
 It is a remarkable fact that the Kazhdan--Lusztig polynomials corresponding to a bounded simple highest weight  $\mathfrak{sp}(2n)$-module $L(\lambda)$  are equal to the respective Kazhdan--Lusztig polynomials corresponding to the simple finite-dimensional   $\mathfrak{o}(2n)$-module $L(\lambda')$ where $\lambda'=\lambda+\sum^n_{i=1}\varepsilon_i$.  Using this,    we establish a one-to-one correspondence between  the set of c.l.s.b. for $\mathfrak{sp}(\infty)$ and the set of c.l.s.b for $\mathfrak{o}(\infty)$, the latter set being equal to the set of c.l.s. for $\mathfrak{o}(\infty)$.
\vspace{1em}

Finally, in  Section \ref{6} it is proved that each nonzero primitive ideal \break$I\subsetneq \operatorname{U}(\mathfrak{sp}(\infty))$ is the annihilator of a unique $\frak{sp}(\infty)$-module of the form
$$(\operatorname{S}^\bullet(V))^{\otimes x}\otimes ({ \Lambda}^\bullet(V))^{\otimes y}\otimes V_{Z} \text{
or }
(\operatorname{ S}^\bullet(V))^{\otimes x}\otimes({ \Lambda}^\bullet(V))^{\otimes y}\otimes V_{Z}\otimes R$$
where $x, y\in\mathbb Z_{\ge0}$, $V$ is the natural $\mathfrak{sp}(\infty)$-module, $V_Z$ is the $\mathfrak{sp}(\infty)$-module defined analogously to the $\mathfrak{sl}(\infty)$ case
(for $Z$ is arbitrary Young diagram), and $R$ is the Shale--Weil $\mathfrak{sp}(\infty)$-module which is equal to the direct limit $\varinjlim SW^+(2n)$.

\vspace{1em}
One may note that the notion of a bounded primitive ideal of $U(\mathfrak{o}(\infty))$ and $U(\mathfrak{sp}(\infty))$ is well-defined. This can be deduced for instance
 from work of  I. Penkov and  V. Serganova \cite{PS}. Furthermore,   the work \cite{GP} of \break D.~Grantcharov and I.~Penkov shows that the only nonintegrable bounded ideal of  $ U(\mathfrak{sp}(\infty))$ is the annihilator of module $\varinjlim SW^+(2n)$. By analogy with the finite-dimensional case, this ideal should be called Joseph ideal. In this way, all nonintegrable primitive ideals of $U(\mathfrak{sp}(\infty))$ are weakly bounded but all of them, with one exception, are not bounded. 

	\newpage
	\section{Preliminaries}\label{sec:pre}

	All Lie algebras and vector spaces are defined over  $\mathbb{C}$. Here we outline some of the preliminaries needed in the sequel.  Let $\mathfrak{g}$ be a finite- or countable-dimensional Lie algebra, and $g,g_0\subset\mathfrak{g}$. Then
$$\operatorname{ad}_g\colon\mathfrak{g}\to\mathfrak{g}, \quad \operatorname{ad}_g(g_0)=[g,g_0]$$
is the adjoint representation of $\mathfrak{g}$.
\emph{The universal enveloping algebra} $U(\mathfrak{g})$ of a Lie algebra $\mathfrak{g}$ is the quotient algebra
$$U(\mathfrak{g}):=T(\mathfrak{g})/I,$$
where $$T(\mathfrak{g})=\mathbb{C} \oplus \mathfrak{g}\oplus(\mathfrak{g}\otimes\mathfrak{g})\oplus(\mathfrak{g}\otimes\mathfrak{g}\otimes\mathfrak{g})\oplus\dots $$
is the tensor algebra of $\mathfrak{g}$, and $I$ is the two-sided ideal of $T(\mathfrak{g})$ generated by all elements of the form
$$a\otimes b-b\otimes a -[a,b]$$ for $a,b\in \mathfrak{g}$.

For a vector space $V$, we define the \emph{dual space}
$$V^*:=\operatorname{Hom}(V,\mathbb{C}).$$ 
Let $V$ be a vector space and $S\subset V$ be a subset of $V$. Then \emph{linear span $\operatorname{span}S$ of} $S$ in $V$ defined as
$$\operatorname{span}S=\bigcap V',$$
where the intersection is taken over all subspaces $V'\subset V$ such that $V'\supset S$.
Let $F$ be a finite set. Then $\# F$ denotes the number of elements of $F$.

Next, we introduce the associative algebra  $\operatorname{Mat}_n=\operatorname{Mat}_n(\mathbb{C})$  of all $(n\times n)$-matrices over  $\mathbb{C}$.
Also we fix the special basis   of the vector space $\operatorname{Mat}_n$, where $e_{ij}$ is the elementary matrix

\[{e}_{ij} = \left(\begin{array}{ccccc}
	0& \dots& 0&\dots & 0 \\
	\vdots &\ddots& \vdots&\iddots&\vdots\\
	0 &\dots& 1_{ij}&\dots&0\\
	\vdots &\iddots&\vdots&\ddots&\vdots\\
	0 &\dots&0&\dots&0
	\end{array}\right). \]
 The  $(i,j)$-th entry of $e_{i,j}$ equals $1$, while all other entries are zero.


Let $X=\{x_{ij}\}$ be a $(n\times n)$-matrix. Then we put
$$\operatorname{tr}X=\sum^n_{i=1}x_{ii} \text{ and } X^t=\{a_{ij}\mid a_{ij}=x_{ji}\}.$$

There is a symmetric bilinear form $(\cdot,\cdot)$ on the Lie algebra $\mathfrak{g}$, defined by
$$(x,y)=\operatorname{tr}(\operatorname{ad}_x \operatorname{ad}_y)$$
for $x,y\in \mathfrak{g}$. This symmetric bilinear form is called the \emph{Killing form on} $\mathfrak{g}$.

	\subsection{Lie algebras with root decomposition }

Here we introduce some basic definitions and facts from the structure theory of the Lie algebras.

	Let $\mathfrak{g}$ be a finite- or countable-dimensional Lie algebra, and let $U(\mathfrak{g})$ denote the universal enveloping algebra of $\mathfrak{g}$.
	\begin{defn}\label{2.0}
		A Lie subalgebra $\mathfrak{h}\subset\mathfrak{g}$ is called a \textup{toral subalgebra of} $\mathfrak{g}$ if, for every nonzero element $h \in \mathfrak{h}$, the linear operator $\operatorname{ad}_h\colon\mathfrak{g}\to\mathfrak{g}$ is diagonalizable.
	\end{defn} 

\begin{lemma}\label{2.1}
	Each toral subalgebra $\mathfrak{h}$ of a finite-dimensional Lie algebra $\mathfrak{g}$ is abelian.
\end{lemma}

\begin{proof}
	Consider two nonzero elements $h_1, h_2 \in \mathfrak{h}$ and let $h_2 = \sum^n_{s=1} h^s_2$ be the decomposition of $h_2$ as a linear combination of $\operatorname{ad}_{h_1}$-eigenvectors with distinct eigenvalues $\lambda_s$. Then, $\operatorname{ad}_{h_1}(h_2) = \sum^n_{s=1} \lambda_s h^s_2$.  As $\mathfrak{h}$ is a subspace of $\mathfrak{g}$, every vector of the form
	$h'_i = \sum^n_{s=1}(\lambda_s-\lambda_i)h^s_2$ belongs to $\mathfrak{h}$.
	The vector $h'_i$ decomposes as a sum of $n-1$ $\operatorname{ad}_{h_1}$-eigenvectors with distinct eigenvalues, and induction on $n$ shows that in fact every vector $h^s_2$ belongs to $\mathfrak{h}$. Next, note that $[h^s_2,[h^s_2,h_1]] = [h^s_2,-\lambda_s h^s_2] = 0$. Since $\operatorname{ad}_{h^s_2}$  is diagonalizable for each $s$ and $\operatorname{ker}_{\operatorname{ad}_{h^s_2}}$ = $\operatorname{ker}_{(\operatorname{ad}_{h^s_2})^2}$, we conclude that   $[h^s_2,h_1]=0$ for all $s$. Hence,  $[h_2,h_1]=0$. 
\end{proof}

We note that, as a corollary of Lemma	\ref{2.1} and Definition \ref{2.0}, every toral subalgebra of a finite-dimensional Lie algebra is diagonalizable, i.e. all operators in it are simultaneously diagonalizable.

\begin{defn}
	Let $\mathfrak{h}$ be a toral subalgebra of $\mathfrak{g}$. Then $\mathfrak{h}$ is a \textup{ maximal toral subalgebra} if there is no  proper inclusion  $\mathfrak{h} \subset \mathfrak{h}'$ for a toral subalgebra $\mathfrak{h}'\in\mathfrak{g}$.
\end{defn}

	
%

	\begin{defn}
		A maximal toral subalgebra $\mathfrak{h}\subset\mathfrak{g}$ is  \textup {a splitting Cartan subalgebra} if 
		$\mathfrak{g}$ has the following \textup {$\mathfrak{h}$-module  decomposition}:

		\begin{equation}
		\label{decomp}
		\mathfrak{g}= \bigoplus_{\alpha \in \mathfrak{h}^*}\mathfrak{g}^\alpha 
		\end{equation}
		where $\mathfrak{g}^\alpha$ is the eigenspace  $\{x \in \mathfrak{g}\mid [ h,x
		]=\alpha(h)x: \: \: \forall h \in \mathfrak{h}\}$  and $ \mathfrak{g}^0 = \mathfrak{h}$.

	\end{defn}
	
	For a splitting Cartan subalgebra, the set of nonzero elements $\alpha$ appearing in the decomposition (\ref{decomp}) is called the \textit{root system} of $\mathfrak{g}$, or simply \textit{the set of roots of} $\mathfrak{g}$ with  respect to $\mathfrak{h}$. We denote the set of roots by $\Delta$.

	Let now $\Delta^+$, $\Delta^- \subset \Delta$ be two subsets satisfying the conditions

	$$ \Delta = \Delta^+ \sqcup \Delta^-, \quad -\Delta^+ = \Delta^-,\quad \alpha,  \beta\in \Delta^+, \alpha+\beta\in \Delta\Rightarrow \alpha+\beta\in\Delta^+.$$
	Given such subsets, we call $\Delta^+$  the set of \emph{positive  roots}, and $\Delta^-$  the set of \emph{negative  roots}. Then the $\mathbb{Z}$-submodule $\Lambda_{\Delta}$ of $\mathfrak{h}^*$ generated by $\Delta$ is called the \emph{root lattice} of the root system $\Delta$.

If the Lie algebra $\mathfrak{g}$ is finite dimensional, we introduce  the notation
$$\rho_{\mathfrak{g}}:=\sum_{\alpha\in \Delta^+}\alpha/2.$$

	\begin{defn}
		Let $\mathfrak{g}$ be a Lie algebra. A Lie subalgebra $\mathfrak{b}$ of
		$\mathfrak{g}$ is a \textup {splitting Borel subalgebra} if 
		
			$$ \mathfrak{b}= 
		\mathfrak{h}\oplus\bigoplus_{\alpha \in \Delta^+}\mathfrak{g}^\alpha$$
		for some splitting Cartan subalgebra $\mathfrak{h}$, and some subset $\Delta^+$ of positive roots.
	\end{defn}


	\begin{defn}
Let $\mathfrak{b}$  be a Borel subalgebra defining $\Delta^+$.	Then	an element $\alpha \in \Delta^+$ is said to be a \emph{simple $\mathfrak{b}$-positive root}, or a \emph{simple root} with respect to $\mathfrak{b}$, if it cannot be decomposed as a (finite) sum of two or more $\mathfrak{b}$-positive roots.  We usually use the symbol $\Sigma^+$ or $\Sigma$ for the set of all simple $\mathfrak{b}$-positive roots.
		Similarly, we say that $\alpha \in \Delta^-$ is a \emph{simple negative root} with respect to $\mathfrak{b}$ if it cannot be decomposed as a sum of two or more $\mathfrak{b}$-negative roots.  The symbol $\Sigma^-$ denotes the set of simple $\mathfrak{b}$-negative roots.
	\end{defn}
\newpage
\subsection{Some finite-dimensional Lie algebras}
\label{2.3}
Here we review some important examples of finite-dimensional  Lie algebras. 

\begin{enumerate}
	\item The Lie algebra $\mathfrak{g}=\mathfrak{gl}(n)=\mathfrak{gl}(n,\mathbb{C})$ is the Lie algebra of the algebra $\operatorname{Mat}_n$, where $[X,Y]=XY-YX$ for $X,Y \in \operatorname{Mat}_{n}$.

The \emph{ general linear Lie algebra} is the Lie algebra obtained from the associative algebra $\operatorname{Mat_{n}}$.
	
We can choose a splitting Cartan subalgebra $\mathfrak{h}_{\mathfrak{gl}(n)}$ as the algebra of all diagonal matrices in $\operatorname{Mat}_{n}$. Indeed, $\mathfrak{h}_{\mathfrak{gl}(n)}$ is clearly a maximal toral subalgebra of $\mathfrak{g}$. We have the following $\mathfrak{h}_{\mathfrak{gl}(n)}$-root decomposition:
	
	\begin{equation*}
	\label{gl_decomp}
	\mathfrak{g}= \mathfrak{h}_{\mathfrak{gl}({n})}\oplus\bigoplus_{1\leq i,j\leq n, i\neq j}\operatorname{span}\{e_{ij}\}.
	\end{equation*}
	
Consider the basis $b=\{e_{11},e_{22},\dots,e_{n\:n}\}$ of the Lie algebra $\mathfrak{h}_{\mathfrak{gl}(n)}$. The dual basis
 $$b^*=\{\varepsilon_{1},\varepsilon_{2},\dots,\varepsilon_{n}\}$$ 
of $\mathfrak{h}^*_{\mathfrak{gl}(n)}$ satisfies
$$\varepsilon_i(e_{jj})=\delta^i_j$$
for $1\leq i,j\leq n$. Here $\delta^i_j$ is  Kronecker's delta.

Thus the root system $\Delta_{\mathfrak{gl}(n)}$ of $\mathfrak{gl}(n)$ with respect to $\mathfrak{h}_{\mathfrak{gl}(n)}$ is
$$\Delta_{\mathfrak{gl}(n)}=\{\varepsilon_{i}-\varepsilon_{j}\mid\:i,j\in\mathbb{Z}_{>0},1\leq i\neq j\leq n\}.$$

We can choose 
\begin{equation}\label{new}
\Delta_{\mathfrak{gl}(n)}^+=\{\varepsilon_{i}-\varepsilon_{j}\mid\:i,j\in\mathbb{Z}_{>0}, 1\leq i<j\leq n \}
\end{equation}

as the set  of positive roots, with simple roots
$$\{\varepsilon_{1}-\varepsilon_{2},\varepsilon_{2}-\varepsilon_{3},\dots,\varepsilon_{n-1}-\varepsilon_{n}\}.$$

Then the splitting Borel subalgebra $\mathfrak{b}_{\mathfrak{gl}(n)}$ of $\mathfrak{gl}(n)$ containing $\mathfrak{h}_{\mathfrak{gl}(n)}$ and corresponding to $\Delta_{\mathfrak{gl}(n)}^+$ consists of all upper-triangular matrices.
 
The Lie algebra $\mathfrak{gl}(n)$ is not a simple Lie algebra, and we can split it as
$\mathfrak{gl}({n})=\mathfrak{sl}(n)\oplus\{\textup{scalar matrices}\}.$

Let's describe the Lie algebra $\mathfrak{sl}(n)$.

\item
 The Lie algebra $\mathfrak{g}=\mathfrak{sl}(n)=\mathfrak{sl}(n,\mathbb{C})$ is the Lie subalgebra of  $\mathfrak{gl}(n)$ consisting of all $g$ such that  $\operatorname{tr}{g}=0$.

The Lie algebra $\mathfrak{sl}({n})$ is called \emph{the special linear Lie algebra}.
	
We can choose as a Cartan subalgebra the algebra  
$$\mathfrak{h}_{\mathfrak{sl}(n)}=\operatorname{span}\{e_{ii}-e_{i+1\; i+1}\mid\: 1\leq i\leq n\}$$
of all diagonal matrices of $\mathfrak{sl}({n})$. 
Then we have the root decomposition

	\begin{equation*}
	\label{gl_decomp}
	\mathfrak{g}= \mathfrak{h}_{\mathfrak{sl}(n)}\oplus\bigoplus_{1\leq i,j\leq n, i\neq j}\operatorname{span}\{e_{ij}\}.
	\end{equation*}

Next, we define $\varepsilon_{i}$ as in the previous case. Note that $$\{\varepsilon_{i}-\varepsilon_{i+1}\mid 1\leq i \leq n\}$$  is the basis of $\mathfrak{h}^*_{\mathfrak{sl}(n)}$ dual to the basis $$\{\frac{e_{ii}}{2}-\frac{e_{i+1\: i+1}}{2}\mid1\leq i\leq n\}.$$
The root system $\Delta_{\mathfrak{sl}(n)}$ of $\mathfrak{sl}(n)$ with respect to $\mathfrak{h}_{\mathfrak{sl}(n)}$ is equal to the root system of $\mathfrak{gl}({n})$ with respect to $\mathfrak{h}_{\mathfrak{gl}(n)}$.

Let the positive and simple roots be as in the previous example.
The Weyl group $W_{\mathfrak{sl}(n)}$ is isomorphic to the permutation group in $n$ letters  $W_{\mathfrak{sl}(n)}\simeq S_n$. It acts on a weight $\lambda=\sum^n_{i=1}\lambda_i\varepsilon_i$ by permuting its coordinates:
$$w(\lambda)=\sum^n_{i=1}\lambda_{i}\varepsilon_{w(i)} \text{, }w\in W_{\mathfrak{sl}(n)}.$$

 The Lie subalgebra of upper triangular matrices in $\mathfrak{sl}(n)$ is a splitting Borel subalgebra  $\mathfrak{b}_{\mathfrak{sl}(n)}$, and $\Delta^+_{\mathfrak{sl}(n)}$ is given by the right-hand side of the  formula (\ref{new}).

\item

The Lie algebra $\mathfrak{o}(2n)=\mathfrak{o}(2n,\mathbb{C})$ is a Lie subalgebra of the Lie algebra $\mathfrak{gl}({2n})$. Fix the matrix
\[  F=\left( \begin{array}{ccccc}
	0&0&\dots&0&1 \\
	0&0&\dots&1&0\\
	\vdots&&\iddots&&\vdots\\
	0&1&\dots&0&0 \\
	1&0&\dots&0&0
	\end{array}\right). \]
Then the Lie algebra $\mathfrak{o}(2n)$ is
$$\mathfrak{o}(2n)=\{X\in \mathfrak{gl}(2n)\mid XF+FX^t=0\}.$$ It is a subalgebra because 
$$(X+Y)F+F(X+Y)^t=(XF+FX^t)+(YF+FY^t)=0,$$
and 
$$[X,Y]F+F[X,Y]^t=XYF-YXF+F(XY)^t-F(YX)^t=$$
$$=-XFY^t+YFX^t+FY^tX^t-FX^tY^t=$$
$$=(YF+FY^t)X^t-(XF+FX^t)Y^t=0,$$
where $X,Y\in \mathfrak{o}(2n)$.

Throughout the rest of this work, we index the columns and rows of matrices in the ambient Lie algebra $\mathfrak{gl}(2n)$ by    $$(-n,-n+1,\dots,-1,1,\dots, n-1,n).$$  The subalgebra 
$$\mathfrak{h}_{\mathfrak{o}{(2n})}=\operatorname{span}\{ v_i=e_{ii}-e_{-i-i}\mid\: 1\leq i\leq n\}$$ 
 of all diagonal matrices in $\mathfrak{o}(2n)$ is a splitting Cartan subalgebra of $\mathfrak{o}(2n)$. 

Let  $\{\varepsilon_{i}\}$ be  the  basis of $\mathfrak{h}_{\mathfrak{o}({2n})}^*$ dual to the basis $\{v_i\}$.
The root system $\Delta_{\mathfrak{o}({2n})}$ of $\mathfrak{o}(2n)$ with respect to $\mathfrak{h}_{\mathfrak{o}({2n})}$ is
$$\Delta_{\mathfrak{o}({2n})}=\{\pm(\varepsilon_{i}\pm \varepsilon_{j})\mid\: i,j\in\mathbb{Z}_{>0}, 1\leq i\neq j\leq n\}.$$

Let 
$$\Delta_{\mathfrak{o}({2n})}^+=\{\varepsilon_{i}\pm\varepsilon_{j}\mid\:i,j\in\mathbb{Z}_{>0}, 1\leq i<j\leq n \}$$
be the set of positive roots,
with the set of simple roots
\begin{equation}\label{sigmao}
\{\varepsilon_{1}-\varepsilon_{2},\varepsilon_{2}-\varepsilon_{3},\dots,\varepsilon_{n-1}-\varepsilon_{n},\varepsilon_{n-1}+\varepsilon_{n}\}.
\end{equation}

The splitting Borel subalgebra $\mathfrak{b}_{\mathfrak{o}({2n})}$ of $\mathfrak{o}(2n)$ with these positive roots  is the subalgebra of all upper-triangular matrices in $\mathfrak{o}(2n)$.


In the sequel 
we fix an inclusion
$\mathfrak{h}^*_{\mathfrak{o}({2n})}\hookrightarrow\mathfrak{h}^*_{\mathfrak{sl}({2n})}$
 defined by $\varepsilon_i\mapsto \bar\varepsilon_i -\bar\varepsilon_{-i}$, where $ \{\bar\varepsilon_1,\bar\varepsilon_2\dots, \bar\varepsilon_n,\bar\varepsilon_{-n},\dots ,\bar\varepsilon_{-2},\bar\varepsilon_{-1}\}$  is the basis of $\mathfrak{h}^*_{\mathfrak{gl}(2n)}$ dual to the basis $\{2e_{ii}\}$.

\newpage

Under this inclusion, a weight  $\lambda=\sum^n_{i=1}\lambda_i\varepsilon_i$ is mapped to
$$\sum^n_{i=1}\lambda_i\bar\varepsilon_i +\sum^{-n}_{j=-1}\lambda_{j}\bar\varepsilon_j$$
where $\lambda_j=-\lambda_{-j}$ for $j<0$.

Let $S_{2n}$ denote the symmetric group on the $2n$ letters $$-n,\ldots,-1,1,\ldots,n.$$ The Weyl group $W_{\mathfrak{o}(2n)}$ of the Lie algebra $\mathfrak{o}(2n)$ is isomorphic to the subgroup of $S_{2n}$ consisting of permutations ${w\in S_{2n}}$ such that $w(-i)=-w(i)$, $1\leq i\leq n$, for which the number $\#\{i>0\mid w(i)<0\}$ is even.
00

We will identify $W_{\mathfrak{o}(2n)}$ with this subgroup, and will use the usual two-line notation

$$g=\biggl(\begin{array}{ccccccccc}
-n & -n+1 & \dots & -1 &|& 1& \dots &n-1 & n \\
g(-n) & g(-n+1) & \dots & g(-1) &|& g(1)& \dots & g(n-1) & g(n)
\end{array}\biggr)$$
for an element $g\in W$. Note that,  if we identify a weight $\lambda=\sum^n_{i=1}\lambda_i\varepsilon_i$ with the sequence of integers $$(\lambda_1,\lambda_2,\dots,\lambda_n,\lambda_{-n},\dots,\lambda_{-2},\lambda_{-1} ),$$ then $g$ sends this sequence to $$(\lambda_{g^{-1}(1)}, \lambda_{g^{-1}(2)},\dots, \lambda_{g_{-1}(n)}, \lambda_{g_{-1}(-n)},\dots,    \lambda_{g^{-1}(-2)}, \lambda_{g^{-1}(-1)}           )    .$$

\item

The Lie algebra $\mathfrak{sp}(2n)=\mathfrak{sp}(2n,\mathbb{C})$ is a Lie subalgebra of  $\mathfrak{gl}({2n})$. Fix the matrix
\[  F=\left( \begin{array}{ccccc}
	0&0&\dots&0&1 \\
	0&0&\dots&1&0\\
	\vdots&&\iddots&&\vdots\\
	0&-1&\dots&0&0 \\
	-1&0&\dots&0&0
	\end{array}\right) \]
with $n$ $1$s and $(-1)$s on the antidiagonal.
Then the Lie algebra $\mathfrak{sp}({2n})$ is
$$\mathfrak{sp}({2n})=\{X\in \mathfrak{gl}({2n})\mid XF+FX^t=0\}.$$ One can  check that it is a Lie subalgebra indeed.

 As in the case of $\mathfrak{o}(2n)$, we use  the set of indices   $\{-n,-n+1,\dots,-1,1,\dots, n-1,n\}$. The subalgebra 
$$\mathfrak{h}_{\mathfrak{sp}(2n)}=\operatorname{span}\{e_{ii}-e_{-i-i}\mid  1\leq i\leq n\},$$ 
of all diagonal matrices in $\mathfrak{sp}(2n)$ is a splitting Cartan sublagebra\break of $\mathfrak{sp}({2n})$.


We denote by $\varepsilon_{i}$ the  same weights as for $\mathfrak{o}(2n)$.
Then the root system $\Delta_{\mathfrak{sp}(2n)}$ of $\mathfrak{sp}(2n)$ with respect to $\mathfrak{h}_{\mathfrak{sp}(2n)}$ is
$$\Delta_{\mathfrak{sp}(2n)}=\{\pm(\varepsilon_{i}\pm \varepsilon_{j}), \pm2\varepsilon_{i}\mid i,j\in\mathbb{Z}_{>0}, 1\leq i\neq j\leq n\}.$$

Let
$$\Delta_{\mathfrak{sp}(2n)}^+=\{\varepsilon_{i}\pm\varepsilon_{j},2\varepsilon_{i}\mid i,j\in\mathbb{Z}_{>0}, 1\leq i<j\leq n\}$$
be the set of positive roots with the set of simple roots
\begin{equation}\label{sigmasp}
\{\varepsilon_{1}-\varepsilon_{2},\varepsilon_{2}-\varepsilon_{3},\dots,\varepsilon_{n-1}-\varepsilon_{n},2\varepsilon_{n}\}.
\end{equation}

Then the splitting Borel subalgebra $\mathfrak{b}_{\mathfrak{sp}(2n)}$ of $\mathfrak{sp}(2n)$ with positive roots $\Delta_{\mathfrak{sp}(2n)}^+$ is the subalgebra of all upper-triangular matrices in $\mathfrak{sp}(2n)$.

As in the case of $\mathfrak{o}(2n)$, we can rewrite any weight $\lambda=\sum^n_{i=1}\lambda_i\varepsilon_i$ as
$$\lambda=\sum^n_{i=1}\lambda_i\bar\varepsilon_i +\sum^{-n}_{j=-1}\lambda_{j}\bar\varepsilon_j$$
where $\lambda_j=-\lambda_{-j}$ for $j<0$.  The Weyl group $W_{\mathfrak{sp}(2n)}$ of $\mathfrak{sp}(2n)$ is isomorphic to the subgroup of $S_{2n}$ consisting of permutations ${w\in S_{2n}}$ such that \break $w(-i)=-w(i)$, $1\leq i\leq n$. 
As for $\mathfrak{o}(2n)$, we will identify $W_{\mathfrak{sp}(2n)}$ with this subgroup and will  use usual two-line notation.


In the sequel, we will refer to the Cartan and Borel subalgebras of    $\mathfrak{sl}(n)$,  $\mathfrak{o}(2n)$ or  $\mathfrak{sp}(2n)$ introduced above, as fixed Cartan and Borel sibalgebras.

\end{enumerate}

\subsection{Some countable-dimensional Lie algebras}\label{sub:infdimLA}

Here we review some important examples of infinite-dimensional  Lie algebras. 
\begin{enumerate}
	\item The Lie algebra $\mathfrak{sl}(\infty)=\mathfrak{sl}(\infty,\mathbb{C})$ is a countable-dimensional Lie algebra.
Consider the embeddings
$$\mathfrak{sl}(i) \to \mathfrak{sl}({i+1}),$$
	\[\textbf{X} \longmapsto \left(\begin{array}{cc}
	\textbf{X} & 0_{i\times1} \\
	0_{1\times i} & 0 
	\end{array}\right). \]

 We  set
$$\mathfrak{sl}(\infty) := \varinjlim\mathfrak{sl}(i).$$

Next, we choose a splitting Cartan subalgebra $\mathfrak{h}_{\mathfrak{sl}(\infty)}$ as the direct limit
$$\mathfrak{h}_{\mathfrak{sl}(\infty)} := \varinjlim\mathfrak{h}_i,$$
where $\mathfrak{h}_i$ is the diagonal splitting Cartan subalgebra of $\mathfrak{sl}(i)$. 
	The Lie algebra $\mathfrak{h}_{\mathfrak{sl}(\infty)}$ is obviously a toral subalgebra. It is also maximal toral,
as any larger subalgebra contains an elementary nondiagonal matrix, and the latter is a nilpotent element of $\mathfrak{sl}(\infty)$.

The root decomposition of the Lie algebra $\mathfrak{sl}(\infty)$ is 
	\begin{equation*}
	\label{gl_decomp}
	\mathfrak{sl}(\infty)= \mathfrak{h}_{\mathfrak{sl}(\infty)}\oplus\bigoplus_{i,j \in \mathbb{Z}_{>0}}\operatorname{span}\{e_{ij}\}.
	\end{equation*}


The root system $\Delta$ of  $\mathfrak{sl}(\infty)$ with respect to $\mathfrak{h}_{\mathfrak{sl}(\infty)}$ is
$$\Delta=\{\varepsilon_{i}-\varepsilon_{j}\mid\:i\neq j, \:i,j\in\mathbb{Z}_{>0}\},$$
where
$$\varepsilon_i(e_{jj})=\delta^i_j.$$
Let
$$\Delta^+=\{\varepsilon_{i}-\varepsilon_{j}\mid\: i<j, \:i,j\in\mathbb{Z}_{>0}\}$$  be the set of positive roots with the simple roots
$$\{\varepsilon_{1}-\varepsilon_{2},\varepsilon_{2}-\varepsilon_{3},\varepsilon_{3}-\varepsilon_{4}\dots\}.$$

Then the splitting Borel subalgebra $\mathfrak{b}_{\mathfrak{sl}(\infty)}\supset \mathfrak{h}_{\mathfrak{sl}(\infty)}$  with positive roots $\Delta^+$  is the direct limit of $\varinjlim\mathfrak{b}_{\mathfrak{sl}(n)}$.
 
\item The Lie algebra $\mathfrak{o}(\infty)=\mathfrak{o}(\infty,\mathbb{C})$ is a countable-dimensional Lie algebra.
Consider the embeddings
$$\phi_{2i}\colon\mathfrak{o}(2i) \to \mathfrak{o}({2i+2}),$$
	\[\textbf{X} \longmapsto \left(\begin{array}{ccc}
	0&0_{1\times i}&0\\
	0_{i\times1}&\textbf{X} & 0_{i\times1} \\
	0&0_{1\times i} & 0 
	\end{array}\right). \]

We set
$$\mathfrak{o}(\infty) := \varinjlim\mathfrak{o}(2i).$$

We use  $\mathbb{Z}\backslash\{0\}$ as the set of indices for matrices of $\mathfrak{o}(\infty)$.
We can choose a splitting Cartan subalgebra $\mathfrak{h}_{\mathfrak{o}(\infty)}$ as the direct limit
$$\mathfrak{h}_{\mathfrak{o}(\infty)} := \varinjlim\mathfrak{h}_{\mathfrak{o}(2i)},$$
where
$$\mathfrak{h}_{\mathfrak{o}(\infty)}=\operatorname{span}\{v_i=e_{ii}-e_{-i\; -i}\mid\: 1\leq i \}.$$ 


 By
 $$\{\varepsilon_{1},\varepsilon_{2},\varepsilon_{3},\dots\}$$ we denote set dual to the basis $\{v_i\}$ of $\mathfrak{h}_{\mathfrak{o}(\infty)}$, i.e.
$$\varepsilon_i(v_j)=\delta^i_j$$
for $1\leq i,j$ and $i,j \in \mathbb{Z}_{>0}$.

The root system $\Delta$ of  $\mathfrak{o}(\infty)$ with respect to $\mathfrak{h}_{\mathfrak{o}}$ is
$$\Delta=\{\pm(\varepsilon_{i}\pm\varepsilon_{j})\mid\:i\neq j, \:i,j\in\mathbb{Z}_{>0}\}.$$

Let
$$\Delta^+=\{-\varepsilon_{i}\pm\varepsilon_{j}\mid\: i>j, \:i,j\in\mathbb{Z}_{>0}\}$$ be the set of positive roots; the corresponding simple roots are
$$\{-\varepsilon_{1}-\varepsilon_{2},\varepsilon_{1}-\varepsilon_{2},\varepsilon_{2}-\varepsilon_{3},\varepsilon_{3}-\varepsilon_{4}\dots\}.$$

\item \label{sp} The Lie algebra $\mathfrak{sp}(\infty)=\mathfrak{sp}(\infty,\mathbb{C})$ is a countable-dimensional Lie algebra.
We consider the embedding
$$\psi_{2i}\colon\mathfrak{sp}(2i) \to \mathfrak{sp}({2i+2}),$$
	\[\textbf{X} \longmapsto \left(\begin{array}{ccc}
	0&0_{1\times i}&0\\
	0_{i\times1}&\textbf{X} & 0_{i\times1} \\
	0&0_{1\times i} & 0 
	\end{array}\right), \]
and put
$$\mathfrak{sp}(\infty) := \varinjlim\mathfrak{sp}(2i).$$

We use the set $\mathbb{Z}\backslash\{0\}$ as a set of indexes for matrices of $\mathfrak{sp}(\infty)$. One can choose  a splitting Cartan subalgebra $\mathfrak{h}_{\mathfrak{sp}(\infty)}$ as the direct limit
$$\mathfrak{h}_{\mathfrak{sp}(\infty)} := \varinjlim\mathfrak{h}_{\mathfrak{sp}(2i)}.$$
	By construction,
\begin{equation}\label{formula:v_i}
\mathfrak{h}_{\mathfrak{sp}(\infty)}=\operatorname{span}\{v_i=e_{ii}-e_{-i\; -i}\mid\: 1\leq i \}.
\end{equation}

The vectors $\varepsilon_{i}$ are defined in the same way as for $\mathfrak{o}(\infty)$.

The root system $\Delta$ of  $\mathfrak{sp}(\infty)$ with respect to $\mathfrak{h}_{\mathfrak{sp}(\infty)}$ is
$$\Delta=\{\pm(\varepsilon_{i}\pm\varepsilon_{j}), \pm2\varepsilon_{i}\mid\:i\neq j, \:i,j\in\mathbb{Z}_{>0}\}.$$

We set 
$$\Delta^+=\{-\varepsilon_{i}\pm\varepsilon_{j}, -2\varepsilon_{i}\mid\: i>j, \:i,j\in\mathbb{Z}_{>0}\}$$ as the  set of positive roots; the corresponding simple roots are
$$\{-2\varepsilon_{1},\varepsilon_{1}-\varepsilon_{2},\varepsilon_{2}-\varepsilon_{3},\varepsilon_{3}-\varepsilon_{4}\dots\}.$$

\end{enumerate}

\medskip
	\subsection{Highest weight modules} \label{sub:pre}

In this subsection we introduce the notion of highest weight module, and give some related definitions.

\begin{defn} A Lie algebra $\mathfrak{g}$ is \emph{simple} if every ideal $I\subset \mathfrak{g}$ is equal to zero or to $\mathfrak{g}$.
\end{defn}


For the purposes of this paper, we call a Lie algebra semisimple if it is isomorphic to a direct of a simple Lie algebras. 
	Let $\mathfrak{g}$ be a semisimple Lie algebra, $\mathfrak{h}\subset\mathfrak{g}$ be a splitting Cartan subalgebra,  $\mathfrak{b}\subset\mathfrak{g}$ be a splitting Borel subalgebra with $\mathfrak{b}\supset\mathfrak{h}$, and let
	$\mathfrak{n}=\bigoplus_{\alpha \in \Delta^+}\mathfrak{g}^\alpha$. Fix a  $\mathfrak{g}$-module  $M$.
	For each $\lambda \in \mathfrak{h}^*$, define $M^\lambda$ to be the subspace
	$$\{v \in M \mid h\cdot{v}=\lambda(h)v: \: \: \forall h \in \mathfrak{h}\} \subset M.$$
	\vspace{\baselineskip}
	\begin{defn}
		If 

 $$M= \bigoplus_{\lambda \in \mathfrak{h}^*}M^\lambda $$ then $M$ is  a \textup
		{weight module} over $\mathfrak{g}$. If $M^\lambda \neq 0$, then $\lambda\neq 0$ is said to
		be a \textup {weight} of $M$, $M^\lambda$ is called a \textup {weight subspace} of
		$M$, and the elements of $M^\lambda$ are called \textup {weight vectors of weight} $\lambda$. 
		
	\end{defn}

	\begin{defn}
 The set $\operatorname{supp}M$ of weights $\lambda$ for which $\operatorname{dim}M_{\lambda}>0$ is  called  \textup{the support}  of the weight module $M$.
	\end{defn}
	One can prove that submodules, quotients and direct sums of weight\break $\mathfrak{g}$-modules are weight
	modules as well.
		
	\begin{defn}
		A  $\mathfrak{g}$-module $M$ is said to be a \textup {cyclic} over
		$\mathfrak{g}$ if $M$ is generated as an $U(\mathfrak{g})$-module by a single
		nonzero vector. 
	\end{defn}

	\begin{defn}
		A  $\mathfrak{g}$-module $M$  is called a \textup
		{highest weight module} with respect to a splitting Borel subalgebra $\mathfrak{b}$ if it is generated by a vector $v\neq0$ satisfying
		$$\mathfrak{n}\cdot v=0,$$
		and there exists a weight $\lambda$ (which we will call \textup {the highest weight}
		of $M$) such that
		$$h\cdot v=\lambda(h)v$$
		for each $h \in \mathfrak{g}$. The vector $v$ is called a \textup
		{highest weight vector} of $M$.
	\end{defn}
	Each highest weight  $\mathfrak{g}$-module is a weight $\mathfrak{g}$-module since $\mathfrak{g}$ is an $\operatorname{ad}_{\mathfrak{h}}$-weight module.
	\begin{defn}
		
		For a fixed finite-dimensional semisimple Lie algebra $\mathfrak{g}$, a splitting Borel sublagebra $\mathfrak{b}\subset\mathfrak{g}$, a splitting Cartan subalgebra $\mathfrak{h}\subset\mathfrak{b}$,   and for every
		$\lambda \in \mathfrak{h}^*$, we define the highest weight module
		$M(\lambda)=M(\lambda;\mathfrak{g},\mathfrak{b},\mathfrak{h})$  by setting
		
		$$M(\lambda;\mathfrak{g},\mathfrak{b},\mathfrak{h}) := U(\mathfrak{g})/I , $$
		where $I$ is the left ideal in $U(\mathfrak{g})$ generated by $ \mathfrak{n}$ and by  $h-\lambda(h)1_{U(\mathfrak{g})}$ for all $h\in\mathfrak{h}$.
		The modules  $M(\lambda;\mathfrak{g},\mathfrak{b},\mathfrak{h})$ are known as \textup {Verma modules}.
	\end{defn}
	
	Each Verma module $M(\lambda)$ has a unique maximal proper
	$U(\mathfrak{g})$-submodule $N$, which is the sum of all proper submodules of
	$M(\lambda)$, \cite{DP}. Accordingly, 
	
	$$L(\lambda):=M(\lambda)/N$$
	is the unique simple quotient of $M(\lambda)$.
	
	We denote by $v_{\lambda}$ the image of $1_{U(\mathfrak{g})} $ under the canonical
	projection \break$U(\mathfrak{g})\twoheadrightarrow U(\mathfrak{g})/I$. Clearly,
	$v_{\lambda}$ is a highest weight vector of $M(\lambda)$ of weight $\lambda$. Moreover, for any highest weight module $M$ over $\mathfrak{g}$ with highest
	weight $\lambda$, there is a unique surjective homomorphism $\phi\colon M(\lambda)
	\twoheadrightarrow M$. Hence,  up to isomorpsim, $L(\lambda)$ is the unique simple highest weight
	module over $\mathfrak{g}$ with highest weight $\lambda$.
	
\medskip
\medskip

\subsection{BGG category $\mathcal{O}$}
Here we recall some basics concerning the BGG category $\mathcal{O}$, which was introduced in the early 1970s by Joseph Bernstein, Israel Gelfand, and Sergei Gelfand. 

\begin{defn} The \emph{BGG category} $\mathcal{O}$ is defined to be  the full subcategory of $U(\mathfrak{g})$-modules whose objects  are the modules satisfying the following three conditions:
\begin{enumerate}
\item $M$ is a finitely generated $U(\mathfrak{g})$-module.
\item $M$ is $\mathfrak{h}$-semisimple, that is, $M$ is a weight module.
\item $M$ is a locally $\mathfrak{n}$-finite: for each $v\in M$, the subspace $U(\mathfrak{n})\cdot v$ of $M$ is finite dimensional.
\end{enumerate}
\end{defn}

We will call a weight $\lambda$ \emph{dominant} if $\frac{2(\alpha,\lambda)}{(\alpha,\alpha)}$ is not a negative integer for any $\alpha \in \Delta^+$. A weight $\lambda$  is \emph{regular} if $\frac{2(\alpha,\lambda)}{(\alpha,\alpha)}\neq0$ for any $\alpha \in \Delta^+$.

\begin{defn} The \emph{dot action} of  $W$ on  $\mathfrak{h}^*$ is defined by letting $$w\cdot\lambda=w(\lambda+\rho_{\mathfrak{g}})-\rho_{\mathfrak{g}}.$$ \end{defn}

\begin{defn}
\emph{ The reflection group corresponding to a linear function $\lambda$}  is the subgroup $W_{[\lambda]}$ of  $W$ which consists of the elements

$$w\in W\text{ such that }w\cdot\lambda-\lambda\in\Lambda_{\Delta},$$
where $\Lambda_{\Delta}$ is the root lattice of $\mathfrak{g}$.
\end{defn}

By definition, two weights $\gamma$ and $\kappa$ are \emph{linked by} $W_{[\lambda]}$ whenever $\gamma=w\cdot \kappa$ for $w\in W_{[\lambda]}$.

\emph{The blocks} of  the category $\mathcal{O}$ are precisely the subcategories 	consisting of modules whose all composition factors  have highest weights     linked by $W_{[\lambda]}$     to a weight $\lambda$ such that $-\lambda$  is dominant. Thus the blocks are in natural bijection with the dominant weights.

	\subsection{Duflo's Theorem}
	
In this subsection we recall the important Duflo Theorem, which allows us to reduce the set of primitive ideals to the set of annihilators of simple highest weight modules.

	Let $A$ be an associative algebra with identity, and $I$ be an  ideal of $A$.

	\begin{defn}
		
		The ideal $I\subset A$ is \textup{primitive} if $I$  is the annihilator of a simple left
		$A$-module.
	\end{defn}
	
	This following statement is well known  as Duflo's Theorem.
	\begin{theorem}
		\textup{\cite{D}}
		Let $\mathfrak{g}$ be a finite-dimensional reductive Lie algebra, $I$ be a  primitive ideal of $ U(\mathfrak{g})$, and $\mathfrak{b}$ be  a Borel subalgebra  of $\mathfrak{g}$. Then there exists  an irreducible $\mathfrak{b}$-highest weight  $\mathfrak{g}$-module  whose annihilator is $I$.
		
	\end{theorem}

	\subsection{Associated variety}
	
	Let $\mathfrak{g}$ be a finite-dimensional  Lie algebra.  The following theorem is known as the   Poincaré--Birkhoff--Witt Theorem, and plays an important role in this thesis.
	
	\begin{theorem}
		Let $\phi\colon\mathfrak{g}\rightarrow U(\mathfrak{g})$ be the canonical map.
		Denote by\break $\{g_1,g_2,g_3,\ldots,g_n\}$  a basis of $\mathfrak{g}$. Then  the monomials
		${g_1^{v_1}g_2^{v_2}g_3^{v_3}\ldots g_n^{v_n}}$, where $v_1,v_2,v_3, \ldots, v_n
		\in \mathbb{Z}_{\geq0}$, constitute a basis of $U(\mathfrak{g}) $.
		
	\end{theorem}
	
	Define $U^i $ to be the vector subspace of $U(\mathfrak{g})$ spanned by all monomials
	${g_1^{v_1}g_2^{v_2}g_3^{v_3}\ldots g_n^{v_n}}$ with $\sum v_j \leq i $. The
	chain of subspaces $\{U^i \}_{i\in \mathbb{Z}_{\geq0}}$ is a \textit{filtration} of $U(\mathfrak{g})$.

	\begin{defn}
	Put $U^{-1}=\{0\}$.	Define \textup{ the associated graded algebra of $U(\mathfrak{g})$}:
		$$ gr U(\mathfrak{g}) := \bigoplus_{d \in \mathbb{Z}_{\geq0}}
		(U^{d}/U^{d-1})$$
		
	\end{defn}
	
	As a consequence of the Poincaré--Birkhoff--Witt Theorem, we conclude that this algebra is isomorphic to the symmetric algebra $S^{\bullet}(\mathfrak{g})$.

	\begin{defn}
		In the same way we define the \textup{associated graded ideal} of an ideal $I\subset U(\mathfrak{g})$:
		$$ gr I := \bigoplus_{d \in \mathbb{Z}_{\geq0} } (U^{d}\cap I) /(U^{d-1} \cap I) \subset gr U(\mathfrak{g})\simeq S^{\bullet}(\mathfrak{g}).$$
	\end{defn}
		
	\begin{defn} \label{sub:var}
		\textup We denote by $Var (I)$ the algebraic variety corresponding to $grI$; by definition this is the set
		of common  zeros of $gr I$ in $\mathfrak{g}^*$ (Here we identify $S^{\bullet}(\mathfrak{g})$ with the algebra 
of polynomial functions of $\mathfrak{g}^*$). By identifying $\mathfrak{g}$ and
		$\mathfrak{g}^*$ via the Killing form, we can  assume that $Var (I) \subset
		\mathfrak{g} $.
		
	\end{defn}

	\subsection{Weakly bounded ideals}\label{section:LBM}

In this subsection we give  some definitions which are needed to state the results of this work.  

Here $\mathfrak{g}(\infty)$  is one of the Lie algebras $\mathfrak{sl}(\infty)$, $\mathfrak{o}(\infty)$ or $\mathfrak{sp}(\infty)$. 

	\begin{defn}
		A $\mathfrak{g}(\infty)$-module $M$ is \textup{integrable} if, for any finitely
		generated 	subalgebra $U' \subset U(\mathfrak{g}(\infty))$ and any $m \in M$, we
		have $\operatorname{dim}(U'\cdot m) < \infty$.
	\end{defn}
	
	\begin{defn}
		A two-sided ideal $I \subset U(\mathfrak{g}(\infty))$ is \textup{integrable}, if
		$I$ is the annihilator of an integrable $\mathfrak{g}(\infty)$-module.
	\end{defn}

	\begin{defn}
	An ideal $I \subset U(\mathfrak{g}(\infty))$ is \textup{locally integrable} if,
		for any finitely generated subalgebra $A' \subset U(\mathfrak{g}(\infty))$, the
		ideal $I \cap A'$ is an integrable ideal of $A'$.
	\end{defn}
		One can check that an ideal $I \subset U(\mathfrak{g}(\infty))$ is locally
	integrable if and only if, for every $n \in \mathbb{Z}_{\geq0}$, $I\cap
	U(\mathfrak{g}(2n))$ is an intersection of ideals  of finite codimension in 
	$ U(\mathfrak{g}(2n))$.
	\vspace{\baselineskip}
\begin{defn}\label{def:deg}
 Let $\mathfrak{g}$ be a (possibly infinite-dimensional) Lie algebra. For every weight $\mathfrak{g}$-module $M$ we define \textup{the degree of module} $$\operatorname{deg}(M):=\operatorname{sup}_{\lambda\in \operatorname{supp}M}(\operatorname{dim}M_{\lambda}).$$
\end{defn}
	
\begin{defn}
		Let $\mathfrak{g}$ be a finite-dimensional  Lie algebra. Then a weight $\mathfrak{g}$-module $M$ is called \textup{bounded} if  $\operatorname{deg}(M)<\infty$.
	\end{defn}

\begin{defn}
Let $\mathfrak{g}$  be  a finite-dimensional Lie algebra with a splitting Cartan subalgebra $\mathfrak{h}$. An ideal $I\subset U(\mathfrak{g})$ is an $\mathfrak{h}$-\emph{bounded ideal} if $I=\operatorname{Ann}_{U(\mathfrak{g})}(M)$ where $M$ is a bounded $\mathfrak{h}$-weight $\mathfrak{g}$-module. 

\end{defn}

Further, we will simply say  $bounded$ instead of $\mathfrak{h}$-bounded since the subalgebra $\mathfrak{h}$ will be fixed.

 Let $I$ be a bounded ideal such that $I=\operatorname{Ann}_{U(\mathfrak{g})}(M)$ for some simple weight $\mathfrak{g}$-module $M$. Then $M$ is a bounded module (see \cite{PS}).

Note that there are no infinite-dimensional bounded $\mathfrak{o}(2n)$-modules \cite{F}. 
There is a classification of bounded infinite-dimensional highest weight simple $\mathfrak{g}$-modules, where $\mathfrak{g}$ is one of Lie algebras $\mathfrak{sl}(n)$, $\mathfrak{o}(2n)$, $\mathfrak{o}(2n+1)$ or $\mathfrak{sp}(2n)$ (actually, such modules exist only for $\mathfrak{sl}(n)$ and $\mathfrak{sp}(2n)$). In particular, for $\mathfrak{sp}(2n)$ we have

\begin{lemma} \textup{\cite{M}} \label{prop:sp.bounded}
Let $L(\lambda)$ be a simple infinite-dimensional  highest weight $\mathfrak{sp}(2n)$-module with highest weight $\lambda$. It is bounded if and only if 
\begin{enumerate}
\item $\lambda(v_i-v_{i+1})\in \mathbb{Z}_{>0}$ for any $0<i<n$,
\item $\lambda(v_n)\in 1/2+\mathbb{Z}$,
\item $\lambda(v_{n-1}+v_{n})\in\mathbb{Z}_{\geq-2}$,
\end{enumerate}
where $v_i$ is defined in Subsection \ref{2.3} in formula (\ref{formula:v_i}).
\end{lemma}

The following bounded $\mathfrak{sp}(2n)$-modules will play an important role in this  thesis.
\begin{exam}\label{SW}
Consider the ring  $\mathbb{C}[x_1,x_2,\dots,x_n]$ of polynomials in $n$ variables. The Lie algebra $\mathfrak{g}(2n)=\mathfrak{sp}(2n)$ can be realized as follows: $$\mathfrak{g}(2n)=\bigoplus_{1\leq i,j\leq n}\operatorname{span}\{\frac{\partial^2}{\partial x_i\partial x_j}\} \oplus \bigoplus_{1\leq k,l\leq n}\operatorname{span}\{x_k \frac{\partial}{\partial x_l}+\frac{\delta^k_l}{2}\}\oplus\bigoplus_{1 \leq m,n\leq n} \operatorname{span}\{x_m x_n\},$$
where $\delta_l^m$ is  the Kronecker symbol.   The space   $\bigoplus_{i}\operatorname{span}\{x_i \frac{\partial}{\partial x_i} +\frac{1}{2}\}$ is a splitting Cartan sublagebra $\mathfrak{h}\subset\mathfrak{sp}(2n)$, with simple coroots $h_i=-x_i\frac{\partial}{\partial x_i}+x_{i+1}\frac{\partial}{\partial x_{i+1}}$ for $i\leq i \leq n-1$ and $h_n=-x_n\frac{\partial}{\partial x_n}-\frac{1}{2}$. As an $\mathfrak{sp}(2n)$-module,  $\mathbb{C}[x_1,x_2,\dots,x_n]$ equals $SW^+(2n) \oplus SW^-(2n)$, where $SW^+(2n) $ is the subspace of homogeneous polynomials of even degree, and $SW^-(2n)$ is the subspace of homogeneous polynomials of odd degree. These two subspaces are simple bounded highest weight  modules with  respective highest weights $-\frac{1}{2}\sum^n_{i=1}\varepsilon_i$ and $-\frac{1}{2}\sum^{n-1}_{i=1}\varepsilon_i-\frac{3}{2}\varepsilon_{n}$. They are known as \emph{Shale--Weil} (or \emph{oscillator) representations}.

\end{exam}
\begin{defn}
		An ideal $I \subset U(\mathfrak{g}(\infty))$ is \textup{weakly bounded} if
		 $I\cap U_n$ is an intersection of  annihilators of bounded weight  modules of $U_n$ 
for every $n\geq 2$, where $U_n=U(\mathfrak{sl}(n))$ for $\mathfrak{g}(\infty)=\mathfrak{sl}(\infty)$, $U_n=U(\mathfrak{o}(2n))$ \break for $\mathfrak{g}(\infty)=\mathfrak{o}(\infty)$, and $U_n=U(\mathfrak{sp}(2n))$ for $\mathfrak{g}(\infty)=\mathfrak{sp}(\infty)$. 
	\end{defn}
	\vspace{\baselineskip}

The next theorem is one of the main results of the  present work.

\begin{theorem}\label{fir}
Every primitive ideal of $U(\mathfrak{o}(\infty))$ or  $U(\mathfrak{sp}(\infty))$ is weakly bounded. Moreover, each primitive ideal of $U(\mathfrak{o}(\infty))$ is  integrable.
\end{theorem}
 Note that  this theorem is an analogue of the  result for $\mathfrak{sl}(\infty)$ proved by I. Penkov and A. Petukhov   in \cite{PP4}.

	\subsection{Integrable ideals and coherent local systems}

\label{2.9}
	As before,   $\mathfrak{g}(\infty)=\mathfrak{sl}(\infty)$, $\mathfrak{o}(\infty)$ or $\mathfrak{sp}(\infty)$. Also, in this subsection and the next subsection, $\mathfrak{g}(n)$ denotes  one of the Lie algebras $\mathfrak{sl}(n)$, $\mathfrak{o}(2n)$ or $\mathfrak{sp}(2n)$. Since we  express $\mathfrak{o}(\infty)$ as $\varinjlim\mathfrak{o}(2n)$, we do not need the \break $\mathfrak{o}(2n+1)$-series of Lie algebras.
	Let $Irr_n$ denote the set of isomorphism classes of simple finite-dimensional
	$\mathfrak{g}(n)$-modules.
	\begin{defn}
		A \textup{coherent local system of modules} (further c.l.s.)  for
		$\mathfrak{g}(\infty)$ is a collection of sets
		
		$$\{Q_n\}_{n \in \mathbb{Z}_{\geq 1}} \subset \Pi_{n \in \mathbb{Z}_{\geq 1} }
		Irr_n $$
		such that $Q_m = \langle Q_n \rangle_m$ for any $n>m$, where	$\langle Q_n
		\rangle_m$ denotes the set of isomorphism classes of all simple $\mathfrak{g}({m})$-constituents of the
		$\mathfrak{g}({n})$-modules from $Q_n$.
		
	\end{defn} 
	\vspace{\baselineskip}
	\begin{defn}
		A c.l.s. $Q$ is \textup{irreducible} if $Q \neq Q' \cup Q''$ with
		$Q'\not\subset Q''$ and $Q' \not\supset Q''$, where $Q'$ and $Q''$ are nonempty coherent local systems of modules for $\mathfrak{gl}(\infty)$.
		
	\end{defn}

	Each c.l.s. $Q$ can be represented uniquely as a finite union $\cup_i Q(i)$ \cite{Zh1} of some maximal (by
	inclusion within $Q$) irreducible c.l.s. $Q(i)$; we call $Q(i)$ \emph{the irreducible components} of $Q$.

	Each integrable $\mathfrak{g}(\infty)$-module $M$ determines a c.l.s.
	$Q:=\{Q_n\}_{n \in \mathbb{Z}_{\geq 1} }$, where
	
	$$ Q_n:=\{z \in  Irr_n\mid \operatorname{Hom}_{\mathfrak{g}(n)} (z,M) \neq \{0\}\}$$
	We denote this c.l.s. by $Q(M)$.
	\begin{defn}
		We say that a c.l.s. Q is  of \textup{finite type} if the set $Q_n $ is finite for all
		$n \geq 1$.
	\end{defn}

	\begin{defn}
		An integrable $\mathfrak{g}(\infty)$-module $M$  is called \textup{locally simple} if
		$M=\varinjlim M_n$ for a chain $$M_3 \subset M_n \subset M_{n+1} \subset
		\ldots$$ of simple finite-dimensional $\mathfrak{g}(n)$-submodules $M_n$ of $M$.
	\end{defn}
	
	For every c.l.s. $Q=\{Q_n\}_{n \in \mathbb{Z}_{\geq 1}}$ we can define the following ideal  
	$$I(Q):=\cup_m(\cap_{z\in Q_m}
	\operatorname{Ann}_{U(\mathfrak{g}(m))}z)\subset U(\mathfrak{g(\infty)}).$$
	  We say that $I(Q)$ is the \emph{annihilator} of $Q$.
	
	\begin{prop}\textup{\cite[Lemma 1.1.2]{Zh1}}\label{prop:annih}
		If $Q$ is  an irreducible  c.l.s., then $I(Q)$ is the annihilator of some
		locally simple integrable $\mathfrak{g(\infty)}$-module. In particular, the ideal $I(Q)$ is
		primitive.
	\end{prop}
	
	\begin{corollary}
		The ideal $I(Q)$ is integrable for each c.l.s. $Q$.
	\end{corollary}
	\begin{proof}
		Let $Q=\cup_i Q(i)$, where $Q(i)$ are the irreducible  components  of $Q$. By
		Proposition \ref{prop:annih},  the ideal $I(Q(i))$ is the annihilator of a
		simple integrable $\mathfrak{g(\infty)}$-module $M(i)$. Therefore $I(Q)$ is the
		annihilator of the integrable \break $\mathfrak{g(\infty)}$-module $\oplus_i M_i$.
	\end{proof}

	\vspace{\baselineskip}

	Every isomorphism class $z \in Irr_n$ of simple $\mathfrak{g}(n)$-modules 
	corresponds to an integral dominant weight $\lambda$ of $\mathfrak{g}(n)$. Let $z_1, z_2$ be
	isomorphism classes of simple $\mathfrak{g}(n)$-modules with respective highest
	weights $\lambda_1, \lambda_2$. We denote by $z_1z_2$ the isomorphism class of a
	simple module with highest weight $\lambda_1+\lambda_2$. If $S_1, S_2 \subset
	Irr_n$ we set 
	
	$$S_1S_2:=\{z\in Irr_n \mid z=z_1z_2 \:\: \textup{for  some}\:\: z_1 \in S_1 \:\:
	\textup{and} \:\: z_2 \in  S_2\}. $$
	Let $Q'$ and $Q''$ be c.l.s. We denote  by $Q'Q''$ the smallest c.l.s. such
	that $(Q')_n(Q'')_n \subset (Q'Q'')_n$. By definition, $Q'Q''$ is the \emph{product} of
	$Q'$ and $Q''$. The definition implies that the  operation of product is associative and commutative.
	
	\subsection{Zhilinskii's classification of c.l.s.}
	
	We set  $V = \varinjlim V(n)$, where $V(n)$ is the natural
	$\mathfrak{g}(n)$-module, and we set $(V)_* = \varinjlim V(n)^*$ where
	$V(n)^*$ is the conatural $\mathfrak{g}(n)$-module. We denote by $S^{\bullet}(K)$ and $\Lambda^{\bullet}(K)$ the symmetric algebra and  the exterior algebra of a module $K$ respectively. Also $S^p(K)$ and $\Lambda^p(K)$ denote respectively the  $p$th symmetric power and the $p$th exterior power of a module $K$.
A simple highest weight $\mathfrak{o}(2n)$-module  with highest weight $(\frac{1}{2}\sum^{n-1}_1\varepsilon_i)\pm \frac{1}{2}\varepsilon_n$ is called a \emph{spinor} module.

If $K$ is a $\mathfrak{g}(\infty)$-module, we define the c.l.s. $Q(K)$ for which $Q(K)_n$ is precisely the set of isomorphism classes of all simple constituents of  $K$ considered as a $\mathfrak{g}(n)$-module.

	For simplicity we will use the following notations:
	
	$$E:=Q(\Lambda^\bullet V),\: L_p := Q(\Lambda^p V),\: L^\infty_p :=
	Q(S^\bullet(V \otimes \mathbb{C}^p)),$$  $$R_p := Q(\Lambda^p (V)_*), 
 \: R^\infty_p := Q(S^\bullet((V)_*)
	\otimes \mathbb{C}^p), R := \textup{\{spinor  modules\}},$$
$$ \: E^\infty := \textup{\{all  irreducible  modules, which highest weight consists integral entries \}}$$
	
	
	
	
	
\noindent	where $p,q \in \mathbb{Z}_{\geq 1} $. Moreover,  the following table  defines the  \emph{basic c.l.s.} for the  Lie algebras $\mathfrak{sl}(\infty)$, $\mathfrak{o}(\infty)$ and $\mathfrak{sp}(\infty)$.

\begin{table}[h]
\begin{center}
\begin{tabular}{|c|c|}
\hline
Lie algebra & Basic c.l.s. \\
\hline
$\mathfrak{sl}(\infty)$ & $E$, $ L_p$ , $L^\infty_p$, $ R_p, R^\infty_p$, $ E^\infty$\\
\hline
$\mathfrak{o}(\infty)$ & $E$, $ L_p$ , $L^\infty_p$, $ E^\infty$, $R$\\
\hline
$\mathfrak{sp}(\infty)$ & $E$, $ L_p$ , $L^\infty_p$, $ E^\infty$\\
\hline
\end{tabular}
\end{center}
\end{table}

By definition the \emph{trivial c.l.s.} is the $c.l.s.$ $Q$ such that $Q_n=\{\mathbb{C}\}$, where $g\cdot \mathbb{C}=0$ for any $g\in \mathfrak{g}(n)$.
	
	\begin{prop}\textup{\cite{Zh1}} \label{2.6}
		Any irreducible c.l.s. can be uniquely expressed  as a product of basic c.l.s. as follows:
		
		$$(L^\infty_v L_{v+1}^{x_{v+1}} L_{v+2}^{x_{v+2}} \dots L_{v+r}^{x_{v+r}}) E^m
		(R^\infty_w R_{w+1}^{z_{w+1}} R_{w+2}^{z_{w+2}} \dots R_{w+t}^{z_{w+t}})  \: \:\:\:\:\:
		\:\:\:\:\:\:\:\:\:\:\:\:\:\:\:\:\: $$  {for} $\mathfrak{g}(\infty) =
		\mathfrak{sl}(\infty)$,
		$$(L^\infty_v L_{v+1}^{x_{v+1}} L_{v+2}^{x_{v+2}} \dots L_{v+r}^{x_{v+r}}) E^m \:
		\textup{or} \:(L^\infty_v L_{v+1}^{x_{v+1}} L_{v+2}^{x_{v+2}} \dots L_{v+r}^{x_{v+r}}) E^m
		R \:\:\:\:\: \:\:\:$$
{for} $\mathfrak{g}(\infty) = \mathfrak{o}(\infty),$
		$$(L^\infty_v L_{v+1}^{x_{v+1}} L_{v+2}^{x_{v+2}} \dots L_{v+r}^{x_{v+r}}) E^m \:
		\:\:\:\:\:
		\:\:\:\:\:\:\:\:\:\:\:\:\:\:\:\:\:\:\:\:\:\:\:\:\:\:\:\:\:\:\:\:\:\:\:\:\:\:\:\:\:\:\:\:\:\:\:\:\:\:\:\:\:\:\:\:\:\:\:\:\:$$ 
{for} $\mathfrak{g}(\infty) = \mathfrak{sp}(\infty)$,
		where
		$$m,r,v,w \in \mathbb{Z}_{\geq0},$$
		$$ x_i,z_j \in \mathbb{Z}_{\geq 0} \: \textup{for} \: v+1 \leq i \leq n \:
		\textup{and} \: w+1 \leq j \leq t.$$
		Here, for $v=0$, $L^\infty_v$ is assumed to be trivial c.l.s., and, similarly,
		$R^\infty_w$ is assumed to be trivial c.l.s. for $w=0$.
	\end{prop}

	\subsection{Tensor product of c.l.s.}\label{sub:tensor}
	Here we  reformulate the  expression of Proposition \ref{2.6} in terms of tensor products.
	\begin{defn}
		Let $S_1, S_2 \subset Irr_n$ then

		$$ S_1 \otimes S_2 := \{z \in Irr_n\mid \: \operatorname{Hom}_{\mathfrak{g}(n)}(z, z_1 \otimes z_2 ) \neq \{0\} \:
		\textup{for  some} \: z_1 \in S_1 \: \textup{and} \: z_2 \in S_2 \}.$$
		Given two c.l.s. $Q'$ and $Q''$, their \textup{tensor product }  is the c.l.s. defined by
		$(Q'\otimes Q'')_i = Q'_i \otimes Q''_i $.
	\end{defn}
	
	One can check that 
	$$ (L^\infty_v L_{v+1}^{x_{v+1}} L_{v+2}^{x_{v+2}} \dots L_{v+r}^{x_{v+r}}) E^m
	(R^\infty_w R_{w+1}^{z_{w+1}} R_{w+2}^{z_{w+2}} \dots R_{w+t}^{z_{w+t}}) = $$
	$$ (L^\infty_1)^{\otimes v} \otimes (R^\infty_1)^{\otimes w} \otimes
	((L_{1}^{x_{v+1}} L_{2}^{x_{v+2}} \dots L_{r}^{x_{v+r}}) E^m ( R_{1}^{z_{w+1}}
	R_{2}^{z_{w+2}} \dots R_{t}^{z_{w+t}}))$$
for $\mathfrak{g}(\infty) = \mathfrak{sl}(\infty)$, 
	
	$$(L^\infty_v L_{v+1}^{x_{v+1}} L_{v+2}^{x_v+2} \dots L_{v+r}^{x_{v+r}}) E^m =
	(L^\infty_1)^{\otimes v} \otimes (L_{1}^{x_{v+1}} L_{2}^{x_{v+2}} \dots
	L_{r}^{x_{v+r}}) E^m $$ 
	for $\mathfrak{g}(\infty) = \mathfrak{o}(\infty) ,\mathfrak{sp}(\infty)  $, and
	
	$$(L^\infty_v L_{v+1}^{x_{v+1}} L_{v+2}^{x_{v+2}} \dots L_{v+r}^{x_{v+r}}) E^m R =
	(L^\infty_1)^{\otimes v} \otimes (L_{1}^{x_{v+1}} L_{2}^{x_{v+2}} \dots
	L_{r}^{x_{v+r}}) E^m R$$
	for $\mathfrak{g}(\infty) = \mathfrak{o}(\infty)  $.
	
	We will call an irreducible c.l.s. $Q$ for $\mathfrak{sl}(\infty)$ a \textit{left
		irreducible c.l.s.} if 
	$$Q= (L^\infty_1)^{\otimes v}  \otimes ((L_{1}^{x_{v+1}} L_{2}^{x_{v+2}} \dots
	L_{r}^{x_{v+r}}) E^m ( R_{1}^{z_{w+1}} R_{2}^{z_{w+2}} \dots R_{t}^{z_{w+t}})).$$
		
	\begin{prop}\textup{\cite{PP1}}
		Let $\mathfrak{g}= \mathfrak{sl}(\infty), \mathfrak{o}(\infty)$,
		$\mathfrak{sp}(\infty)$. An integrable ideal of $U(\mathfrak{g}(\infty))$ is prime
		if and only if it is primitive.
	\end{prop}
	
	For any integrable ideal $I \subset  U(\mathfrak{g}(\infty))$, let

	$$ Q(I)_n:=\{z\in Irr_n \mid I \cap U(\mathfrak{g}(n)) \subset 
	\operatorname{Ann}_{U(\mathfrak{g}(n))}z \}.$$
The collection $\{Q(I)_n\}$ is a well-defined c.l.s., which we denote by $Q(I)$.

	For $\mathfrak{g}(\infty) = \mathfrak{sl}(\infty)$ we denote by $Q_l(I)$ the union of all irreducible components of $Q(I)$ which are left irreducible c.l.s.
	\vspace{\baselineskip}
	
	\begin{theorem}\textup{\cite{PP1}}
		
		\begin{enumerate}
			\item If $\mathfrak{g}(\infty)=\mathfrak{o}(\infty)$, $\mathfrak{sp}(\infty)$, then the
			maps 
			$$ I \longrightarrow Q(I),$$
			$$ Q \longrightarrow I(Q)$$
			
			are mutually inverse bijections (which reverse the inclusion relation) between the set  of  integrable
			ideals in $U(\mathfrak{g}(\infty))$  and the set of c.l.s. for $\mathfrak{g}(\infty)$.
			
			\item  In case $\mathfrak{g}(\infty)=\mathfrak{sl}(\infty)$ , then the maps
			$$ I \longrightarrow Q_l(I)$$
			$$ Q \longrightarrow I(Q)$$
				are mutually inverse bijection (which reverse the relation of inclusion) between the set of prime ideals in $U(\mathfrak{g}(\infty))$  and the set of left irreducible c.l.s. for $\mathfrak{g}(\infty)$.
			\item 
			
            Each integrable ideal  of $
			U(\mathfrak{sl}(\infty))$ has the form  $I(Q)$ for some left  c.l.s.~$Q$.
			
		\end{enumerate}
	\end{theorem}
	\subsection{Robinson--Schensted algorithm}\label{sub:RS}

A \emph{partition} $\lambda$ of an integer $n \in \mathbb{Z}_{\geq0}$ is a nonincreasing finite sequence $\lambda_1 \geq \lambda_2 \geq\dots$ of positive integers,  whose sum $|\lambda| =\Sigma \lambda_i$
 equals $n$. The terms $\lambda_i$ of this sequence are called \emph{the parts}
of the partition $\lambda$.  Let  $P_n$ be the (obviously finite) set
of all partitions of $n$, and  $P$ be the union of all $P_n$ for $n \in \mathbb{Z}_{\geq0}$.

To each $\lambda \in P_n$ one can associate a subset of $\mathbb{Z}_{>0} \times \mathbb{Z}_{>0}$, called  \emph{Young diagram} $Y(\lambda)$; it
is defined by $(i, j) \in Y(\lambda) \Longleftrightarrow j < \lambda_i$ (so that $\#Y (\lambda) = |\lambda|$). The elements of a Young diagram will be called  boxes, and we
may correspondingly depict the Young diagram as rows of boxes of respective lengths $\lambda_i$, aligned by their left ends, the row of length $\lambda_i$ lying higher then the row of length $\lambda_j$ for all $i<j$.

As an example we consider the partition
$\lambda = (7, 5, 3, 3, 1)$ in $P_{19}$, and draw its Young diagram

$$Y(\lambda)=
\begin{array}{ccccccc}
\hline\multicolumn{1}{|c|}{\: }&\multicolumn{1}{|c|}{\:}&\multicolumn{1}{|c|}{\:}&\multicolumn{1}{|c|}{\: }&\multicolumn{1}{|c|}{\:}&\multicolumn{1}{|c|}{\:}&\multicolumn{1}{|c|}{\:}\\
\hline \multicolumn{1}{|c|}{}&\multicolumn{1}{|c|}{}&\multicolumn{1}{|c|}{\:}&\multicolumn{1}{|c|}{\:}&\multicolumn{1}{|c|}{\:}& &\\
\hhline{-----~~}\multicolumn{1}{|c|}{}&\multicolumn{1}{|c|}{}&\multicolumn{1}{|c|}{\:}& & & &\\
\hhline{---~~~~}\multicolumn{1}{|c|}{}&\multicolumn{1}{|c|}{}&\multicolumn{1}{|c|}{\:}& & & &\\
\hhline{---~~~~}\multicolumn{1}{|c|}{}& & & & & &\\
\hhline{-~~~~~~}
\end{array}.$$

Clearly, a partition $\lambda \in P$ is  determined by $Y (\lambda)$. The principal reason for referring to the elements of a Young diagram $Y (\lambda)$ as boxes (rather than
as points), is that it allows one to represent maps $f\colon Y (\lambda) \to \mathbb{Z}$ by filling each box $s \in Y (\lambda)$ with
the number $f(s)$. We shall call such a filled Young diagram a \emph{standard Young tableau} (or simply a standard tableau) \emph{of
shape} $\lambda$ if it satisfies the following condition:  all numbers $f{(s)	}$ strictly  decrease along each row and weakly decrease along each column.

Let's describe the Robinson--Schensted (or Robinson--Schensted--Knuth) algorithm.  It starts from  an ordered set $d=\{d_i\}$  positive integer, where $1\leq i \leq n$, and produces as  output  two Young tableaux: the insertion tableau $Y$ and the recording tableau $Y'$. This algorithm is based on a procedure of  inserting a new positive integer into a Young tableau,
 displacing certain entries, and creating  a tableau with one more box  than the
original one.

The starting Young tableaux $Y_0=Y'_0:=\{\varnothing\}$ are empty. We set the counter of steps  $s$ to be equal to $1$.  
 From this moment, we will perform subsequent steps until we reach $s=n+1$.
The algorithm is as follows.


\begin{enumerate}

\item  If the current step is $s=n+1$, then we finish the algorithm. If $s<n$,  we name $e:=d_s$ \emph{the current number}. Furthermore, we name the first row  the \emph{current row} and assign $r:=1$ ($r$ is the  number of the current row).
\item Find the leftmost number $l$ which is less or equal than the current number in the current row. If such $l$ exists, then go to  step (3). If there is no such an element, then  add a box filled  by $e$ to the end of the current row of $Y_{s-1}$, denote this new Young tableau by $Y_{s}$ and add $n-s+1$ to  the end of the current row of $Y'_{s-1}$. Set $s:=s+1$.   Return to  step $(1)$.
\item Change $l$ in $Y_{s-1}$ by the current number $e$, assign $e:=l$, and change the current row to the next  row (even if it the latter empty) by putting $r:=r+1$. Return to  step (2).
\end{enumerate}

The Young tableaux $Y_{n}$ and $Y'_{n}$ obtained at the last step  constitute the output of the Robinson--Schensted algorithm.

Here is an example. Set \break $\{d_1=5, d_2=1, d_3=3, d_4=2, d_5=3, d_6=6 ,d_7=4\}$.  We start  with  current   step $s=0$,  current row $r=1$,  current number $e=d_1=5$, and $Y_0=Y'_0=\varnothing$.

Below we list all $Y_i$ and $Y'_i$, which we obtain in the course of the Robinson--Schensted algorithm.

$$Y_0=
\varnothing,\:
 Y'_0=
\varnothing,$$

$$Y_1=
\begin{array}{c}
\hline\multicolumn{1}{|c|}{5 }\\
\hhline{-}
\end{array}, \:
 Y'_1=
\begin{array}{c}
\hline\multicolumn{1}{|c|}{7}\\
\hhline{-}
\end{array},$$

\bigskip
$$Y_2=
\begin{array}{cc}
\hline\multicolumn{1}{|c|}{5 }&\multicolumn{1}{|c|}{1 }\\
\hhline{--}
\end{array}, \:
 Y'_2=
\begin{array}{cc}
\hline\multicolumn{1}{|c|}{7}& \multicolumn{1}{|c|}{6 }\\
\hhline{--}
\end{array},$$

\bigskip

$$Y_3=
\begin{array}{cc}
\hline\multicolumn{1}{|c|}{5 }&\multicolumn{1}{|c|}{3 }\\
\hhline{--}\multicolumn{1}{|c|}{1 }&\\
\hhline{-~}
\end{array}, \:
 Y'_3=
\begin{array}{cc}
\hline\multicolumn{1}{|c|}{7}& \multicolumn{1}{|c|}{6 }\\
\hline\multicolumn{1}{|c|}{5}&\\
\hhline{-~}
\end{array},$$

\bigskip

$$Y_4=
\begin{array}{ccc}
\hline\multicolumn{1}{|c|}{5 }&\multicolumn{1}{|c|}{3 }&\multicolumn{1}{|c|}{2 }\\
\hline\multicolumn{1}{|c|}{1}&&\\
\hhline{-~~}
\end{array}, \:
 Y'_4=
\begin{array}{ccc}
\hline\multicolumn{1}{|c|}{7}& \multicolumn{1}{|c|}{6 }&\multicolumn{1}{|c|}{4 }\\
\hline\multicolumn{1}{|c|}{5}&&\\
\hhline{-~~}
\end{array},$$

\bigskip

$$Y_5=
\begin{array}{ccc}
\hline\multicolumn{1}{|c|}{5 }&\multicolumn{1}{|c|}{3 }&\multicolumn{1}{|c|}{2 }\\
\hline\multicolumn{1}{|c|}{3}&&\\
\hhline{-~~}\multicolumn{1}{|c|}{1}&&\\
\hhline{-~~}
\end{array}, \:
 Y'_5=
\begin{array}{ccc}
\hline\multicolumn{1}{|c|}{7}& \multicolumn{1}{|c|}{6 }&\multicolumn{1}{|c|}{4 }\\
\hline\multicolumn{1}{|c|}{5}&&\\
\hhline{-~~}\multicolumn{1}{|c|}{3}&&\\
\hhline{-~~}
\end{array},$$

\bigskip

$$Y_6=
\begin{array}{ccc}
\hline\multicolumn{1}{|c|}{6 }&\multicolumn{1}{|c|}{3 }&\multicolumn{1}{|c|}{2 }\\
\hline\multicolumn{1}{|c|}{5}&&\\
\hhline{-~~}\multicolumn{1}{|c|}{3}&&\\
\hhline{-~~}\multicolumn{1}{|c|}{1}&&\\
\hhline{-~~}
\end{array}, \:
 Y'_6=
\begin{array}{ccc}
\hline\multicolumn{1}{|c|}{7}& \multicolumn{1}{|c|}{6 }&\multicolumn{1}{|c|}{4 }\\
\hline\multicolumn{1}{|c|}{5}&&\\
\hhline{-~~}\multicolumn{1}{|c|}{3}&&\\
\hhline{-~~}\multicolumn{1}{|c|}{2}&&\\
\hhline{-~~}
\end{array},$$

\bigskip

$$Y_7=
\begin{array}{ccc}
\hline\multicolumn{1}{|c|}{6 }&\multicolumn{1}{|c|}{4 }&\multicolumn{1}{|c|}{2 }\\
\hline\multicolumn{1}{|c|}{5}&\multicolumn{1}{|c|}{3}&\\
\hhline{--~}\multicolumn{1}{|c|}{3}&&\\
\hhline{-~~}\multicolumn{1}{|c|}{1}&&\\
\hhline{-~~}
\end{array}, \:
 Y'_7=
\begin{array}{ccc}
\hline\multicolumn{1}{|c|}{7}& \multicolumn{1}{|c|}{6 }&\multicolumn{1}{|c|}{4 }\\
\hline\multicolumn{1}{|c|}{5}&\multicolumn{1}{|c|}{1}&\\
\hhline{--~}\multicolumn{1}{|c|}{3}&&\\
\hhline{-~~}\multicolumn{1}{|c|}{2}&&\\
\hhline{-~~}
\end{array}.$$

\bigskip

We also can apply the Robinson--Schensted algorithm to the elements of the permutation group $S_n$. Let
$$\delta=\biggl(\begin{array}{cccc}
1 & 2 & \dots & n \\
\delta(1) & \delta(2) & \dots & \delta(n)
\end{array}\biggr) $$
be a permutation.
Then we apply the Robinson--Schensted algorithm to the sequence $$\{\delta(1), \delta(2), \dots, \delta(n)\}.$$ We denote the output  Young tableaux by $Y(\delta)$ (insertion tableau) and $Y'(\delta)$ (recording tableau).

\newpage	
\section{Primitive ideals of $U(\mathfrak{o}(\infty))$ and  $U(\mathfrak{sp}(\infty))$}\label{3}
	
Our main result in this section is that every primitive ideal of $U(\mathfrak{o}(\infty))$ or  $U(\mathfrak{sp}(\infty))$ is weakly bounded. This implies that, every primitive ideal  of $U(\mathfrak{o}(\infty))$ is a locally integrable.

Let $U$ stand for  $U(\mathfrak{o}(\infty))$ or  $U(\mathfrak{sp}(\infty))$, $\mathfrak{g}({2n})$ stand for $\mathfrak{o}({2n})$ or  $\mathfrak{sp}({2n})$, and $W(2n)$
stand for the Weyl group of $\mathfrak{g}(2n)$. From now on, we slightly change the notation: we will denote  $\mathfrak{o}({2n})$ and  $\mathfrak{sp}({2n})$  by $\mathfrak{g}(2n)$, while before we used the notation $\mathfrak{g}(n)$. This is needed in order to simplify some formulas.
\subsection{Symbols}

Define a \emph{symbol of type $C_n$} to be a collection of nonnegative integers

\[\Lambda= \left(\begin{array}{c}
	 \alpha \\
	\beta
	\end{array}\right)=\left(\begin{array}{ccc}
	 \alpha_1, & \dots, & \alpha_{m+1} \\
	\beta_1, & \dots, & \beta_{m}
	\end{array} \right), \]
such that $\alpha_i<\alpha_{i+1}$, $\beta_i<\beta_{i+1}$ and $\sum\limits_{i=1}^{m+1} \alpha_i + \sum\limits_{i=1}^m \beta_i = n+m^2$.
We consider the following equivalence relation on the set of symbols of type $C_n$
\[ \left(\begin{array}{ccc}
	 \alpha_1, & \dots, & \alpha_{m+1} \\
	\beta_1, & \dots, & \beta_m
	\end{array} \right)\sim\left(\begin{array}{cccc}
	0, & \alpha_1+1, & \dots, & \alpha_{m+1}+1 \\
	0, & \beta_1+1, & \dots, & \beta_m+1
	\end{array} \right).\]
If  $\alpha_i\leq\beta_i\leq\alpha_{i+1}$ for $1\leq i \leq m$, then the symbol  $\Lambda $ is \emph{special}. The set of  special symbols  in a natural one-to-one correspondence with the set of nilpotent orbits of $\mathfrak{sp}(2n)$ (see \cite{BV}). Take the set $\{2\alpha_i, 2\beta_j+1\}$ for $1\leq i\leq m+1$,  $1\leq j\leq m$, order its elements in increasing order, and denote it by $\{\nu_j\}^{2m+1}_{j=1}$. Then $\nu_C(\Lambda):=\{\nu_j-j+1\}$ is a partition of $2n$.

Define a \emph{symbol of type $D_n$} as a collection of nonnegative integers
	\[\Lambda= \left(\begin{array}{c}
	 \alpha \\
	\beta
	\end{array}\right)=\left(\begin{array}{ccc}
	 \alpha_1, & \dots, & \alpha_m \\
	\beta_1, & \dots, & \beta_m
	\end{array} \right), \]
such that $\alpha_i<\alpha_{i+1}$, $\beta_i<\beta_{i+1}$ and $\sum\limits_{i=1}^m \alpha_i + \sum\limits_{i=1}^m \beta_i = n+m(m-1)$. Introduce the  equivalence relation on the set of symbols of type $D_n$ : 
\[ \left(\begin{array}{c}
	 \alpha \\
	\beta
	\end{array}\right)\sim \left(\begin{array}{c}
	\beta \\
	\alpha
	\end{array} \right), \]
	\[ \left(\begin{array}{ccc}
	 \alpha_1, & \dots, & \alpha_m \\
	\beta_1, & \dots, & \beta_m
	\end{array} \right)\sim\left(\begin{array}{cccc}
	0, & \alpha_1+1, & \dots, & \alpha_m+1 \\
	0, & \beta_1+1, & \dots, & \beta_m+1
	\end{array} \right).\]
If $\beta_i\leq\alpha_i\leq\beta_{i+1}$ or $\alpha_i\leq\beta_i\leq\alpha_{i+1}$ for $1\leq i\leq m-1$, and respectively also $\beta_i\leq\alpha_i$ or $\alpha_i\leq\beta_i$, then we call the symbol  $\Lambda $ \emph{special}. The set of special symbols in one-to-one correspondence with the set of nilpotent orbits of $\mathfrak{o}(2n)$ \cite{BV}. Take the set $\{2\alpha_i+1, 2\beta_i\}$ for all $i$, order its elements in increasing order, and denote it by $\{\nu_j\}^{2m}_{j=1}$. Then $\nu_D(\Lambda):=\{\nu_j-j+1\}$ is a partition of $2n$.

\medskip
	
\subsection{Primitive ideals of $U(\mathfrak{o}(2n))$ and $U(\mathfrak{sp}(2n))$}

\label{prim id}


In this section we recall the classification of primitive ideals of $U(\mathfrak{o}(2n))$ and $U(\mathfrak{sp}(2n))$.

 Let $L(\lambda)$ be the unique irreducible $\mathfrak{g}(2n)$-module with highest weight $\lambda$.  By $W(2n)$ and $\Delta(2n)$ we denote  the Weyl group and the root system of $\mathfrak{g}(2n)$ respectively.  We fix  a set of positive roots $\Delta(2n)^+$ as in Subsection \ref{2.3}.   
 Then we denote $I(\lambda)=\operatorname{Ann}(L(\lambda-\rho_{\mathfrak{g}(2n)}))$.   Recall that our notation for the Killing form is $(\cdot, \cdot)$.

Let  $w$ be an element of $W(2n)$.     Recall       of the  Young tableau $Y(w)$ from Subsection \ref{sub:RS}. Let $q_1, q_2, \dots, q_s$ be the lengths of the rows of $Y(w)$. Note that the Young tableaux $Y(w)$ and $Y(w^{-1})$ have the same shape. We consider the set $\{q_i\}$  as a partition $p(w)$ of $2n$.

\begin{prop}\textup{\cite[Proposition 17]{BV}} 
Given $w\in W(2n)$, there exists a unique symbol $\Lambda=\Lambda(w)$ such that $p(w)=\nu_D(\Lambda)$ (respectively, $\nu_C(\Lambda)$) for $\mathfrak{g}(2n)=\mathfrak{o}(2n)$ (respectively, $\mathfrak{sp}(2n)$).

\end{prop}

Let us describe the construction of $\Lambda(w)$. 
For $\mathfrak{g}(2n)=\mathfrak{sp}(2n)$,  we consider the set $\{q_1,q_2,\dots, q_s\}$ for odd $s$,  and the set $\{0, q_1,q_2,\dots, q_s\}$ for even $s$. We put $\mu_i:=q_i+i-1$ and split the set $\{\mu_i\}$ into two subsets: the set  $\{\bar\alpha_j\}$ of even $\mu_i$'s and the set $\{\bar\beta_j\}$ of odd $\mu_i$'s. The symbol $\Lambda(w)$ is then defined as


\[\left(\begin{array}{ccc}
	 \bar\alpha_1/2 &\dots &\bar\alpha_{[s/2]+1}/2\\
	(\bar\beta_1-1)/2 & \dots & (\bar\beta_{[s/2]}-1)/2
	\end{array}\right) .\]

For $\mathfrak{g}(2n)=\mathfrak{o}(2n)$, we consider the set $\{q_1,q_2,\dots, q_s\}$ for even $s$,  and the set $\{0, q_1,q_2,\dots, q_s\}$ for odd $s$. We let $\mu_i:=q_i+i-1$ and split the set $\{\mu_i\}$ into two subsets: the set $\{\bar\alpha_j\}$ of even $\mu_i$'s  and the set $\{\bar\beta_j\}$ of odd $\mu_i$'s. The symbol $\Lambda(w)$ is then defined as 

\[\Lambda(w)=\left(\begin{array}{ccc}
	 (\bar\alpha_1-1)/2 &\dots &(\bar\alpha_{[s+1/2]}-1)/2\\
	\bar\beta_1/2 & \dots & \bar\beta_{[s+1/2]}/2
	\end{array}\right) .\]

In what follows we put $\nu_{C,D}(w):=\nu_{C,D}(\Lambda(w))$.

From now on, we fix $\lambda$ with the property that $-\lambda$ is dominant. Then we put
$$ \Delta_\lambda := \{\alpha\in \Delta(2n)\mid\: \frac{2(\alpha,\lambda)}{(\alpha,\alpha)}\in \mathbb{Z} \},$$
$$\Delta_\lambda^+:=\Delta(2n)^+\cap \Delta_\lambda,$$
$$W_\lambda:= W(\Delta_\lambda)\subseteq W(2n).$$

One can show  that $W_{\lambda}=W_{[\lambda]}$.

Let $\lambda=\lambda_{-n}\bar\varepsilon_{-n}+ \dots+\lambda_{-1}\bar\varepsilon_{-1}+\lambda_{1}\bar\varepsilon_{1}+\dots+\lambda_{n}\bar\varepsilon_{n}$ (where $\lambda_{-k}=-\lambda_k$) be a weight of $\mathfrak{g}(2n)$. 
 We now introduce an equivalence relation on the set of indices $$[\pm n]:=\{-n,-n+1,\dots,n-1,n\}.$$ By definition two indices $i$ and $j$ are equivalent  if $\lambda_i - \lambda_j \in \mathbb{Z}$. Denote by $E_1$the equivalence  class of indices of $\lambda_i \in \mathbb{Z}$, by $E_2$ the equivalence class of indices of $\lambda_i \in \mathbb{Z}+1/2$, and by $E_3, E_4, \dots $ all other equivalence classes. Note that  the classes $E_i$ are invariant under the action of $W_\lambda$. We can represent $W_\lambda$ as the direct product $W_1\times  W_2\times\dots\times W_s$, where $$W_i=\{w\in W_{\lambda}\mid\:  \restr{w}{[\pm n]\backslash E_i} =\operatorname{id} \}.$$ Let $\Delta_i$ be the root subsystem of $\Delta_\lambda$, which corresponds to the subgroup $W_i$ and $\Delta^+_i=\Delta^i\cap\Delta^+_{\lambda}$.
Then each element $ w\in W_\lambda$  can be uniquely expressed as $w=w_1 w_2   \dots   w_s$ where $w_i\in W_i$.

We define the \emph{symbol} $\Lambda^\lambda(w)$ of  an element  $w=w_1 w_2 \dots w_s \in W_{\lambda}$ to be the   pair of symbols $\Lambda^\lambda(w)=(\Lambda(w_1), \Lambda(w_2))$ 
 We  call $\Lambda^\lambda(w)$ \emph{special} if both symbols $\Lambda({w_1})$ and $\Lambda({w_2})$ are special.
If
\[\Lambda(w)=\left(\begin{array}{ccc}
	 \alpha_1, & \dots, & \alpha_s \\
	\beta_1, & \dots, & \beta_m
	\end{array} \right),\:
\Lambda(w')=\left(\begin{array}{ccc}
	 \alpha'_1, & \dots, & \alpha'_{s'} \\
	\beta'_1, & \dots, & \beta'_{m'}
	\end{array} \right)\]
are symbols (where $s=m+1$ or $s=m$ and $s'=m'+1$ or $s'=m'$),
then we say that $\Lambda(w')$ is a \emph{permutation} of $\Lambda(w)$ if the sets $\{ \alpha_1, \alpha_2, \dots \alpha_s, \beta_1, \beta_2,\dots,\beta_m\}$ and $\{ \alpha'_1, \alpha'_2, \dots \alpha'_{s'}, \beta'_1, \beta'_2,\dots,\beta'_{m'}\}$ coincide.	

Recall that the dot action of  $W(2n)$ on  $\mathfrak{h}_{\mathfrak{g}(2n)}^*$ is defined in Subsection \ref{sub:pre} by setting $w\cdot\lambda=w(\lambda+\rho_{\mathfrak{g}(2n)})-\rho_{\mathfrak{g}(2n)}$.
For an element $w\in W({2n})$, we denote $I(w):=I(w\cdot\lambda)$.
Two elements $w_1$ and $w_2$ of $W(2n)$ are called \emph{equivalent}, written $w_1\sim w_2$, if $I(w_1)=I(w_2)$.

\begin{theorem}\label{teor}\textup{\cite[Theorem 18]{BV}}
Let $\Delta({2n})$ be the root system of type $D_n$ or $C_n$, $W(2n)$ be the Weyl group of $\Delta(2n)$, and $w, w_1, w_2$ be elements of $W(2n)$. Then the following holds.

\begin{enumerate}
			\item The elements $w_1$  and  $w_2$  have the same tableaux  $Y({w_1})=Y({w_2})$  if and only if $w_1\sim w_2$.
				
			\item There exists $w'\in W(2n)$  such that  $w'\sim w$ and the symbol $\Lambda(w')$ of $w'$  is special and is a permutation of $\Lambda(w)$.

		\end{enumerate}
\end{theorem}



Let $\Sigma_\lambda$ be the set of simple roots in $\Delta_\lambda^+$, and let $w \in W_\lambda$. Put

$$S_\lambda(w):=\{\alpha \in \Delta_\lambda^+ \mid w\cdot \alpha \notin \Delta_\lambda^+ \},$$
$$\tau_\lambda(w):=S_\lambda(w)\cap \Sigma_\lambda .$$

\begin{prop}\textup{\cite{J2}} \label{prop}
Let $\alpha \in \Sigma_\lambda$ and  $w \in W_\lambda$. Suppose  $\alpha\in \tau_\lambda(w^{-1})$ is such  that $\tau_\lambda(w^{-1}s_\alpha)\nsubseteq \tau_\lambda(w^{-1})$, where $s_\alpha$ is the reflection corresponding to the root $\alpha$. Then
$$I(s_\alpha w)= I(w).$$
\end{prop}
  
We should also note that, for $\mathfrak{g}(2n)=\mathfrak{sp}(2n)$, $W_1$ is a Weyl group of type $C$, $W_2$ is a Weyl group of type $D$, and $W_i$ for $i\neq 1,2$ is of type $A$. For the case $\mathfrak{g}(2n)=\mathfrak{o}(2n)$, $W_1$ and $W_2$ are  Weyl groups of type $D$, and $W_i$  is of type $A$ for $i\neq 1,2$.

\begin{corollary}  \label{cor}  Let $W_j$ be of type $D_n$ or $C_n$ and let $s_{\alpha}=s_{-i,-i+1}\in W_j$ be the simple reflection corresponding to a root $\alpha=\varepsilon_{-i}-\varepsilon_{-i+1}$, for $ -n+1\leq i\leq -1$. Denote  $v=w^{-1}$ for $w\in W_j$.
Then
$$I(s_{-i,-i+1}w)=I(w)$$
whenever one of the following inequalities holds
\begin{enumerate}
 \item $v(i-1)>v(i+1)>v(i)>0$\textup{ where }$-n<i<-1$,
\item $v(i)>v(i+1)>v(i-1)>0$\textup{ where }$-n<i<-1$,
 \item $v(i-1)>v(i-2)>v(i)>0$\textup{ where }$-n+1<i<0$,
 \item $v(i)>v(i-2)>v(i-1)>0$\textup{ where }$-n+1<i<0$,
 \item $v(i)>0, v(i-1)<0\textup{ and }v(i-1)>v(i+1)$\textup{ where }$-n<i<-1$,
 \item $v(i)<0, v(i-1)>0\textup{ and }v(i+1)>v(i)\textup{ where }$$-n<i<-1$,
 \item  $v(i)>0, v(i-1)<0\textup{ and }v(i-2)>v(i)$\textup{ where }$-n+1<i<-1$,
 \item  $v(i)<0, v(i-1)>0\textup{ and }v(i-2)>v(i-1)$\textup{ where }$-n+1<i<-1$.
\end{enumerate}

\end{corollary}

\begin{proof}
Recall the choice of the sets of  positive and simple roots from  Subsection \ref{2.3}. 
We will argue  simultaneously  in both cases $D_n$ and $C_n$. This is possible because in the proof we only use short roots.

 Assume $\varepsilon_{-i}-\varepsilon_{-i+1} \in  \tau_\lambda(v)$, $v(\varepsilon_{-i}-\varepsilon_{-i+1})\notin \Delta^+_{\lambda}$. This implies exactly one of the  three inequalities:

\begin{enumerate}
 \item[a)] $v(i-1)>v(i)>0$,
\item[b)] $v(i)<v(i-1)<0$,
 \item[c)] $v(i)>0>v(i-1)$.
\end{enumerate}
Clearly, we have  $I(s_{-i,-i+1}w)=I(w)$ if  $\tau_\lambda(vs_\alpha)\nsubseteq \tau_\lambda(v)$, i.e., if there exists  a simple root $\beta=\varepsilon_{-j}-\varepsilon_{-j+1}$ such that $vs_\alpha\cdot\beta\notin \Delta^+_{\lambda} $ and $v\cdot\beta\in\Delta^+_{\lambda}$. 

First, assume that $v$ satisfies inequality $1)$, and hence also  inequality a). We have,
$$v\cdot(\varepsilon_{-i-1}-\varepsilon_{-i})=-\varepsilon_{v(i+1)}+\varepsilon_{v(i)}\in\Delta^+_{\lambda},$$
\textup{because  }$v(i-1)>v({i-2})>0,$  \textup{and}
$$vs_{-i,-i+1}\cdot(\varepsilon_{-i-1}-\varepsilon_{-i})=v\cdot(\varepsilon_{-i-1}-\varepsilon_{-i+1})=-\varepsilon_{v(i+1)}+\varepsilon_{v(i-1)}\notin \Delta^+_{\lambda},$$
 \textup{because  }$v(i-1)>v(i+1)>0.$

Next, assume that $v$ satisfies inequality $3)$, and hence also inequality a). Then,
$$v\cdot(\varepsilon_{-i+1}-\varepsilon_{-i+2})=-\varepsilon_{-v(i-1)}+\varepsilon_{v(i-2)}\in\Delta^+_{\lambda},$$ 
\textup{because }$v(i-2)>v({i})>0,$ \textup{and}
$$vs_{-i,-i+1}\cdot(\varepsilon_{-i+1}-\varepsilon_{-i+2})=v\cdot(\varepsilon_{-i}-\varepsilon_{-i+2})=-\varepsilon_{v(i)}+\varepsilon_{v(i-2)}\notin \Delta^+_{\lambda},$$
\textup{because }$v(i-2)>v({i})>0.$

Now, assume that $v$ satisfies inequality $5)$, hence also inequality c). In this case,
$$v\cdot(\varepsilon_{-i-1}-\varepsilon_{-i})=\varepsilon_{-v(i+1)}+\varepsilon_{v(i)}\in\Delta^+_{\lambda},$$
\textup{because }$v(i)>0>v({i+1}),$ \textup{and}
$$vs_{-i,-i+1}\cdot(\varepsilon_{-i-1}-\varepsilon_{-i})=v\cdot(\varepsilon_{-i-1}-\varepsilon_{-i+1})=\varepsilon_{-v(i+1)}-\varepsilon_{-v(i-1)}\notin \Delta^+_{\lambda},$$
\textup{because }$v(i+1)<v({i-1})<0.$

Finally, assume that $v$ satisfies inequality $7)$, hence also  inequality c). Then,
$$v\cdot(\varepsilon_{-i+1}-\varepsilon_{-i+2})=\varepsilon_{-v(i-1)}+\varepsilon_{v(i-2)}\in\Delta^+_{\lambda},$$
\textup{because }$v(i-2)>0>v(i-1),$ \textup{and}
$$vs_{i,i+1}\cdot(\varepsilon_{-i+1}-\varepsilon_{-i+2})=v\cdot(\varepsilon_{-i}-\varepsilon_{-i+2})=-\varepsilon_{v(i)}+\varepsilon_{v(i+2)}\notin \Delta^+_{\lambda},$$
 \textup{because }$v(i+2)>0>v(i).$

Thus we proved the corollary for the inequalities $1),3),5),7)$. Note that if $I(s_{i,i+1}w)=I(w)$ then $I(s^2_{i,i+1}w)=I(s_{i,i+1}w)$. Hence, the element $w$ satisfies inequality $2),4),6$ or $8)$ if  and only of the element $ws_{i,i+1}$ satisfies  inequality $1),3),5),7)$ respectively. The proof is complete.

\end{proof}

\medskip

\subsection{Primitive ideals of $U(\mathfrak{o}(\infty))$ and $U(\mathfrak{sp}(\infty))$ }

In this section we show that every primitive ideal of $U(\mathfrak{o}(\infty))$ and $U(\mathfrak{sp}(\infty))$ is weakly bounded.

Let $\mathfrak{g}(\infty)$  be  equal to  $\mathfrak{o}(\infty)$ or  $\mathfrak{sp}(\infty)$, and  $\mathfrak{g}({2n})$ be  the Lie algebra $\mathfrak{o}({2n})$ or  $\mathfrak{sp}({2n})$ respectively. As usual, denote by $\operatorname{SL}(2n,\mathbb{C})$ the group of $2n\times 2n$ complex matrices with determinant equal to $1$. Recall that $\mathfrak{g}(2n)=\{X\in \mathfrak{gl}(2n)\mid XF+FX^t=0\}$, where $F$ is defined in Subsection \ref{sub:infdimLA}. Put $$G(2n)=\{g\in {SL}(2n,\mathbb{C})\mid    g^t F g=F\}.$$ Then $\mathfrak{g}(2n)$ is the Lie algebra of the Lie group $G(2n)$. 

Set $U:=U(\mathfrak{g})$ and  $U_{2n}:=U(\mathfrak{g}({2n}))$.

		Our goal in this subsection is to prove the following proposition.

\begin{prop}\label{Prad} Let $I$ be an ideal of $U$, and let  $I_{2n}=I\cap U_{2n}$. Then there exists $r \in \mathbb{Z}_{>0}$ such that,  for $n\gg 0$ (i.e, for each sufficiently large $n$)    the intersection $J(2n)\cap U_{2f(2n)}$ for an arbitrary primitive ideal $J(2n)$ containing $I_{2n}$, is a bounded ideal of $U_{2f(2n)}$ where $f(2n)= \left[\frac{n-3r/2}{r+1}\right]-r/2$.\end{prop}

For the proof of Proposition \ref{Prad} we need to discuss some facts related to the associated variety of a primitive ideal defined in Subsection \ref{sub:var}.
		
 It is clear that \begin{center}if $J_1\subset J_2$ then $\operatorname{Var}(J_2)\subset\operatorname{Var}(J_1)$.\end{center}
If $I$ is an ideal of $U$, then the intersections $I_{2n}=I\cap U_{2n}$ determine a sequence of ${ G}(2n)$-stable varieties $\operatorname{Var}(I_{2n})\subset\frak{g}(2n)^*$, and we have

\begin{equation}\label{Ephimn}\phi_{2m, 2n}(\operatorname{Var}(I_{2m}))\subset\operatorname{Var}(I_{2n})                     \end{equation}
for $m\geq n$, where the map $\phi_{2m, 2n}\colon\frak{g}({2m})^*\to\frak{g}(2n)^*$ is induced by the natural inclusion $\frak{g}(2n)\hookrightarrow\frak{g}({2m})$. 

For any $n\ge2$ and any $r'\in\mathbb Z_{\ge0}$ we put $$\frak{g}(2n)^{\le r'}:=\{x\in\frak{g}(2n)\mid 
\operatorname{rk}(x)
\le r'\},$$
where $\operatorname{rk}$ refers to the rank of a matrix. We identify $\frak{g}(2n)$ and $\frak{g}(2n)^*$ via the Killing form, and so we consider $\frak{g}(2n)^{\le r'}$ as a subset of $\frak{g}(2n)^*$. 

\begin{lemma}\label{Lrk17}Let $I$ be a nonzero ideal of $U$. Then there exists $r\in\mathbb Z_{\ge0}$ such that $$\operatorname{Var}(I_{2n})\subset(\frak{g}(2n))^{\le r}$$ for all $n\gg0$.\end{lemma}
\begin{proof}If $I$ is nonzero then $\operatorname{Var}(I_{2m})\ne\frak{g}(2m)^*$ for some $m\ge 2$. For every $n\ge m$ and every $X\in\operatorname{Var}(I_{2m})$, formula~(\ref{Ephimn}) shows that
$$\phi_{2n, 2m}({ G}(2n)\cdot X)\subset \operatorname{Var}(I_{2m})\ne\frak{g}(2m)^*,$$
where ${G}(2n)\cdot X$ is the coadjoint orbit of $X$ in $\frak{g}(2n)^*$. Hence $\phi_{2n, 2m}({G}(2n) \cdot X)$ is not dense in $\frak{g}(2n)^*$. This, together with \cite[Lemma 4.12]{PP2}, implies the required result for $r=m$ under the assumption that $n>3m$.
\end{proof}

Further, without loss of generality, we assume that the number $r$ is even (we reassign $r:=r+1$ in the case of odd $r$).

A well-known theorem of A.~Joseph \cite{J1} implies that the associated variety of a primitive ideal $J(2n)\subset U_{2n}$ equals the closure of a nilpotent coadjoint orbit. The natural inclusion $\mathfrak{g}(2n)\hookrightarrow \mathfrak{gl}(2n)$ induces the surjection
 $\mathfrak{gl}^*(2n)\twoheadrightarrow \mathfrak{g}^*(2n)$.
  The conjugacy classes of nilpotent $(2n\times 2n)$-matrices surject naturally to the  nilpotent coadjoint orbits of $\frak{g}(2n)$.  
Moreover, these conjugacy classes are related to partitions of $2n$: the partition attached to a conjugacy class comes from the Jordan normal form  of a representative of this class. In this way we attach the partition $s(J(2n))$ of $2n$ to $J(2n)$. 
By $p(J(2n))$ we denote the partition conjugate to that partition. Let $r(J(2n))$  be the difference between $2n$ and the maximal element of $p(J(2n))$. It is easy to check that  $r(2n)\colon=r(J(2n))$ equals the rank of an arbitrary element in the orbit defined by the partition $p(2n):=p(J(2n))$. 
Note that, given $w\in W$ we have  $\mu_{C,D}(\Lambda(w))=p(I(w))$ if $\Lambda(w)$ is special, \cite[Theorem 7]{BV}.

\medskip

\begin{lemma}\label{Lslnrx}Let $X\in\frak{g}(2n)^{\le r}$ be a nilpotent matrix and $p(2n)$ be the partition attached to the conjugacy class of $X$. Then $r(2n)\le r$.\end{lemma}
\begin{proof} We have ${\operatorname{rk}}(X)\le r$.   Then $r(2n)={\operatorname{rk}}X\le r$
\end{proof}

Let $J(2n)=\operatorname{Ann}(L(\lambda'-\rho_{\mathfrak{g}(2n)}))$ be a primitive ideal as in Proposition \ref{Prad} for a weight $\lambda'\in \mathfrak{h}^*_{\mathfrak{g}(2n)}$.
Then there exists $w\in W$  such that $\lambda=w^{-1}\lambda'$ and $-\lambda$ is a dominant weight.  Decompose the subgroup $W_\lambda\simeq W_{1}\times W_2 \times\dots\times W_s$ as in Subsection \ref{prim id}. In what follows, given an element $w'\in W_{\lambda}$, we write $w'=w'_1 w'_2\dots w'_s$ for $w'_i\in W_i$.
In particular, $w=w_1 w_2\dots w_s$ for $w_i\in W_i$. Let $Y(w_i)$ be the Young tableaux defined in Subsection~\ref{2.6}. 
 Lemmas \ref{Lrk17}, 3.5 
 and Theorem  \ref{teor} allow us to  conclude that $2n-r$ is the number of Jordan  blocks of the  conjugacy class corresponding to the nilpotent orbit whose closure is the associated variety of $I(\lambda)$. Next, let $l_i$ be the length of the longest row of $Y(w_i)$. Then $\operatorname{max}\{l_1,\dots,l_s\}=2n-r$, and one can easily see that  in fact $2n-r$ equals  $l_1$ or $l_2$.


Choose an element $w'=w'_1w'_2\dots w'_s$   such that  $ w \sim w'$ and the symbol  $\Lambda^\lambda(w')$ is special (i.e., $\Lambda(w'_1)$ and $\Lambda(w'_2)$ are special).
Since $w\sim w'$, we have $$l:=\operatorname{max}\{l_1,\dots,l_s\}=\operatorname{max}\{l'_1,\dots,l'_s\}=2n-r,$$ where $l'_1,\dots,l'_s$ are defined with respect to the decomposition  $w'=w'_1w'_2\dots w'_s$. Moreover $l=l'_1$ or $l=l'_2$. Therefore $l'_1 \neq l'_2$ for $2n>2r$.    In what follows we assume that $Y(w'_c)$, for  $c=1$ or $2$, has a row of length  $l$ . Note that tableaux of types $C$ and $D$ have even number of elements, and let  $2h$ be the number of elements of $Y(w'_c)$. 


In order to state the next result we need the following definition.

\begin{defn}
 A weight $\gamma=\sum\gamma_i\varepsilon_i\in \mathfrak{h}_{\mathfrak{g}(2n)}^*$  is \emph{half-integral} if $\gamma_i-\gamma_j\in \mathbb{Z}_{}$  and $\gamma_i+\gamma_{j}\in \mathbb{Z}_{}$ for all $i,j$.
\end{defn}


\begin{lemma}\label{qwe}
There exist an element $\tilde w\in W_{\lambda}$, satisfying $\tilde w \sim w'$, and an integer $k(2n)\in 2\mathbb{Z}_{>0}$ such that, after erasing the first   and last  $k(2n)+r$   coordinates of $\tilde w\lambda$, as well as the $2n-2k(2n)-2f(2n)-2r$ central coordinates, we obtain a half-integral dominant regular weight of $\mathfrak{g}({2f(2n)})$ (where \break$f(2n)=[\frac{n-3r/2}{r+1}]-r/2$ as in  Proposition 3.3).

\end{lemma}

\begin{proof}

According  to Theorem \ref{teor} we can suppose    without loss of generality that the symbols $\Lambda(w'_1)$ and $\Lambda(w'_2)$ are special.
Recall that the set of indices of the weight $\lambda$  is $\{-n,-n+1,\dots,-1,1,\dots,n-~1,n\}$.
Note that  interchanging of coordinates of $\lambda$ without changing the order within the classes $E_{i}$ preserves the primitive ideal $I(\lambda)$.
Therefore   we can assume  that
the equivalence class   $E_c$ has the form 
$$\left\{\frac{-\#E_c}{2},\frac{-\#E_c}{2}+1,\dots,-1,1,\dots,\frac{\#E_c}{2}-1,\frac{\#E_c}{2}\right\}.$$ 
 Lemma \ref{Lslnrx} implies  $\#E_c\geq 2n-r$, hence after erasing the first and last $r/2$ coordinates of $\lambda$ we obtain a half-integral weight $\lambda'$.

Let $r'=2h-l$. Note that $r'\leq r$.
 The length of the longest row of $Y(w_c)$ equals  the length of the longest decreasing subsequence of the sequence 
$$a=(w_c'(-h), w_c'(-h+1),\dots,w_c'(-1),  w_c'(1), \dots, w_c'(h-1), w_c'(h)).$$
 Set $k= \#\{i\mid w'_c(i)>0, i>0\}$.  Obviously, the length of the longest decreasing subsequence of $a$ is less or equal to $2h-k$, because  $a$ cannot contain both $w'(i)$ and $w'(-i)$ if $i>0$ and $w'(i)<0$.

Note that the simple reflections  satisfying one of   conditions  $1)-4)$ of Corollary  \ref{cor} preserve the shape of the Young tableau $Y(w'_c)$ (see \cite{K}). 
For each $t\in \mathbb{Z}_{>0}$ and  each $w\in W(2t)$,  we define the sequence of integers
 $$\sigma(w)=\{w(-t), w(-t+1)\dots w(-1), w(1),\dots, w(t-1), w(t)\}.$$
 Let $s_\alpha$ be a simple reflection satisfying one of conditions    $1)-4)$ of Corollary  \ref{cor}. Then, by the very definition of the transformation $s_\alpha$, we can change $\sigma(w'_c)$ to  $\sigma(  s_\alpha w'_c)$ without changing the ideal $I(w')$. Furthermore, $ s_\alpha (\sigma( w'_c ))=\sigma(s_\alpha w'_c)$.  Denote $h(\delta):=\{\delta(-t), \delta(-t+1), \dots, \delta(-1)\}$  for $\delta \in W(2t)$ or $\delta \in W(2t)$. Observe that $h(w'_c)$ determines $\sigma(\delta)$  since $\delta(j)=-\delta(-j)$. 

We call a subsequence $A$ of a finite sequence  of numbers  $\{a_1,a_2\ldots, a_v\}$  \textit{positive interval} if   $A$ has the form $\{a_u, a_{u+1},\ldots a_{w-1}, a_w\}$ for\break $1\leq u<w\leq v$ and any $a_j\in A$ is positive. The notation $|A|$ stands for the number of elements of a positive interval $A$. Note that $\operatorname{max}_{A\subset h(w'_c)}|A|\geq [\frac{n-3r/2}{r+1}]$. Indeed,  the sequence $h(w'_c)$ consists  of $n-r/2$ elements with no more than $r$ negative elements. Consequently, $h(w'_c)$ contains at least $n-3/2r$  positive elements and at most $r+1$ positive intervals. Hence, there exists at least one positive interval with at least $[\frac{n-3r/2}{r+1}]$ elements.

Next, we find the leftmost maximal (by inclusion) positive interval $A_0$ of  the sequence $ h( w'_c)$ such that $|A_0|\geq [\frac{n-3r/2}{r+1}]$. 
Then we apply the Robinson--Schensted algorithm to the interval $A_0$  of $h(w'_c)$, and  denote by $Y_{0}$ the output insertion Young  tableau. Starting   from the bottom left corner of $Y_0$, we place all rows one after another in increasing length order. This transforms  the sequence $h(w'_c)$ to a new sequence  $h'$, and the interval $A_0$ to an interval $A'_0$. 
As a next step, we express the sequence $h'$  as $h(\tilde w_c)$ for some $\tilde w_c\in W(2h)$.

Now, we extend $\tilde w_c$ to $\tilde w\in W(2n)$ by putting $\tilde w_i=w'_i$ for all $i\neq c$. Then \break
$I(\tilde w)=I(w')$, because  $\tilde w(w')^{-1}$ equals a product of simple reflections  satisfying some condition among $1)-4)$ (see \cite{K}).

After that, we erase in $h(\tilde w_c)$ all elements to the left of $A'_0$ and denote the number of these elements by $k(2n)$. Next, we erase all elements to the right of $A'_0$ and denote the number of these elements by $m(2n)$. Note that $$m(2n)<n-r/2-k(2n)-[\frac{n-3r/2}{r+1}]=n-k(2n)-f(2n).$$  In this way we get a sequence which equals $A'_0$.  Since $A'_0$ consists of positive integers or half-integers  ordered as described above ($A'_0$ is comprised of the rows of the Young tableau $Y_0$ in the opposite order),  Lemmas  \ref{Lrk17},  \ref{Lslnrx} show  that after erasing the  first $r/2$ elements of $A'_0$ we obtain a strictly decreasing sequence $A''_0$. This means precisely that after erasing the first   and last $k(2n)+r$  coordinates, as well as the $2n-2k(2n)-2f(2n)$ central coordinates of $\tilde w\lambda$, we obtain a half-integral dominant regular weight of $\mathfrak{g}({2f(2n)})$.
\end{proof}

Let  $\tilde w$ be as in Lemma \ref{qwe}. Denote by $\bar\lambda$ the weight obtained from $\tilde w \lambda$ via replacing by zeros  the first   and last $k(2n)+r$  coordinates, as well as the $2n-2k(2n)-2f(2n)$ central coordinates.

Denote by $\mathfrak{g}(\bar\lambda)$  the root subalgebra of $\mathfrak{g}(2n)$ whose dual Cartan subalgebra $\mathfrak{h}(\bar\lambda)$   is spanned by all elements $\varepsilon_i$ such that $(\bar\lambda,\varepsilon_i )\neq0$.

Let $I$ be a primitive ideal of $U(\mathfrak{g}(\infty))$, and $J(2n)$ be as in Proposition \ref{Prad}.

\begin{corollary}\label{3.8}
In the notation  of Proposition \ref{Prad}, $J(2n)\cap U(\mathfrak{g}(\bar \lambda))$ is a bounded ideal of $U(\mathfrak{g}(\bar \lambda))$. Moreover, this ideal is integrable if $\mathfrak{g}(2n)=\mathfrak{o}({2n})$ or if $\bar\lambda^{2f(2n)}_i\in \mathbb{Z}$.
\end{corollary}
\begin{proof}

By  Lemma \ref{qwe} we can construct an element $\tilde w\in W_{\mathfrak{g}({2n})}$ satisfying $$J(2n)=\operatorname{Ann}_{U_{2n}}L(\tilde w\lambda-\rho_{\mathfrak{g}({2n})}),$$ and such that after erasing  the  first and last $k(2n)+r$ coordinates, as well as the $2n-2k(2n)-2f(2n)$ central coordinates of $\tilde w\lambda$,  we obtain  a  half-integral dominant weight $\bar\lambda^{2f(2n)}$ of $\mathfrak{g}({2f(2n)})$ (so that $\restr{\bar\lambda}{\mathfrak{h}_{\mathfrak{g}(2f(2n))}}=\lambda^{2f(2n)}$).

The module  $L(\bar\lambda^{2f(2n)}-\rho_{\mathfrak{g}(2f(2n))})$ is a simple $\frak{g}(2f(2n))$-module with  highest weight $\bar\lambda^{2f(2n)}-\rho_{\mathfrak{g}(2f(2n))}$.
The ideal $J(2n)\cap U(\frak{g}({\bar\lambda}))$ is a bounded ideal of $U(\frak{g}({\bar\lambda}))$. Indeed, in the case of $\mathfrak{g}(2n)=\mathfrak{o}({2n})$ a half-integral dominant weight is integral dominant, and hence the ideal $J(2n)\cap U(\frak{g}({\bar\lambda}))$ is integrable. If $\mathfrak{g}(2n)=\mathfrak{sp}(2n)$ and ${\bar{\lambda}}^{2f(2n)}_i\in \mathbb{Z}$, then the weight $\bar\lambda^{2f(2n)}$ is integral dominant  hence  $J(2n)\cap U(\frak{g}({\bar\lambda}))$ is an integrable ideal. For the last case, where   $\mathfrak{g}(2n)=\mathfrak{sp}(2n)$ and ${\bar{\lambda}}^{2f(2n)}_i\in \mathbb{Z}-\frac{1}{2}$, the fact that $J(2n)\cap U(\frak{g}({\bar\lambda}))$ is  bounded follows from Lemma \ref{prop:sp.bounded}.
\end{proof}
\medskip

Let $I$ be an ideal of an associative algebra $A$. We denote by $\sqrt I$ the intersection of all primitive ideals of $A$ containing $I$. Note that $\sqrt I$ is the pullback in $A$ of the Jacobson radical of the ring $A/I$. If $I$ is a primitive ideal then $I=\sqrt I$.

\begin{lemma}\label{Lalg}\textup{\cite{PP4}} Assume that the dimension of $A$ is finite or countable. Then the following conditions on an element $z\in A$ are equivalent:

1) $z\in \sqrt I$,

2) for every $a\in A$ there is $k\in\mathbb Z_{>0}$, such that $(az)^k\in I$.\end{lemma}
\begin{proof} The fact that 1) implies 2) follows from~\cite[~Corollary~1.8]{MR}. We will show that 2) implies 1).

Let $z\in A$ satisfy 2), and let $\bar x$ be the image of $x\in A$ in $A/I$. Assume to the contrary that there exists a simple $A/I$-module $M$ such that $\bar z\cdot M\ne0$. Pick $m\in M$  with $\bar z\cdot m\ne0$. There is $ a\in A$ such that $\bar a\cdot(\bar z\cdot m)=m$. Let $k\in\mathbb Z_{>0}$ satisfy $(\bar a\bar z)^k=0$. Then
$0=(\bar z(\bar a\bar z)^k)\cdot m=\bar z\cdot m\ne0.$
This contradiction concludes the proof.
\end{proof}

Now we have all tools to prove Proposition \ref{Prad}.

\begin{proof}
Consider the  primitive ideal $J(2n)=\operatorname{Ann}_{U_{2n}}L(\tilde w\lambda-\rho_{\mathfrak{g}({2n})})$ such that $I_{2n}\subset J(2n)$.   The associated variety Var$(J(2n))$ is the closure of a nilpotent orbit, and there exists an integer $r$ such that the rank  of any  $X\in\operatorname{Var}(J(2n))$ is less or equal $ r$ for all $n$. More precisely, according to Lemma~\ref{Lrk17}, there exists $r\in\mathbb Z_{\ge0}$ such that $\operatorname{Var}(I_{2n})\subset\frak{g}(2n)^{\le r}$ for $n\gg0$. Since $J(2n)\supset I_{2n}$, we have \begin{equation}\operatorname{Var}(J(2n))\subset\operatorname{Var}(I_{2n})\subset\frak{g}(2n)^{\le r}.\label{Erk17}\end{equation}

 Corollary \ref{3.8} states that   $J(2n)\cap U(\frak{g}({\bar\lambda}))$ is a bounded ideal of $U(\frak{g}(\bar\lambda))$, where $\mathfrak{g}(\bar \lambda)\simeq \mathfrak{g}(2f(2n))$ for $f(2n)=[\frac{n-3r/2}{r+1}]-r/2$.
In order to conclude that the ideal $J(2n)\cap U(2f(2n))$ is bounded, it suffices to observe that the root subalgebra  $\mathfrak{g}(\bar\lambda)$ of $\mathfrak{g}(2n)$ is conjugate to  $\mathfrak{g}(2f(2n))$ naturally embedded in $\mathfrak{g}(2n)$.
\end{proof}

\begin{theorem}\label{asd} If $\mathfrak{g}(\infty)=\mathfrak{o}(\infty), \mathfrak{sp}(\infty)$ then any primitive ideal $I\subset U(\mathfrak{g}(\infty))$  is weakly bounded. Moreover, each primitive ideal of $U(\mathfrak{o}(\infty))$ is locally integrable.

\end{theorem}
\begin{proof} Note that $2f((r+1)(2n+r)+3r)=2n$. Hence, Proposition~\ref{Prad} implies the existence of $r\ge0$ such that
$\sqrt{I_{(r+1)(2n+r)+3r}}\cap U_{2n}$ is a bounded (integrable for $\mathfrak{g}(\infty)=\mathfrak{o}(\infty)$) ideal of $U_{2n}$ for $n\gg0$. 
 Next, Lemma~\ref{Lalg} shows that $(\sqrt I)_{2n}=\cap_{2n'\ge 2n}\sqrt{I_{2n'}}$ for all $n\ge 2$. However, $$\cap_{2n'\ge2 n}\sqrt{I_{2n'}}=(\cap_{2n'\ge {(r+1)(2n+r)+3r}}\sqrt{I_{2n'}})\cap U_{2n}.$$
Being bounded (integrable for $\mathfrak{g}(\infty)=\mathfrak{o}(\infty)$) in $U_{2f(n')}$, the ideal $\sqrt{I_{2n'}}\cap U_{2f(n')}$ is an intersection of bounded  ideals (ideals of finite  codimension for $\mathfrak{g}(\infty)=\mathfrak{o}(\infty)$)  in $U_{2f(2n')}$, hence $$(\sqrt I)_{2n}=(\cap_{2n'\ge {(r+1)(2n+r)+3r}}\sqrt{I_{2n'}})\cap U_{2n}$$ is an intersection of bounded ideals (ideals of finite codimension for $\mathfrak{g}(\infty)=\mathfrak{o}(\infty)$) in $U_{2n}$. This means that the ideal $(\sqrt I)_{2n}$ is bounded (integrable for $\mathfrak{g}(\infty)=\mathfrak{o}(\infty)$) for $n\gg 0$. A very similar argument shows that $(\sqrt I)_{2n}$ is bounded (integrable for $\mathfrak{g}(\infty)=\mathfrak{o}(\infty)$) for all $n\ge 2$. 
\end{proof}
\newpage

\section{ Integrable  ideals  and c.l.s. for $U(\mathfrak{o}(\infty))$ and $U(\mathfrak{sp}(\infty))$}\label{4}

In this section we recall the notion of precoherent local system. It allows us to show that each primitive locally integrable ideal of $U(\mathfrak{o}(\infty))$ and $U(\mathfrak{sp}(\infty))$ is integrable.

\subsection{Precoherent local systems and integrability of locally integrable  ideals  }

Let $\mathfrak{g}(\infty)=\mathfrak{o}(\infty),\mathfrak{sp}(\infty)$.
Our next goal is to establish that the conditions on an ideal $I$ in $U(\mathfrak{g}(\infty))$ to be integrable and locally integrable are equivalent,  see  Theorem~\ref{asd1} below.

	Let $I $ be a locally integrable ideal of $U(\mathfrak{g}(\infty))$. For every $n
	\in \mathbb{Z}_{\geq0}$, $I \cap U(\mathfrak{g}(2n))$ is an intersection of  ideals of
	finite codimension in $U(\mathfrak{g}(2n))$. Thus, $I \cap
	U(\mathfrak{g}(2n))$ is an intersection of annihilators of finite-dimensional
	$\mathfrak{g}(2n)$-modules. Since any finite-dimensional module of a  semisimple
	finite-dimensional Lie algebra is semisimple, it follows that $I \cap
	U(\mathfrak{g}(2n))$ is an intersection of annihilators of simple finite-dimensional  $U(\mathfrak{g}(2n))$-modules.
Recall that by $Irr_{2n}$ we denote the set of classes of isomorphism of simple $\mathfrak{g}(2n)$-modules. 
	

			\begin{defn}
			A \textup{precoherent local system of modules} (further p.l.s.)  for
			$\mathfrak{g}(\infty)$ is a collection of sets
			
			$$\{Q_n\}_{n \in \mathbb{Z}_{\geq 1}} \subset \Pi_{n \in \mathbb{Z}_{\geq 1} }
			Irr_{2n} $$
			such that $Q_m \supset \langle Q_n \rangle_m$ for any $n\geq m$, where	$\langle Q_n
			\rangle_m$ denotes the set of all simple $\mathfrak{g}({2m})$-constituents of the
			$\mathfrak{g}({2n})$-modules from $Q_n$.
			
		\end{defn} 

The fact that, for any locally integrable ideal $I$, the intersection  $I \cap
	U(\mathfrak{g}(2n))$ is an intersection of annihilators of simple finite-dimensional \break $U(\mathfrak{g}(2n))$-modules implies that we can extend the definition of $Q(I)$ (see Section \ref{sub:tensor}) to locally integrable ideals. Namely, if $I\subset U(\mathfrak{g}(\infty))$ is locally integrable ideal  we define  the p.l.s.  $Q(I)$  by putting

	$$ Q(I)_n:=\{z\in Irr_{2n} \mid I \cap U(\mathfrak{g}(2n)) \subset 
	\operatorname{Ann}_{U(\mathfrak{g}(2n))}z \}.$$
It is clear that $Q(I)$ is a p.l.s., because if $I\cap U(\mathfrak{g}(2n))$ annihilates a module $M$ from a class $z\in Irr_{2n}$ then $I\cap U(\mathfrak{g}(2m))$ annihilates all simple constituents of $M$ as $\mathfrak{g}(2m)$-module, for $n\geq m$.

Moreover, we claim that $I(Q(I))=I$ (see Subsection \ref{2.9} for the definition $I(Q)$).
 Indeed, from the definitions of $Q(I)$ and $I(Q)$ we have that $I(Q(I))\cap U(\mathfrak{g}(2n))$ equals  the  intersection of the   annihilators of  simple modules annihilated by  $I\cap U(\mathfrak{g}(2n))$ 
and, possibly, some  annihilators  of isomorphism classes $\bar z$ such that $\operatorname{Ann}_{U(\mathfrak{g}(2n))}\bar z \supset I\cap U(\mathfrak{g}(2n))$.  Since $$I\cap U(\mathfrak{g}(2n)) \cap\operatorname{Ann}_{U(\mathfrak{g}(2n))}\bar z =I\cap U(\mathfrak{g}(2n))$$ for any  $\bar z$, we conclude $ I(Q(I))\cap U(\mathfrak{g}(2n))=I\cap U(\mathfrak{g}(2n))$.

	Now we are ready to formulate the main result of this section.
	This  is an analogue of  I. Penkov's and A. Petukhov's  result for   $\mathfrak{sl}(\infty)$ \cite[Theorem 4.2]{PP4}.
	\begin{theorem}\label{asd1} 
		
		If $I \subset U(\mathfrak{g}(\infty))$ is a locally integrable ideal then $I $ is integrable.
	\end{theorem}

	Since $I(Q(I))=I$ for any locally integrable ideal $I$, Theorem \ref{asd1} follows from  the following proposition.
	
	\begin{prop}
\label{clsint}
		If $Q$ is a p.l.s. then $I(Q)$ is an integrable ideal.
	\end{prop}
	
\begin{defn}
		 Two p.l.s. $Q$ and $Q'$ are \textup {equivalent} if there exists an integer $n$ such that $Q_{n'} = Q'_{n'}$ for any $n' > n$.
		
\end{defn} 

It follows directly from the definition of equivalence of p.l.s. that \break $I(Q)=I(Q')$ whenever $Q$ and $Q'$ are equivalent p.l.s. 
Thus, to prove Proposition \ref{clsint} it is enough to prove the following proposition.
	
	\begin{prop}
\label{4.3}
		For any p.l.s. $Q$, there exists a c.l.s. $Q'$ such that $Q$ and $Q'$ are equivalent.
	\end{prop}

The rest of this section is devoted to the proof of Proposition \ref{clsint}.

\medskip
\subsection{Equivalence of p.l.s. and c.l.s.}
\label{prop}

In this subsection we provide a somewhat explicit construction  of a c.l.s. $Q'$ which is equivalent to a given p.l.s. $Q$, and  in this way give a  proof of Proposition \ref{4.3}.


Finite-dimensional  $\mathfrak{o}(2n)$-modules are in one-to-one correspondence with  $n$-tuples  $ \lambda = (\lambda_1,...,\lambda_n) $ of numbers $\lambda_i$ 
which are simultaneously either integers or half-integers and satisfy
	\begin{equation} \label{o:case}
	\lambda_1 \geqslant \lambda_2 \geqslant ... \geqslant \lambda_{n-1} \geqslant |\lambda_n|.
\end{equation}
	Finite-dimensional  $\mathfrak{sp}(2n)$-modules are in one-to-one correspondence with  $n$-tuples  $ \lambda = (\lambda_1,...,\lambda_n) $ of integers $\lambda_i$ satisfying 
	\begin{equation}\label{sp:case}\lambda_1 \geqslant \lambda_2 \geqslant ... \geqslant \lambda_{n-1} \geqslant \lambda_n \geqslant 0.\end{equation}

	This correspondence is nothing but the assigment of the highest weight $\lambda$ to a  simple finite-dimensional module $L$ so that $ L=L(\lambda)$. In what follows we call $n$-tuples satisfying~(\ref{o:case})  in the case of  $\mathfrak{o}(2n)$ and (\ref{sp:case}) in the case of $\mathfrak{sp}(2n)$ \emph{admissible $n$-tuples}.
	We refer to $n$ as the  \textit{width} of $\lambda$, and write  $\#\lambda=n$.

	\emph{The Gelfand--Tsetlin rule} \cite{Mo} claims that, for two admissible $n$-tuples $\lambda$ and $\mu $ with $  \#\lambda = n$,  $ \#\mu = n-1$, the following conditions are equivalent.
	\begin{itemize}
		
		\item $ \operatorname{Hom
}_{\mathfrak{g}({2\#\mu})} ( L({\mu}) , L({\lambda}) ) \neq 0.$
		
		\item There  exists an $n$-tuple of integers $\nu = (\nu_1,\dots,\nu_n)$ which satisfy the inequalities
\begin{equation*}
\begin{split}
\lambda_1 \geqslant \nu_1 \geqslant \lambda_2 \geqslant\nu_2 \geqslant  &\ldots \geqslant \lambda_{n-1} \geqslant \nu_{n-1} \geqslant  \lambda_n\geqslant  \nu_{n}\geqslant 0,\\
\nu_1 \geqslant \mu_1 \geqslant\nu_2 \geqslant \mu_2 \geqslant   &\ldots \geqslant \nu_{n-1} \geqslant \mu_{n-1} \geqslant \nu_{n} \geqslant 0  
\end{split}
\end{equation*}
 in the case of  $\mathfrak{sp}(2n)$.		
		
\item There exists a $(n-1)$-tuple of integers or half-integers $\nu = (\nu_1,\dots,\nu_n)$ which satisfy the inequalities
\begin{equation*}
\begin{split}
 \lambda_1 \geqslant \nu_1 \geqslant \lambda_2 \geqslant\nu_2 \geqslant  &\ldots \geqslant \nu_{n-2} \geqslant \lambda_{n-1} \geqslant \nu_{n-1} \geqslant |\lambda_n | ,\\
\nu_1 \geqslant \mu_1 \geqslant\nu_2 \geqslant \mu_2 \geqslant   &\ldots \geqslant \mu_{n-2} \geqslant \nu_{n-1} \geqslant| \mu_{n-1} |  
\end{split}
\end{equation*}	
in the case of $\mathfrak{o}(2n)$ .

	\end{itemize}

	We note that the set of admissible $n$-tuples for $\mathfrak{sp}(2n)$ is a subset of the set of admissible $n$-tuples for $\mathfrak{o}(2n)$. Furthermore, the Gelfand--Tsetlin rule for the set of admissible  $n$-tuples for $\mathfrak{sp}(2n)$ can be obtained by  restriction of the Gelfand--Tsetlin  rule  for $\mathfrak{o}(2n)$. 
We will write $\lambda > \mu$ whenever a pair $(\lambda,\mu)$ as above satisfies the Gelfand--Tsetlin rule.
For  tuples $\lambda$ and $\mu $ with $\#\lambda \geqslant \# \mu$, the Gelfand--Tsetlin rule implies that the following conditions are equivalent.
\begin{itemize}
	
	\item $\operatorname{Hom}_{\mathfrak{g}(2\#\mu)} ( L({\mu}) , L({\lambda}) ) \neq 0 $
	
	\item There exists a sequence of admissible tuples $\lambda=\lambda^0,\lambda^1,\ldots,\lambda^{\#\lambda-\#\mu}=\mu$ such that $\#\lambda^i=\#\lambda-i$ and
	$$\lambda=\lambda^0>\lambda^1>\ldots>\lambda^{\#\lambda-\#\mu}=\mu$$
	We write $\lambda \succ \mu$ whenever these conditions hold.
\end{itemize}

By writing $\{a_1,a_2\}\geqslant\{b_1,b_2\}$ we indicate the  validity of all  inequalities $a_i \geqslant b_j$ for all $i$ and $j$. 
	Using this notation we can rewrite the Gelfand--Tsetlin rule in a more convenient form.
	It is easy to check that the following two conditions are equivalent:
	
	\begin{itemize}
		
		\item $\lambda \succ \mu$ and $\#\lambda - \#\mu=1$,

		\item  $\{\lambda_1,\lambda_1\}\geqslant \{\lambda_2,\mu_1\}\geqslant\{\lambda_3,\mu_2\}\geqslant \ldots \geqslant \{\lambda_{n-1},\mu_{n-2}\}\geqslant\{|\lambda_{n}|,|\mu_{n-1}|\}$.
	\end{itemize}

	
	We can now rephrase the definitions of p.l.s. and c.l.s. 
	
\noindent a$)$ The following conditions are equivalent:
	\begin{itemize}
		\item $Q$ is a p.l.s.
		
		\item for all $\lambda$, $\mu$ such that $\lambda \succ \mu$ and $\lambda \in Q_{\#\lambda}$, we have $\mu \in Q_{\#\mu}. $
	\end{itemize}
	b$)$ The following conditions are equivalent:
	\begin{itemize}
		\item $Q$ a is c.l.s.
		
		\item $Q$ a is p.l.s. and for every $\mu \in Q_{\#\mu}$ there is $\lambda \in Q_{\#\mu}$ such that $\lambda \succ \mu .$ \\
		
	\end{itemize}

We denote by $Q^{\vee}(\lambda)$ the largest p.l.s. $Q$ which does not contain a given tuple $\lambda$.
	
\begin{prop}
\label{plsnota}		For any admissible $n$-tuple $\lambda$, the p.l.s. $Q^{\vee}(\lambda)$ is equivalent to the c.l.s.
		$$Q(\lambda):=\bigcup\limits_{1\leqslant k \leqslant \#\lambda} Q(k,\lambda_k), $$
		where the collection of sets $Q(k,a)$ for $k\in \mathbb{Z}_{\geq 0}$ and $a\in \mathbb{Z}/2$ is defined by putting
		$$Q(k,a)_m:=\{\mu \in Irr_{2m}\mid\mu_k<a,\:if\:k\leq\#\mu\}.$$\textbf{}

	\end{prop}

Note that Proposition \ref{plsnota} implies Proposition \ref{4.3}.
Indeed, let $Q$  be a p.l.s. Then 
	$$Q=\cap_{\lambda \notin Q}Q^{\vee}(\lambda).$$
	According to Proposition \ref{plsnota}, each p.l.s.  $Q^{\vee}(\lambda)$ is equivalent to a c.l.s. $Q(\lambda)$. The lattice of c.l.s. is artinian \cite{Zh3}, and therefore we conclude that the p.l.s. $Q$ is equivalent to the c.l.s. 
	
	$$Q(\lambda_1)\cap Q(\lambda_2)\cap\ldots\cap Q(\lambda_s)$$
	for some finite set of elements $\lambda_1,\lambda_2,\ldots,\lambda_s \notin Q$.

It remains to prove Proposition \ref{plsnota}. It is clear that Proposition \ref{plsnota} follows from the following lemma.

	\begin{lemma} \label{equiv}
Let $\lambda$ and $\mu$ be  admissible tuples such that $\#\mu \geqslant 2\#\lambda$. Then the following conditions are equivalent:
	\begin{enumerate}
		\item $\mu \succ \lambda$,
		\item $\mu_k\geqslant\lambda_k$ for each $1\leqslant k \leqslant \#\lambda$.
	\end{enumerate}
	\end{lemma}
	
	Without loss of generality, we can consider our admissible tuples as admissible tuples of integers.  Indeed, if  $\lambda \succ\mu$  and both admissible tuples consist of half-integers, we can add $1/2$ to all entires of the corresponding admissible tuples.

	Let $\lambda=(\lambda_1,\ldots,\lambda_n)$ be an admissible tuple and $k \in \mathbb{Z}$. Set
	\begin{equation*}
R(\lambda,k):= 
\begin{cases}
(\lambda_1,\ldots,\lambda_i,k,\lambda_{i+2},\ldots\lambda_n) &\text{if $k\geq |\lambda_{n}|$} \\
\lambda&\text{if $k < |\lambda_{n}|$, }
\end{cases}
\end{equation*}
	so that $R(\lambda,k)$ is an admissible tuple and $i+1$ is  maximal possible such that $k\geq \lambda_{i+1}$. Set
	\begin{equation*}
L(\lambda,k):= 
\begin{cases}
(\lambda_1,\ldots,\lambda_i,k,\lambda_{i+2},\ldots\lambda_n) &\text{if $k\leq \lambda_{1}$} \\
\lambda&\text{if $k > \lambda_{1}$, }
\end{cases}
\end{equation*}
so that $L(\lambda,k)$ is an admissible tuple and $i$ is maximal possible such that $k\leq \lambda_{i+1}$ .
	
	Let us prove the following technical lemma.
	\newpage
	
	\begin{lemma}
		\label{RandL}	
		Let $\lambda$ and $\mu$ be admissible tuples such that $\#\mu-\#\lambda=1$, $\mu > \lambda$ and let  $k \in \mathbb{Z}$. Then 
		\begin{enumerate}
			\item $R(\mu,k)>R(\lambda,k)$,
				
			\item $L(\mu,k)>L(\lambda,k)$,
				\item $L(\mu,k)>R(\lambda,k)$ whenever one of the following conditions is satisfied
$$ \mu_{i+1} \geqslant k > \mu_{i+2},$$
			$$\mu_{i+2} \geqslant k \geqslant \mu_{i+3}.  $$
for  $i$ such that $\lambda_{i} \geqslant k \geqslant \lambda_{i+1}$.

		\end{enumerate}
\end{lemma}

		\begin{proof}

$1)$ Obviously, $\mu_{i} \geqslant k \geqslant \mu_{i+3}$. There are  three possibilities:
			$$ \mu_{i} \geqslant k > \mu_{i+1,}\eqno{(*)}$$
			$$ \mu_{i+1} \geqslant k > \mu_{i+2},\eqno{(*~*)}$$
			$$ \mu_{i+2} \geqslant k \geqslant \mu_{i+3}.\eqno{(***)}$$
			
			In the case $(*)$ we have $R(\lambda,k)=\{\lambda_1,\ldots,\lambda_i,k,\lambda_{i+2},\ldots,\lambda_n \}$,
			$R(\mu,k)=\{\mu_1,\ldots,\mu_i,k,\mu_{i+2},\ldots,\mu_n \}$.
			To check that $R(\mu,k)>R(\lambda,k)$, we need to check that the corresponding  inequalities hold. But all inequalities not involving $k$ follow from the inequality $\mu >\lambda$. Hence it remains to check that

			$$\{\mu_{i},\lambda_{i-1}\} \geqslant \{k,\lambda_{i}\} \geqslant \{\mu_{i+2},k\} \geqslant \{\mu_{i+3},\lambda_{i+2}\}.$$
			These inequalities are implied by $(*)$. The cases $(*~*)$ and
			$(***)$ can be  verified in a similar way.

		$2)$	The proof is similar to  case $1)$.

	$3)$   We have
		
		$$\lambda_{i} \geqslant k \geqslant \lambda_{i+1}$$
		and one of the three cases $(*)$, $(*~*)$, $(***)$ holds.
		
		For  cases $(*~*)$ and $(***)$ one can easily check the following inequalities
		$$\{\mu_{i},\lambda_{i-1}\} \geqslant \{k,\lambda_{i}\} \geqslant \{\mu_{i+2},k\} \geqslant \{\mu_{i+3},\lambda_{i+2}\},$$
		$$ \{\mu_{i+1},\lambda_{i}\} \geqslant \{k,k\} \geqslant \{\mu_{i+3},\lambda_{i+2}\}.$$
		So, $L(\mu,k)>R(\lambda,k)$ whenever  one of the conditions $(*~*)$ and $(***)$ is satisfied.
\end{proof}

Note that	in the case $(*)$ the inequality  $L(\mu,k)>R(\lambda,k)$ may be false. Indeed
		$$\{\mu_{i-1},\lambda_{i-2}\} \geqslant \{k,\lambda_{i-1}\} \geqslant \{\mu_{i+1},\lambda_{i}\} \geqslant \{\mu_{i+2},k\} \geqslant \{\mu_{i+3},\lambda_{i+2}\}.$$
				We see that  $\lambda_i \geqslant k$ and $k \geqslant \lambda_i$. So $k=\lambda_i$, which is false in the general case.

We are now ready to prove Lemma \ref{equiv}.
	We may assume without loss of generality that  $\#\mu=2\#\lambda$. Indeed, the entries $\lambda_i$ are independent of the entries $\mu_{2n+j}$ for $i,j\in \mathbb{Z}_{>0}$.
	\begin{proof}[Proof of Lemma 4.5]
	The implication  $ 1) \Rightarrow 2)$ is obvious.
	We need to show that $ 2) \Rightarrow 1)$.
	
	Let $\lambda = (\lambda_{1},\lambda_{2},\ldots,\lambda_{n})$ and $k \leqslant \lambda_{n}$. Put $\lambda[k]=(\lambda_{1},\lambda_{2},\ldots,\lambda_{n},k)$. Obviously, if $\lambda\prec\mu$ then $\lambda[k]<\mu[k]$.
We will proceed by induction on $\#\lambda$.
The base:
if	$\#\lambda = 1$, $\#\mu=2$ and $\mu_{1}>\lambda_{1}$, then, clearly, $\mu>\lambda$.

	 To perform the inductive step, we introduce the following notation: for admissible tuples $\lambda= (\lambda_{1},\lambda_{2},\ldots,\lambda_{n})$ and $\mu= (\mu_{1},\mu_{2},\ldots,\mu_{2n})$ such that $\mu \succ \lambda$, and for $\mu_{n+1} \geqslant \lambda_{n+1}$,  $\mu_{2n}\geq\mu_{2n+1}\geq\mu_{2n+2}$, we set
	\begin{equation*}
	\begin{split}
	\lambda^*&:=(\lambda_{1},\lambda_{2},\ldots,\lambda_{n},\lambda_{n+1}), \\
	\mu^*&:=(\mu_{1},\mu_{2},\ldots,\mu_{2n},\mu_{2n+1},\mu_{2n+2}). 
	\end{split}
	\end{equation*}
	
	The inequality $\mu\succ\lambda$ means that we have a chain of admissible tuples

	$$\mu=\lambda^0>\lambda^1>\ldots>\lambda^n=\lambda.$$
	Without loss of generality, we can assume that the last entries of all $\lambda^i$ for $0 \leqslant i \leqslant n-1$ all equal to $\mu_{2n}$. Denote $\mu_{2n}$ by $m$. Put
	$$\mu[m]=\lambda^0[m]>\lambda^1[m]>\ldots>\lambda^n[m]=\lambda[m].$$
	By Lemma \ref{RandL} we have
\begin{equation}\begin{split}
	R(\mu[m],\lambda_{n+1})=&R(\lambda^0[m],\lambda_{n+1})>R(\lambda^1[m],\lambda_{n+1})>\ldots\\&>R(\lambda^n[m],\lambda_{n+1})=R(\lambda[m],\lambda_{n+1}),\\
	L(\mu[m],\lambda_{n+1})=&L(\lambda^0[m],\lambda_{n+1})>L(\lambda^1[m],\lambda_{n+1})>\ldots\\&>L(\lambda^n[m],\lambda_{n+1})=L(\lambda[m],\lambda_{n+1}).
	\end{split}\end{equation}
	
	 Now we are going to show that there exists a pair $(i-1,i)$, where $1 \leqslant i \leqslant n$, for which $\lambda^{i-1}$ and $ \lambda^{i}$ satisfy  conditions $(*~*)$ or $(***)$, i.e., if $\lambda^{i-1}[m]_{k}\geq\lambda_{n+1}\geq\lambda^{i-1}[m]_{k+1}$ then 
	 
	 $$\lambda^{i}[m]_{k+1}\geq\lambda_{n+1}\geq\lambda^{i}[m]_{k+2}$$
	 or
	 $$\lambda^{i}[m]_{k+2}\geq\lambda_{n+1}\geq\lambda^{i}[m]_{k+3}.$$
	 Indeed, if for all pairs $(i-1,i)$  condition $(*)$ is satisfied, then
	$$\lambda_{n+1} \geqslant \lambda[m]^{n-1}_{n+1},\lambda_{n+1} \geqslant \lambda[m]^{n-2}_{n+1},\ldots,\lambda_{n+1} \geqslant \lambda[m]^{0}_{n+1}.$$
	We can rewrite the last inequality as $\lambda_{n+1} \geqslant \mu_{n+1}$. But  we also have  $\lambda_{n+1} \leqslant \mu_{n+1}$. This is a contradiction.
	
	Therefore,  a pair $(i-1,i)$ exists as required.
	This means that $L(\mu[m],\lambda_{n+1})\succ R(\lambda[m],\lambda_{n+1}) $. If $\lambda_{n+1} \geqslant m$ then $ R(\lambda[m],\lambda_{n+1}) = \lambda^*$, because $\lambda_n \geqslant \lambda_{n+1} \geqslant m $. If $\lambda_{n+1} \leqslant m$, we replace $m$ by $\lambda_{n+1}$, and then $L(\mu[m],\lambda_{n+1}) \succ \lambda^*$. This implies
	
	$$L(\mu[m],\lambda_{n+1})=\{\mu_1,\ldots,\mu_{j-1},\lambda_{n+1},\mu_{j+1},\ldots,\mu_{2n},m=\mu_{2n+1} \},$$
	 where 
	   \begin{equation}
	 	\mu_{j} \geqslant \lambda_{n+1} \geqslant \mu_{j+1}
	 \end{equation}
	    for some $j \geqslant n+1$.

	The last thing we need to prove is that  $\mu^*\succ L(\mu[m],\lambda_{n+1})$. For this we need to check the inequalities
	
	\begin{equation*}
	\begin{split}
\mu_{1} \geqslant \{\mu_{2},\mu_{1}\}\geqslant &\ldots\geqslant \{\mu_{j},\mu_{j-1}\}\geqslant \{\mu_{j+1},\lambda_{n+1}\}\geqslant \\\{\mu_{j+2},\mu_{j+1}\}\geqslant &\ldots \geqslant \{\mu_{2n+1},\mu_{2n}\}\geqslant \{\mu_{2n+2},\mu_{2n+1}\}.
\end{split}
\end{equation*}

	All inequalities except
	
	$$\mu_j\geq \lambda_{n+1},\: \mu_j-1\geq \lambda_{n+1},\:\mu_j+1\leq \lambda_{n+1},\:
	\mu_j+2\leq \lambda_{n+1}$$
	are obvious, and these latter inequalities follow from inequality $(3)$.
	This proves the inductive step  and hence the lemma.	
\end{proof}

\newpage

\medskip
\section{Coherent local systems of bounded ideals for $U(\mathfrak{sp}(\infty))$}\label{5}

In this section we generalize the construction of c.l.s. of  simple finite-dimensional $\mathfrak{sp}(2n)$-modules to collections  of bounded   ideals in $U(\mathfrak{sp}(2n))$ when $n$ runs over $\mathbb{Z}_{> 0}$. To do this, we use some tools from Kazhdan--Lusztig theory.

\medskip
\subsection{Kazhdan--Lusztig theory}
\label{sub:KL}
Here we introduce some basic definitions and facts related to Coxeter groups and Kazhdan--Lusztig theory.

\begin{defn}
			Let $G$ be a group with identity $1_G$.  For a (not necessarily finite) subset $S$ of $G$, we say that $G$ is a \emph{(generalized) Coxeter group with respect to }$S$, or that $(G,S)$ is a \emph{(generalized) Coxeter system}, if $G$ is generated by $S$ with a presentation of the form 
			$$
				G=\left\langle S\mid  (st)^{m_{s,t}}=1_G\text{ for }s,t\in S\right\rangle,
			$$
			where, for each $s,t\in S$,
			$
				m_{s,t}=m_{t,s}
			$
		is a positive integer or $\infty$, and, for all $s \in S$,
			$
				m_{s,s}=1\,.
			$
The condition $\displaystyle m_{s,t}=\infty $ means that no relation of the form $\displaystyle (st)^{m}=1_G$ should be imposed.

		The \emph{Coxeter matrix} of $G$ is given by $\left[m_{s,t}\right]_{s,t\in S}$.  We write $\bar{S}$ for the set 
		$$
			\left\{gsg^{-1}\mid\:s\in S\text{ and }g\in G\right\}.
		$$

		\end{defn}

\begin{defn}
			The \emph{Bruhat length} $\ell^G$ of a Coxeter system $(G,S)$ is given by the function $\ell^G \colon G\to\mathbb{Z}_{\geq 0}$ defined in the following way: for all \break$g\in G$, $\ell^G(g)$ is the smallest integer $k\geq 0$ such that $g=s_1s_2\cdots s_k$ for some $s_1,s_2,\ldots,s_k\in S$.   We say that $g=s_1s_2\cdots s_k$ is a \emph{reduced expression}  for $g\in G$ if $s_1,s_2,\ldots,s_k\in S$ and $k=\ell^G(g)$.
		\end{defn}

		\begin{defn}
			Let $(G,S)$ be a Coxeter system and $g,h\in G$.  Then, we write $g \preccurlyeq  h$ if there is a reduced expression $h=s_1s_2\dots s_k$, where $k\in\mathbb{Z}_{\geq 0}$ and $s_1,s_2,\ldots,s_k\in S$, such that $g$ is a \emph{subword} of $s_1s_2\dots s_k$, i.e., there exist $j\in\mathbb{Z}_{\geq 0}$ and integers $i_1,i_2,\ldots,i_j$ with $1\leq i_1<i_2<\ldots <i_j\leq k$ such that $g=s_{i_1}s_{i_2}\dots s_{i_j}$.  The relation $\preccurlyeq $ is called the \emph{Bruhat order} on $G$. It is a well-known fact that this is a partial order. 
		\end{defn}
We will write $g\prec h$ if  $g\preccurlyeq h$ and $g\neq h$.
	\begin{defn}
		Let $(G,S)$ be a  Coxeter system and $q$ be an indeterminate.   The ring $\mathbb{Z}\left[q^{-\frac{1}{2}},q^{+\frac{1}{2}}\right]$ of Laurent polynomials in $q^{\frac{1}{2}}$ is denoted by $\mathcal{A}$.  The \emph{Hecke algebra} $\mathcal{H}$ is an associative algebra which is a free module over $\mathcal{A}$ with generating set $\left\{T_g\mid\:g\in G\right\}$, such that the multiplicative identity of $\mathcal{H}$ is $1_{\mathcal{H}}=T_{1_G}$ and that the following multiplicative relations are satisfied:
		$$
			T_s^2 = (q-1)\,T_s + q\,T_{1_G}\,,
		$$
		$$
			T_g\,T_s = T_{gs} \,\text{   if }g \prec gs\,,
		$$
		$$
			T_s\,T_g = T_{sg} \,\text{   if }g \prec sg\,,
		$$
		for each $s \in S$ and $g \in G$.
	\end{defn}

This algebra has an involution ${v}\mapsto\bar v$ for $v\in\mathcal{H}$ which sends $q^{1/2}$ to $q^{-1/2}$ and each $T_s$ to $T_s^{-1}$.

\begin{theorem}\label{theo:KL}\textup{\cite{KL}}
		Let $(G,S)$ be a Coxeter system with  associated Hecke algebra $\mathcal{H}$.  Denote $l(g)=l^G(g)$ for $g\in G$. For each $g \in G$, there is a unique $C^G_g \in \mathcal{H}$ fixed by the involution on $\mathcal{H}$ and satisfying the condition
		\begin{align}
			C^G_g = q^{-\frac{\ell(g)}{2}}\, \sum_{x \preccurlyeq g}\,(-1)^{\ell(x)-\ell(g)}q^{\ell(x)-\ell(g)}\,\overline{P^G_{x,g}(q)}\,T_x\,,
		\end{align}
		where, for each $x,y \in G$,
		\begin{itemize}
			\item[(i)] $P^G_{x,y}(q) \in \mathbb{Z}[q]$,
			\item[(ii)] $P^G_{x,y}(q)=0$ if $x \not\preccurlyeq y$,
			\item[(iii)] $P^G_{x,x}(q)=1$, and
			\item[(iv)] $\deg\left(P^G_{x,y}(q)\right) \leq \frac{\ell(y)-\ell(x)-1}{2}$ if $x \prec y$.
		\end{itemize}
		The polynomials $P^G_{x,y}(q)$, where $x,y \in G$ for $x\preccurlyeq y$, are known as the \emph{Kazhdan--Lusztig (KL) polynomials} of $G$.
	\end{theorem}

\medskip

\subsection{Kazhdan--Lusztig  multiplicities for $\mathfrak{sp}(2n)$ and $\mathfrak{o}(2n)$}

In this subsection we recall the equalities which are commonly known as the Kazhdan-Lusztig conjecture. This will allow us to establish a connection between simple bounded infinite-dimensional highest weight  modules of $\mathfrak{sp}(2n)$ and simple finite-dimensional modules with half-integral highest weights of $\mathfrak{o}(\infty)$.

Let $\mathfrak{g}(2n)$ equal  $\mathfrak{o}(2n)$ or $\mathfrak{sp}(2n)$.
Let $\lambda$ be an integral dominant weight with half-integral entries for the Lie algebra $\mathfrak{o}(2n)$, and let $\lambda$    be a dominant integral weight, or a weight satisfying the conditions of  Lemma~\ref{prop:sp.bounded},  for the Lie algebra $\mathfrak{sp}(2n)$.   Denote       by     $M_w$  the Verma module with highest weight
$w\cdot \lambda$, and by $L_w$ the unique simple quotient  of $M_w$.  
If $V$ is a weight $\mathfrak{g}(2n)$-module then the formal character $ch(V)$  equals the formal sum $\sum_\mu m_\mu e^\mu$, where $m_\mu$ is the dimension of the weight space $V^{\mu}$ for $m_\mu\in\mathbb{Z}_{\geq 0}\cup\{\infty\}$.  We can define in a similar way the formal character of a weight module over any reductive finite-dimensional Lie algebra.

The following equality is often referred to as  \emph{Kazhdan--Lusztig conjecture}:

$$ch (L_w)= \sum_{y{\succcurlyeq}w} (-1)^{-\ell(w)-\ell(y)} P^W_{y w_0,w w_0}(1)ch(M_y),$$

$$ch (M_w)= \sum_{y{\succcurlyeq}w} P^W_{w, y}(1)ch(L_y),$$
where $W$ is the Weyl group of $\mathfrak{g}(2n)$, the Coxeter system is defined by the set of simple reflections, and  $w_0\in W$   is the element of maximal length.
The Kazhdan-Lusztig conjecture was proved independently  in \cite{BB} and \cite{BK}.



We put $\varepsilon(2n)='\rho_{\mathfrak{sp}(2n)}-\rho_{\mathfrak{o}(2n)}'=(1,1,\dots,1)$.

The subgroup of $W$ which preserves $\lambda$ is called \emph{the isotropy group of }$\lambda$.

\begin{theorem}\label{theorem5.3}\textup{\cite[page 267]{H}}
Let $\mathfrak{g}'$ and $\mathfrak{g}''$ be  finite-dimensional semisimple Lie algebras, with   respective Weyl groups $W'$ and $W''$. Fix weights $\lambda'$ for $\mathfrak{g}'$ and $\lambda''$ for $\mathfrak{g}''$, with  corresponding blocks $\mathcal{O}'_{\lambda'}$ and $\mathcal{O}''_{\lambda''}$ and  reflection subgroups $W'_{[\lambda']}$ and $W''_{[\lambda'']}$. If there is an isomorphism between these Weyl groups as Coxeter groups, which sends the isotropy group of $\lambda'$ to the isotropy group of $\lambda''$, then the category $\mathcal{O}'_{\lambda'}$ is equivalent to $\mathcal{O}''_{\lambda''}$, with $L(\lambda')$ sent to $L(\lambda'')$ and $M(\lambda')$ sent to $M(\lambda'')$.
\end{theorem}

Let $\lambda'$ be a  $\mathfrak{sp}(2n)$-weight with $\lambda'_i\in\mathbb{Z}+\frac{1}{2}$ satisfying the conditions of Lemma \ref{prop:sp.bounded} and let  $\lambda''=\lambda'+\varepsilon(2n)$. Put $W'=W_{\mathfrak{sp}(2n)}$ and $W''=W_{\mathfrak{o}(2n)}$. We consider $\lambda''$ as an integral dominant $\mathfrak{o}(2n)$-weight.
 Moreover, the reflection subgroup $W''_{[\lambda'']}$ equals  $W''$, and $W'_{[\lambda']}$ is a subgroup of index $2$ of the group $W'$consisting of permutations ${w\in S_{2n}}$ such that $w(-i)=-w(i)$, $1\leq i\leq n$, for which the number $\#\{i>0\mid w(i)<0\}$ is even (here we consider $W'$ as a subgroup of $S_{2n}$, see Subsection \ref{2.3}). The Coxeter groups  $W''_{[\lambda'']}$ and $W'_{[\lambda']}$ are isomorphic.
 These facts and the definition of Kazhdan--Lusztig polynomials imply the following

\begin{corollary}\label{5.3} After identifying $W'_{[\lambda']}$ with $ W''_{[\lambda'']}=W''$,
one has
 $P^{W'_{[\lambda']}}_{w,v}(q)=P^{W''_{[\lambda'']}}_{w,v}(q) $          for all elements $w, v \in W'' $.

\end{corollary}



Let $L(\lambda)$ be a simple finite-dimensional $\mathfrak{sp}(2n)$-module with highest weight $\lambda$. Then we set $\varpi_i:=\sum^{i}_{j=1}\varepsilon_j$, for $1\leq j\leq n$.  Then $\lambda=\sum^n_{i=1}v_i \varpi_i$  for  some $v_i\in\mathbb{Z}_{\geq 0}$, where $\varpi_i=\sum^i_{j=1}\varepsilon_i$. Also we denote by $T^j_{\lambda}$ the set  of all weights of the form $ \lambda-\sum^n_{i=1}d_i\varepsilon_i$, where $d_i$ are nonnegative integers, $\sum^n_{i=1}d_i$ is even,  $0\leq d_i\leq v_i$  for  $1\leq i\leq n-1$ and $0\leq d_n+\delta^j_1\leq 2v_n+1$, $\delta^j_1$ being the Kronecker delta.

\begin{lemma}\label{BL}\textup{\cite[Theorem 5.5]{BHL}} Denote $\nu_0=-\frac{1}{2} \varpi_n$ and $\nu_1=\varpi_{n-1}-\frac{3}{2}\varpi_n$.
The $\mathfrak{sp}(2n)$-module $L(\nu_j)\otimes~L(\lambda)$ is completely reducible with the decomposition
$$L(\nu_j)\otimes L(\lambda)\simeq \bigoplus_{\kappa\in T^j_{\lambda}}L(\nu_j+\kappa).$$
\end{lemma}

Lemmas \ref{prop:sp.bounded} and
  \ref{BL} imply  that each simple bounded highest weight infinite-dimensional  module could be constructed as a simple constituent of the above tensor product for some $\lambda$. Then it follows that the tensor product of $L(\nu_j)$ with any  simple bounded infinite-dimensional  highest weight  module is completely reducible.


\begin{corollary} Let $L(\lambda)$ be a  simple bounded  $\mathfrak{sp}(2n)$-module with highest weight $\lambda$. Then the module $L(\lambda)$ is completely reducible as a module over $\mathfrak{sp}(2n-2)$, where the embedding $\psi_{2n-2}\colon\mathfrak{sp}(2n-2)\to \mathfrak{sp}(2n)$ is described in  Subsection \ref{sub:infdimLA}.
\end{corollary}
\begin{proof}
Recall the definition of  Shale-Weil modules from Subsection \ref{SW}.
Set $\lambda_{0}=\frac{1}{2}\varpi_n$. Then there exists a simple  finite-dimensional module $L(\mu)$ with highest weight $\mu$, such that  $L(\lambda)$ is a simple constituent of  \break$ SW^+(2n) \otimes   L(\mu)$, because of Lemma \ref{BL} and the fact that the highest weight of $SW^+(2n)$ equals $-\lambda_0$. Decompose $L(\mu)=\bigoplus^{k}_{j=1} L(\bar\mu_j)$ for some simple finite-dimensional $\mathfrak{sp}(2n-2)$-modules $L(\bar\mu_{j})$.
Next, one can check directly that $L(\lambda_{0})$ as $\mathfrak{sp}(2n-2)$-module is isomorphic to a countable direct sum of copies of $SW^+(2n-2)\oplus SW^-(2n-2)$. 
Therefore the tensor product  $ SW^+(2n)\otimes L(\mu)$ is isomorphic to a direct sum of countably many copies of $SW^+(2n-2) \otimes L(\bar\mu_j))\oplus(  SW^-(2n-2)\otimes L(\bar\mu_j))$.

The tensor products  $       SW^+(2n-2)   \otimes   L(\bar\mu_j) $ and  $ SW^-(2n-2) \otimes L(\bar\mu_j) $ are completely reducible by   Lemma \ref{BL}. In this way, we see that $L(\mu)\otimes  SW^+(2n)$ is completely reducible over $\mathfrak{sp}(2n-2)$, and the same  holds for its submodule $L(\lambda)$.
\end{proof}

Next, we consider the Lie subalgebra  $\operatorname{Im}({\psi_{2n-2}})+\mathfrak{h}_{\mathfrak{sp}(2n)}$  of $\mathfrak{sp}(2n)$. 
 We set $\mathfrak{sp}^{\mathfrak{h}}(2n-2)=\operatorname{Im}({\psi_{2n-2}})+\mathfrak{h}(2n)$. Analogously, we define the Lie subalgebra $\mathfrak{o}^{\mathfrak{h}}(2n-2)$ of $\mathfrak{o}(2n)$.

 





\begin{prop}\label{big}
 Let $\lambda$ be a  weight  of $\mathfrak{sp}({2n})$ satisfying the conditions of  Lemma \ref{prop:sp.bounded}, and $\mu$ be a   weight  of $\mathfrak{sp}(2n-2)$ satisfying Lemma \ref{prop:sp.bounded}. Moreover, let $\lambda'=\lambda+\varepsilon(2n)$ be an integral dominant weight  of $\mathfrak{o}({2n})$ with half-integral marks, and $\mu'=\mu+\varepsilon(2n-2)$ be an integral dominant weight  of  $\mathfrak{o}(2n-2)$ with half-integral marks. Then $\operatorname{Hom}_{\mathfrak{sp}(2n-2)} ( L({\mu}) , L({\lambda}) ) \neq 0$ if and only if $\operatorname{Hom}_{\mathfrak{o}(2n-2)} ( L({\mu'}) , L({\lambda'}) ) \neq 0$.
\end{prop}
\begin{proof}

Consider the decomposition

$$L({\lambda'})=\bigoplus^k_{j=1} L({\lambda'^j}),$$
 over the Lie algebra $\mathfrak{o}^{\mathfrak{h}}(2n-2)$,
where each $L({\lambda'^j})$ is a  simple finite-dimensional representation of $\mathfrak{o}^{\mathfrak{h}}({2n-2})$ with highest weight $\lambda'^j$. Note, that  as an $\mathfrak{o}(2n-2)$-module $L({\lambda'^j})$ is isomorphic to the simple finite-dimensional $\mathfrak{o}(2n-2)$-module $L(\bar\lambda'^j)$ for \break  $\bar\lambda'^j = \sum^{n}_{i=2} \lambda^j_i\varepsilon_{i-1}$, where $\lambda'^j=\sum^{n}_{i=1}\lambda^j_i\varepsilon_i$.
This implies
\begin{equation}\label{12}ch(L({\lambda'}))=\sum^k_{j=1} ch( L({\lambda'^j})).\end{equation}

 We apply the Kazhdan--Lusztig conjecture for $w=\operatorname{id}\in W_{\mathfrak{o}_{2n-2}}$ to each $\mathfrak{o}(2n-2)$-module $L(\bar\lambda'^j)$.
Since each $\bar\lambda'^j$ is  an integral dominant weight, we get

$$ch (L(\bar\lambda'^j))= \sum_{y\in W_{\mathfrak{o}(2n-2)}} (-1)^{-\ell(y)} P^{W_{\mathfrak{o}(2n-2)}}_{y w_0, w_0}(1)ch(M({y\cdot\bar\lambda'^j})).$$


Now we denote by $ M^{\mathfrak{h}}(\tau)$ the $\mathfrak{o}^{\mathfrak{h}}(2n-2)$-module with highest weight $\tau$ such that as $\mathfrak{o}(2n-2)$-module $M^{\mathfrak{h}}(\tau)$ is isomorphic to $M(\bar \tau)$, where $\bar \tau=\sum^{n}_{i=2}\tau_i \varepsilon_{i-1}$. By $M^{\mathfrak{h}}(\alpha, \gamma)$ we denote the Verma module over $\mathfrak{o}^{\mathfrak{h}}(2n-2)$ with highest weight     $\alpha\varepsilon_1+\sum^n_{i=2}\gamma_{i-1}\varepsilon_i$,             where $\gamma=\sum^{n-1}_{i=1}\gamma_i\varepsilon_i$ is a weight of $\mathfrak{o}(2n-2)$ for $\alpha\in \mathbb{C}$. Clearly, $M(\mathfrak{\tau})=M^{\mathfrak{h}}(\tau_1,\bar\tau)$.
Therefore we can write the decomposition 
\begin{equation}\label{13}ch(L(\lambda'^j))=\sum_{{y\in W_{\mathfrak{o}(2n-2)}}} (-1)^{-\ell(y)} P^{W_{\mathfrak{o}(2n-2)}}_{y w_0, w_0}(1)ch(M^{\mathfrak{h}}({\lambda'^j_1, y\cdot\bar\lambda'^j}))\end{equation}
Combining formulas (\ref{12}) and (\ref{13}), we obtain 

\begin{equation}\label{14}ch(L({\lambda'}))=\sum^k_{j=1}\sum_{y\in W_{\mathfrak{o}(2n-2)}} (-1)^{-\ell(y)} P^{W_{\mathfrak{o}(2n-2)}}_{y w_0, w_0}(1)ch(M^{\mathfrak{h}}({\lambda'^j_1, y\cdot\bar\lambda'^j}))\end{equation}

On the other hand, we may apply the Kazhdan--Lusztig conjecture to the $\mathfrak{o}(2n)$-module $L(\lambda')$. This yields
\begin{equation}\label{15}ch (L(\lambda'))= \sum_{y\in W_{\mathfrak{o}(2n)}} (-1)^{-\ell(y)} P^{W_{\mathfrak{o}(2n)}}_{y w_0, w_0}(1)ch(M({y\cdot\lambda'})).\end{equation}

Next, for an $\mathfrak{o}(2n)$-weight $\tau$ we set $$\tau(a_1, a_2, \dots, a_{2n-2}):=\tau-\sum^{n-1}_{i=1}a_i(\varepsilon_{1}-\varepsilon_{i+1})-\sum^{2n-2}_{i=n}a_i(\varepsilon_{1}+\varepsilon_{i-n+2})$$
for $a_i\in \mathbb{Z}_{\geq 0}$.
We  decompose the character of each Verma module $M(y\cdot\lambda')$ over $\mathfrak{o}^{\mathfrak{h}}(2n-2)$:

\begin{equation}\label{16}
ch(M({y\cdot\lambda'}))=\sum^{\infty}_{a_1, a_2\dots, a_{n-1}=0}ch( M^{\mathfrak{h}}((y\cdot\lambda')(a_1,a_2,\dots,a_{n-1}))).
\end{equation}
This is a direct consequence of the definition of Verma module.
Note that, if we consider the restriction of (\ref{16}) to
each weight subspace of $ M({y\cdot\lambda'})$   as an equality of dimensions then the left-hand side is a positive integer, while the right-hand side   is a sum of positive integers. This means that each such restriction has only finitely many terms.

Combining formulas (\ref{15}) and (\ref{16}), we get 

\begin{equation}\begin{split}\label{17}&ch (L'(\lambda))=\\&= \sum_{y\in W_{\mathfrak{o}(2n)}}\sum^{\infty}_{a_1, a_2\dots, a_{n-1}=0} (-1)^{-\ell(y)} P^{W_{\mathfrak{o}(2n)}}_{y w_0, w_0}(1)ch( M^{\mathfrak{h}}((y\cdot\lambda')(a_1,a_2,\dots,a_{n-1}))).\end{split}
\end{equation}
From  equations (\ref{14}) and (\ref{17}) we obtain

$$
\sum^k_{j=1}\sum_{y\in W_{\mathfrak{o}(2n-2)}} (-1)^{-\ell(y)} P^{W_{\mathfrak{o}(2n-2)}}_{y w_0, w_0}(1)ch(M^{\mathfrak{h}}({\lambda'^j_1, y\cdot\bar\lambda'^j}))=$$
\begin{equation}\label{18}
 =\sum_{y\in W_{\mathfrak{o}(2n)}}\sum^{\infty}_{a_1, a_2\dots, a_{n-1}=0} (-1)^{-\ell(y)} P^{W_{\mathfrak{o}(2n)}}_{y w_0, w_0}(1)ch( M^{\mathfrak{h}}((y\cdot\lambda')(a_1,a_2,\dots,a_{n-1}))).
\end{equation}
We rewrite this equation as

$$
\sum^k_{j=1}\sum_{y\in W_{\mathfrak{o}(2n-2)}} (-1)^{-\ell(y)} P^{W_{\mathfrak{o}(2n-2)}}_{y w_0, w_0}(1)ch(M^{\mathfrak{h}}({\lambda'^j_1, y\cdot\bar\lambda'^j}))-$$
\begin{equation}\label{cve}
-\sum_{y\in W_{\mathfrak{o}(2n)}}\sum^{\infty}_{a_1, a_2\dots, a_{n-1}=0} (-1)^{-\ell(y)} P^{W_{\mathfrak{o}(2n)}}_{y w_0, w_0}(1)ch( M^{\mathfrak{h}}((y\cdot\lambda')(a_1,a_2,\dots,a_{n-1})))=0.
\end{equation}

Now we will show that the sum of coefficients in front of $ch( M^{\mathfrak{h}}(\gamma))$, for each $\gamma$ appearing in  formula (\ref{cve}), equals  $0$.
If a weight  $\gamma_0$             appears in  (\ref{cve})    and is  maximal (for the order defined by the fixed Borel subalgebra),  then the above claim is obvious.
Moreover,  finitely many maximal weights $\gamma_0$ exist because of  formulas (\ref{15}) and (\ref{16}). Therefore we can erase from formula (\ref{cve}) all terms of the form $ch(M(\gamma_0))$ for maximal $\gamma_0$.
For any fixed $\gamma$ we prove our claim after finitely many iterations.


In this way, we see that in equation (\ref{18}) the coefficients in front of every character of Verma module  at the left-hand and right-hand sides of the equation are equal.

Next, we apply the Kazhdan--Lusztig conjecture to the $\mathfrak{sp}(2n)$-module $L(\lambda)$. By Corollary \ref{5.3} we have  that the Kazhdan--Lusztig polynomials appearing in (\ref{19}) below are the same  as in formula (\ref{15}) for $L(\lambda')$:

\begin{equation}\label{19}
ch(L(\lambda))=\sum_{{y\in W_{[\lambda]}}} (-1)^{-\ell(y)} P^{W_{[\lambda]}}_{y w_0, w_0}(1)ch(M({y\cdot\lambda})).\end{equation}
 
For a $\mathfrak{sp}(2n)$-weight $\tau$ we set $$\tau(a_1, a_2, \dots, a_{n}):=\tau-2a_1\varepsilon_1-\sum^{n}_{i=2}a_i(\varepsilon_{1}-\varepsilon_{i}).$$  Also we denote by $ M^{\mathfrak{h}}(\tau)$ the $\mathfrak{sp}^{\mathfrak{h}}(2n-2)$-module with highest weight $\tau$. As $\mathfrak{sp}(2n-2)$-module, $M^{\mathfrak{h}}(\tau)$ is isomorphic to $M(\bar \tau)$ for $\bar \tau=\sum^{n}_{i=2}\tau_i \varepsilon_{i-1}$.
Now we  decompose the character of each Verma module $M(y\cdot\lambda)$ over $\mathfrak{sp}^{\mathfrak{h}}(2n-2)$ similarly to formula (\ref{16}):

\begin{equation}\label{20}
ch(M({y\cdot\lambda}))=\sum^{\infty}_{a_1, a_2,\dots, a_{n}=0}ch( M^{\mathfrak{h}}((y\cdot\lambda)(a_1,a_2,\dots,a_{n}))).
\end{equation}

Analogously to the case of $\mathfrak{o}(2n)$ we combine formulas (\ref{19}) and (\ref{20}):

\begin{equation}\begin{split}\label{21}
&ch(L(\lambda))=\\
&=\sum_{y\in W_{[\lambda]}} (-1)^{-\ell(y)} P^{W_{[\lambda]}}_{y w_0, w_0}(1)\sum^{\infty}_{a_1, a_2,\dots, a_{n}=0}ch( M^{\mathfrak{h}}((y\cdot\lambda)(a_1,a_2,\ldots,a_{n})))=\\
&=\sum^{\infty}_{a_1=0}\sum_{y\in W_{[\lambda]}} \sum^{\infty}_{a_2,\dots, a_{n}=0}(-1)^{-\ell(y)} P^{W_{[\lambda]}}_{y w_0, w_0}(1)ch( M^{\mathfrak{h}}((y\cdot\lambda-2a_1\varepsilon_1)(a_2,\dots,a_{n}))).
\end{split}
\end{equation}
From the fact that the  Kazhdan--Lusztig polynomials are the same  as for $\mathfrak{o}(2n)$, and    from our observation concerning the coefficients of each Verma module appearing in formula (\ref{18}), we can rewrite formula  
(\ref{21}) as

\begin{equation}\label{22}
ch(L(\lambda))= \sum^{\infty}_{a_1=0}\sum^k_{j=1}\sum_{y\in W_{[\bar\lambda^j]}} (-1)^{-\ell(y)} P^{W_{[\lambda]}}_{y w_0, w_0}(1)ch(M^{\mathfrak{h}}({\lambda^j_1-2a_1, y\cdot\bar\lambda^j})),
\end{equation}
where $\lambda^j=\lambda'^j-\varepsilon(2n)$ and $\bar\lambda^j = \bar\lambda'^j-\varepsilon(2n-2)$. Now we apply the Kazhdan--Lusztig conjecture  to each inner sum, and keeping in mind that the Kazhdan--Lusztig polynomials here and for $\mathfrak{o}(2n-2)$ are the  same, we obtain

\begin{equation}\label{2n}ch(L(\lambda))= \sum^{\infty}_{a_1=0}\sum^k_{j=1}ch(L(\lambda_j-2a_1\varepsilon_1)).\end{equation}


It is easy to check that  if $M$ is a weight module over $\mathfrak{sp}^{\mathfrak{h}(2n-2)}$, and is semisimple as an $\mathfrak{sp}(2n-2)$-module,
then $M$ is semisimple over $\mathfrak{sp}^{\mathfrak{h}}(2n-2)$. Therefore, $L(\lambda)$ is semisimple over $\mathfrak{sp}^{\mathfrak{h}}(2n-2)$.

 We will show that $L(\lambda_j-2a_1\varepsilon_1)$ is a simple constituent of $L(\lambda)$  for all $j$.
Denote by $A_1$ the set of all weights $\lambda_j-2a_1\varepsilon_1$ for $1\leq j \leq k$, $j,a_1\in\mathbb{Z}_{\geq 0}$. Note that this set is partially ordered $(\alpha\geq\beta \Longleftrightarrow$ $\alpha-\beta$ is a sum  of positive roots from $\Delta_{\mathfrak{sp}(2n)}$). Also this set is bounded from above. Moreover, it is clear that  if a weight $\zeta\in A_1$ is maximal then $\zeta=\lambda_p$ for some $p$, and therefore  there are finitely many maximal weights in $A_1$. Denote this set of maximal weights  by $T_1$. Consider now the set $A_1\backslash T_1$ and repeat the  procedure for this set. We obtain a set $T_2\subset A_1\backslash T_1$, and after $i-1$ steps --- a set $T_i$. Let
 $A_i:=A_1\backslash \bigcup_{j<i} T_j$. By definition, $T_i$ is the set of all maximal weights of $A_i$.

Next, we consider the subspace $M^{T_1}$ of $L(\lambda)$ spanned by all weight spaces with weights $\chi_p\in T_1$ and choose a weight basis $B$ of $M^{T_1}$. By each vector $b\in B$ we generate the submodule $L^b$ of $L(\lambda)$. Each such module is a highest weight module with  respective highest weight $\chi_p\in T_1$.  The semisimplicity of $L(\lambda)$ implies that $L^b$ is  simple, and we set $L^b=L(\chi_p)$.

Denote by $L^{T_1}$ the submodule  generated by $M^{T_1}$.  Each vector $u\in L^T_1$ can be obtained as $u=gv$ for some $v\in M^T_1$ and some $g\in U(\mathfrak{sp}^\mathfrak{h}(2n-2))$. Since $v$ is contained in the span of $B$, the vector $u$ lies in the sum of modules $L^b$. Therefore,  $L^{T_1}$   is isomorphic to the direct sum of the $\mathfrak{sp}^{\mathfrak{h}}(2n-2)$-modules $L(\chi_p)$. Consider a complement $L_1$ in $L(\lambda)$ to $L^{T_1}$. This is a submodule of $L(\lambda)$ with  character
$$ \sum^{\infty}_{a_1=0}\sum^k_{j=1}ch(L(\lambda_j-2a_1\varepsilon_1)) - \sum_{p\in T_1}ch(L(\chi_p))=\sum_{p\in A_2}ch(L(\chi_p)).$$ 
This character is well-defined because all weight spaces are finite-dimensional.

Denote by $L^{T_{e+1}}$ the submodule of $L_e$ generated by $M^{T_{e+1}}$, where $M^{T_{e+1}}$ is the subspace spanned by all weight spaces with weights $\chi_j\in T_{e+1}$. Let $L_{e+1}$ be a complement to $L^{T_{e+1}}$ in $L_{e}$.  One can show that the following formula holds for any $e\in \mathbb{Z}_{\geq 1}$ (by repeating   above decomposition of $L(\lambda)$ for all $L_j$, $j\leq e$):

$$L(\lambda)=\bigoplus^e_{i=1}L^{T_i}\oplus L_e.$$

Since the set $A_1$ is ordered as described above, each $\chi \in A_1$ is an element of $T_i$ for some $i$. Therefore, $L(\lambda_j-2a_1\varepsilon_1)$ is a simple constituent of $L^{T_i}$ for some $i$ as well as a simple constituent of $L(\lambda)$. Hence, 
$$F=\bigoplus^{\infty}_{a_1=0}\bigoplus^k_{j=1} L(\lambda_j-2a_1\varepsilon_1)$$ is a submodule of $L(\lambda)$. However, $F$ has the same character as $L(\lambda)$ therefore $L(\lambda)=F$.

Now, we note that the modules $L(\lambda_j-2a_1\varepsilon_1)$ and $L(\bar\lambda^j)$ are isomorphic as $\mathfrak{sp}(2n-2)$-modules for $\bar\lambda^j=\sum^n_{i=2}\lambda_i\varepsilon_{i-1}.$ We obtain  an isomorphism of $\mathfrak{sp}(2n-2)$-modules


$$L(\lambda)\simeq\bigoplus^k_{j=1}m_j L(\bar\lambda_j),$$
where $m_j$ equals the cardinality $\aleph_0$ for all $j$.

In this way we proved that
$$\operatorname{Hom}_{\mathfrak{o}(2n-2)}(L(\mu'),L(\lambda'))\neq0 \Longleftrightarrow \mu'=\bar\lambda'^j, \textup{for some }j \Longleftrightarrow$$
$$\Longleftrightarrow L(\mu) \textup{ is a simple constituent of } L(\lambda) \textup{ over }\mathfrak{sp}(2n-2) \Longleftrightarrow$$
$$\Longleftrightarrow \operatorname{Hom}_{\mathfrak{sp}(2n-2)}(L(\mu),L(\lambda))\neq0$$
Thus we proved the proposition.
\end{proof}

We have now shown that, for weights  $\lambda$ and $\mu$  respectively of   $\mathfrak{sp}(2n)$ and  $\mathfrak{sp}(2n-2)$,   satisfying Lemma \ref{prop:sp.bounded},  the following conditions are equivalent:
\begin{itemize}
		
		\item $ \operatorname{Hom}_{\mathfrak{sp}(2n-2)} ( L({\mu}) , L({\lambda}) ) \neq 0.$
		
		\item There  exists an $n$-tuple of half-integers $\nu = (\nu_1,\dots,\nu_n)$ which satisfy the inequalities
\begin{equation*}
\begin{split}			
		 \lambda_1+1 \geqslant &\nu_1 \geqslant \lambda_2+1 \geqslant\nu_2 \geqslant  \ldots \geqslant \lambda_{n-1}+1 \geqslant \nu_{n-1} \geqslant  |\lambda_n+1|\geqslant  \nu_{n}\geqslant \frac{1}{2} ,\\
		&\nu_1 \geqslant \mu_1+1 \geqslant\nu_2 \geqslant \mu_2+1 \geqslant   \ldots \geqslant \nu_{n-1} \geqslant |\mu_{n-1}+1| \geqslant \nu_{n} \geqslant \frac{1}{2}.  
\end{split}
\end{equation*}		
\end{itemize}

 In conclusion we would like to recall the following proposition. The degree of a weight module is given in Definition \ref{def:deg}.
\begin{prop}\textup{\cite[Theorem 12.2(ii)]{M}}
Let $L(\lambda)$ be a simple module of $\mathfrak{sp}(2n)$ with highest weight $\lambda$ satisfying Lemma \ref{prop:sp.bounded} and $  L(\lambda+\varepsilon)$ be the simple finite-dimensional module of $\mathfrak{o}(2n)$ with highest weight $\lambda+\varepsilon(2n)$.
 Then
$$\operatorname{deg}(L(\lambda))=\operatorname{dim}( L(\lambda+\varepsilon(2n)))/2^{n-1}.$$
\end{prop}

\subsection{Coherent local systems of bounded ideals: definition and classification}\label{sub:c.l.s.b.}

In this subsection we introduce the notion of  coherent local systems of\break  bounded ideals, which we abbreviate as c.l.s.b.  This is a generalization  of the notion of a c.l.s. Also we obtain a classification of irreducible c.l.s.b.  based on Zhilinskii's classification of irreducible of c.l.s.
As  above, $\mathfrak{g}(2n)$ denotes the Lie algebra $\mathfrak{o}(2n)$ or $\mathfrak{sp}(2n)$,  and $\mathfrak{g}(\infty)$ denotes the Lie algebra $\mathfrak{o}(\infty)$ or $\mathfrak{sp}(\infty)$.  Also we assume that splitting Cartan subalgebras and Borel subalgebras of $\mathfrak{g}(2n)$ are fixed  as  in Subsection \ref{2.3}.

Recall that  simple bounded highest weight $\mathfrak{o}(2n)$-modules  are finite di\-mensional, and that simple  bounded highest weight $\mathfrak{sp}(2n)$-modules are either finite dimensional,  or are modules with  highest weights satisfying Lemma \ref{prop:sp.bounded}. 

Let $J_n$ denote the set of   annihilators of    simple bounded highest weight $\mathfrak{g}(2n)$-modules (in fact, $J_n$ coincides with the set of annihilators of all bounded $\mathfrak{g}(2n)$-modules).
Next, let $R_n$ denote the set of isomorphism classes of simple bounded highest weight modules.
Note that the annihilator $A$ of a simple  bounded $\mathfrak{o}(2n)$-module determines the simple module annihilated by $A$   up to isomorphism.	
Also, the annihilator $A$ of a  finite-dimensional $\mathfrak{sp}(2n)$-module determines this module  up to isomorphism. 
However, one can show that, for a fixed splitting Borel subalgebra of $\mathfrak{sp}(2n)$, there are precisely two   simple  bounded infinite-dimensional  highest weight  $\mathfrak{sp}(2n)$-modules with a given annihilator $A\in J_n$.

\begin{defn}
	A \textup{coherent local system of bounded ideals} (further c.l.s.b.)  for
		$\mathfrak{g}(\infty)$ is a collection of sets
		
		$$\{\mathbb I_n\}_{n \in \mathbb{Z}_{\geq 2}} \subset \Pi_{n \in \mathbb{Z}_{\geq 2} }
		J_n $$
		such that $\mathbb I_m = \langle \mathbb I_n \rangle_m$ for $n>m$, where	$\langle \mathbb I_n
		\rangle_m$ denotes the set of all annihilators of simple $\mathfrak{g}({2m})$-constituents of the
		$\mathfrak{g}({2n})$-modules which are annihilated by at least one ideal from $\mathbb I_n$.
		
\end{defn}

	\begin{defn}
		A c.l.s.b. $\mathbb I$ is \textup{irreducible} if $\mathbb I \neq \mathbb I' \cup \mathbb I''$ with
		$\mathbb I'\not\subset \mathbb I''$ and $\mathbb I' \not\supset \mathbb I''$, $\mathbb{I}'$, $\mathbb{I}''$ being c.l.s.b.
		
	\end{defn}

Any annihilator $z \in J_n$ of a simple bounded $\mathfrak{g}(2n)$-module 
	corresponds to  one or two  classes in $R_n$ (i.e., either to an integral dominant weight  or to two   
	   half-integral  weights  of $\mathfrak{g}(2n)$). Denote by $\{\lambda(z)\}$ the set of weights $\lambda$ such that $L(\lambda)$ is annihilated by $z$. If the set $\{\lambda(z)\}$  contains only one weight, then we denote this weight by $\lambda(z)$ (otherwise $\#\{\lambda(z)\}=2$).

	 Let $z_1, z_2\in J_n$ and $\#\{\lambda(z_1)\}=1$.
	 We denote by $z_1z_2$ the  set if annihilators of the modules $L(\lambda(z_1)+\mu)$ for $\mu\in\{\lambda(z_2)\}$. For $S_1, S_2 \subset
	J_n$ we put 
	
	$$S_1S_2:=\{z\in J_n \mid z\in z_1z_2 \:\: \textup{for  some}\:\: z_1 \in S_1 \textup{ with} \#\{\lambda(z_1)\}=1\:\:
	\textup{and} \:\: z_2 \in  S_2\}.$$

	Let $Q'$ and $Q''$ be c.l.s.b.. We denote  by $Q'Q''$ the smallest c.l.s.b. such
	that $(Q')_n(Q'')_n \subset (Q'Q'')_n$. By definition, $Q'Q''$ is the \emph{product} of
	$Q'$ and $Q''$.

For any ideal $I \subset  U(\mathfrak{g}(\infty))$, define the collection of sets $Q(I)$ by putting 
	$$ Q(I)_n:=\{z\in J_n \mid I \cap U(\mathfrak{g}(2n)) \subset 
	z \}.$$


Recall that the natural $\mathfrak{g}(\infty)$-module $V$ is the direct limit  $\varinjlim V_n$, where $V_n$ is the natural
	$\mathfrak{g}(2n)$-module. Furthermore, $\Lambda^\bullet(M)$ and $S^
\bullet(M)$ denote respectively the symmetric and the exterior algebras of a module $M$, and $\Lambda^p(M)$ and $S^p(M)$ denote respectively the $p$th symmetric and the $p$th exterior powers $M$.

	For simplicity we will use the following notations:
given $p \in \mathbb{Z}_{\geq 0}$,
\begin{equation*}
\begin{split}
&E:=Q(\operatorname{Ann}(\Lambda^\bullet V)),\: L_p := Q(\operatorname{Ann}(\Lambda^p V)),\: L^\infty_p :=
	Q(\operatorname{Ann}(S^\bullet(V \otimes \mathbb{C}^p))),  \\
  &E^\infty:= \textup{\{annihilators of all  modules with integral highest weight\}}, \\
&R:= 
\begin{cases}\{\textup{annihilators of spinor  modules\} for }\mathfrak{o}(\infty)\\
\{ \textup{ or annihilators of Shale--Weil  modules for }\mathfrak{sp}(\infty)\} 
\end{cases} 
\end{split}
\end{equation*}	
 Then the following table  describes the set of basic c.l.s.b. for the  Lie algebras $\mathfrak{o}(\infty)$ and $\mathfrak{sp}(\infty)$.
\vspace{1em}

\begin{table}[h]
\begin{center}
\begin{tabular}{|c|c|}
\hline
Lie algebra & C.l.s.b. \\
\hline
$\mathfrak{o}(\infty)$ & $E$, $ L_p$, $L^\infty_p$, $ E^\infty$, $R$\\
\hline
$\mathfrak{sp}(\infty)$ & $E$, $ L_p$, $L^\infty_p$, $ E^\infty,$ $R$\\
\hline
\end{tabular}
\end{center}
\end{table}


	\begin{prop}\label{prop: c.l.s.b. class}
		Any irreducible c.l.s.b. can be expressed uniquely as a  product as follows:

		$$(L^\infty_v L_{v+1}^{x_{v+1}} L_{v+2}^{x_{v+2}} \dots L_{v+r}^{x_{v+r}}) E^m$$ or
		$$(L^\infty_v L_{v+1}^{x_{v+1}} L_{v+2}^{x_{v+2}} \dots L_{v+r}^{x_{v+r}}) E^m R$$ where
		$$r,v \in \mathbb{Z}_{\geq0},\:   x_i \in \mathbb{Z}_{\geq 0} \: \textup{for} \: v+1 \leq i \leq v+r. \:$$
		Here, for $v=0$, $L^\infty_v$ is assumed to be the empty c.l.s.b.
	\end{prop}
	\begin{proof}

A. Zhilinskii proved the analogous statement for c.l.s., see Theorem \ref{2.6}.
Hence, for $\mathfrak{o}(\infty)$ the proposition is obvious because in this case c.l.s. and c.l.s.b. are  the same objects. 

 In Proposition  \ref{big}  we showed
that  the $\mathfrak{sp}(2n-2)$-branching  of  a  bounded $\mathfrak{sp}(2n)$-module $L(\lambda)$
yields a set of highest weights which is obtained by translation via $\sum^n_{i=1}\varepsilon_i$
from the set of highest weights obtained from the  $\mathfrak{o}(2n-2)$-branching  of the   $\mathfrak{o}(2n)$-module $L(\lambda +\sum^n_{i-1}\varepsilon_i)$.

Thus, there is a one-to-one correspondence between the set of c.l.s. for $\mathfrak{o}(\infty)$ and the set of c.l.s.b. for $\mathfrak{sp}(\infty)$, that respects the relation of inclusion, and the product operation. This completes the proof.
\end{proof}
			

\medskip

\subsection{Classification of precoherent local systems \\of bounded ideals}\label{sub:p.l.s.b.}
In this subsection we show that the collection of sets $Q(I)$ corresponding to a primitive ideal $I\in U(\mathfrak{g}(\infty))$ is equivalent to a c.l.s.b.
The respective notion of  equivalence is defined  below.
Recall that every  primitive ideal of $U(\mathfrak{g}(\infty))$ is a weakly bounded ideal, as proved in Theorem \ref{asd}.

\begin{defn}
			A \textup{precoherent local system of bounded ideals} (further p.l.s.b.)  for
			$\mathfrak{g}(\infty)$ is a collection of sets
			
			$$\{\mathbb I_n\}_{n \in \mathbb{Z}_{\geq 2}} \subset \Pi_{n \in \mathbb{Z}_{\geq 2} }
			J_n $$
			such that $\mathbb I_m \supset \langle \mathbb I_n \rangle_m$ for any $n>m$, where	$\langle \mathbb I_n
			\rangle_m$ denotes the set of all annihilators of simple $\mathfrak{g}(2{m})$-constituents of the
			$\mathfrak{g}(2{n})$-module which are annihilated by at least one ideal from $\mathbb I_n$.
			\end{defn} 

The  definition of a weakly bounded ideal implies that $Q(I)$  is a p.l.s.b whenever $I$ is a weakly bounded ideal.

\begin{defn}
		 Two p.l.s.b. $\mathbb I$ and $\mathbb I'$ are \textup {equivalent} if there exists an integer~$n$ such that $\mathbb I_{n'} = \mathbb I'_{n'}$ for any $n' > n$.
		
\end{defn}

As we pointed out above, there is a one-to-one correspondence between the set of c.l.s. for $\mathfrak{o}(\infty)$ and the set of c.l.s.b. for $\mathfrak{sp}(\infty)$ which respects the relation of inclusion, and product operation. If we consider a c.l.s.b.~$Q$ as a purely combinatorial object (i.e., as a set of highest $\mathfrak{sp}(2n)$-weights for $n\geq 2$) then we can describe the corresponding c.l.s. as follows. If our c.l.s.b.  consists of integral $\mathfrak{sp}(2n)$-weights for $n\geq 2$ then
 the corresponding c.l.s. be  the same set of weights considered as $\mathfrak{o}(2n)$-weights.   If our c.l.s.b.  consists of  $\mathfrak{sp}(2n)$-weights with half-integral entries for $n\geq 2$ then 
 the corresponding c.l.s. is obtained by adding $\varepsilon(2n)$ to each $\mathfrak{sp}(2n)$-weight and considering new weights as  $\mathfrak{o}(2n)$-weights.
Hence, the proofs of the Lemmas \ref{plsnota}, \ref{equiv}, \ref{RandL}   are precisely the same as in Subsection \ref{prop} (we us all the notion from this subsection).

\begin{lemma}
		For any admissible $n$-tuple $\lambda$, the p.l.s. $Q^{\vee}(\lambda)$ is equivalent to the c.l.s.
		$$Q(\lambda):=\bigcup\limits_{1\leqslant k \leqslant \#\lambda} Q(k,\lambda_k), $$
		where the collection of sets $Q(k,a)$ for $k\in \mathbb{Z}_{\geq 0}$ and $a\in \mathbb{Z}/2$ is defined by putting
		$$Q(k,a)_m:=\{\mu \in A_{m}\mid\mu_k<a,\:if\:k\leq\#\mu\}.$$\textbf{}
\end{lemma} 
	\begin{lemma}
Let $\lambda$ and $\mu$ be  admissible tuples such that $\#\mu \geqslant 2\#\lambda$. Then the following conditions are equivalent:
	\begin{enumerate}
		\item $\mu \succ \lambda$,
		\item $\mu_k\geqslant\lambda_k$ for each $1\leqslant k \leqslant \#\lambda$.
	\end{enumerate}
	\end{lemma}

\begin{lemma}
		Let $\lambda$ and $\mu$ be admissible tuples such that $\#\mu-\#\lambda=1$, $\mu > \lambda$ and let  $k \in \mathbb{Z}$. Then 
		\begin{enumerate}
			\item $R(\mu,k)>R(\lambda,k)$,
				
			\item $L(\mu,k)>L(\lambda,k)$,
				\item $L(\mu,k)>R(\lambda,k)$ whenever one of the following conditions is satisfied
$$ \mu_{i+1} \geqslant k > \mu_{i+2},\eqno{(*~*)}$$
			$$\mu_{i+2} \geqslant k \geqslant \mu_{i+3}. \eqno{(***)} $$
for  $i$ such that $\lambda_{i} \geqslant k \geqslant \lambda_{i+1}$.
		\end{enumerate}
\end{lemma}	
	
The following proposition is a corollary of the above three lemmas.

\begin{prop}\label{theor:plsb clsb}
		For any p.l.s.b. $\mathbb I$ there exists a c.l.s.b. $\mathbb I'$ such that $\mathbb I$ and $\mathbb I'$ are equivalent.
	\end{prop}

\begin{theorem}
Let $I$ be a primitive ideal of $U(\mathfrak{g}(\infty))$. Then $Q(I)$ is equivalent to a c.l.s.b.
\end{theorem}
\begin{proof} Follows from Theorem \ref{asd} and Proposition \ref{theor:plsb clsb}.
\end{proof}

\newpage

\section{Classification of primitive ideals of $U(\mathfrak{o}(\infty))$ and $U(\mathfrak{sp}(\infty))$} \label{6}

Here we introduce a set  of modules such that each primitive ideal   of $U(\mathfrak{o}(\infty))$ and $U(\mathfrak{sp}(\infty))$ equals  the annihilator of a unique module from this set. Also, for a module from this set we formulate a criterion for integrability of its annihilator.

In this subsection $\mathfrak{g}(2n)$ denotes the Lie algebra $\mathfrak{o}(2n)$ or $\mathfrak{sp}(2n)$, and $\mathfrak{g}(\infty)$ denotes the Lie algebra $\mathfrak{o}(\infty)$ or $\mathfrak{sp}(\infty)$.


Let $Z$ be a Young diagram with row lengths 
$$l_1\geq l_2 \geq \dots \geq l_s >0.$$ 
Then we denote by $V_Z(n)$ the $\mathfrak{g}(2n)$-module  with highest weight
$$\underbrace{(l_1,~l_2,~\ldots,~l_s,~0,~0,~\ldots,~0)}_{n\text{ numbers}}.$$
 The $\mathfrak{g}(2n)$-module $V_Z(n)$ is isomorphic to a simple direct constituent of the tensor product
$$S^{l_1}(V(n))\otimes S^{l_2}(V(n))\otimes\dots\otimes S^{l_s}(V(n)),$$ 
where $S^{l_i}(V(n))$ denotes the $l_i$th symmetric power of the natural module $V(n)$. 
In this way, the $\mathfrak{g}(\infty)$-module $V_Z$ is defined as the direct limit $\varinjlim V_Z(n)$.
We denote by $R$ the $\mathfrak{o}(\infty)$-module which is equal to the direct limit $\varinjlim R(2n)$, where $R(2n)$ is the $\mathfrak{o}(2n)$-module with highest weight $(\frac{1}{2}\sum^{n}_1\varepsilon_i)$. Also we denote by $R$ the $\mathfrak{sp}(\infty)$-module which is equal to the direct limit $\varinjlim SW^+(2n)$.


\begin{prop}\textup{\cite{PP3}} Any nonzero prime integrable ideal $I\subsetneq \operatorname{U}(\mathfrak g(\infty))$ is the annihilator of a unique $\frak g(\infty)$-module of the form
$$\begin{array}{cc}
(\operatorname{S}^\bullet(V))^{\otimes x}\otimes ({ \Lambda}^\bullet(V))^{\otimes y}\otimes V_{Z}&\mbox{for $\mathfrak g(\infty)=\mathfrak{sp}(\infty)$},\\
\begin{array}{c}(\operatorname{S}^\bullet(V))^{\otimes x}\otimes ({ \Lambda}^\bullet(V))^{\otimes y}\otimes V_{Z}\mbox{  or}\\(\operatorname{S}^\bullet(V))^{\otimes x}\otimes ({ \Lambda}^\bullet(V))^{\otimes y}\otimes V_{Z}\otimes R\end{array}&\mbox{for $\mathfrak g(\infty)=\mathfrak o(\infty)$},\\
\end{array}$$
where $x,y\in\mathbb Z_{\ge0}$,           and $Z$ is an arbitrary Young diagram.

\end{prop}

\begin{defn}
		Let $S_1, S_2 \subset J_n$ and $s_1,s_2\subset R_n$ be respective sets of isomorphism classes of $\mathfrak{g}(2n)$-modules. Put

		$$ S_1 \otimes S_2 := \{z \in J_n\mid  z=\operatorname{Ann} r   \textup{ for some } r\in  R_n    \textup{ such that } $$
$$      \operatorname{Hom}_{\mathfrak{g}(2n)}(r, r_1 \otimes r_2 ) \neq 0 \:
		\textup{ for  some }  r_1 \in s_1 \: \textup{and} \: r_2 \in s_2 \}$$
		
		Let $\mathbb I'$ be a  c.l.s.b. of the form
		$$(L^\infty_v L_{v+1}^{x_{v+1}} L_{v+2}^{x_{v+2}} \dots L_{v+r}^{x_{v+r}}) E^m$$
and $\mathbb I''$ be a c.l.s.b. of the form
		$$(L^\infty_v L_{v+1}^{x_{v+1}} L_{v+2}^{x_{v+2}} \dots L_{v+r}^{x_{v+r}}) E^mR \textup{ or }(L^\infty_v L_{v+1}^{x_{v+1}} L_{v+2}^{x_{v+2}} \dots L_{v+r}^{x_{v+r}}) E^m$$
Then the \textup{tensor product} $\mathbb{I}'\otimes \mathbb{I}''$ of this two c.l.s.b.  is the collection of sets defined by
		$(\mathbb I'\otimes \mathbb I'')_i = \mathbb I'_i \otimes \mathbb I''_i $.
	\end{defn}


\begin{lemma}\label{lem: class}
Let $Q$ be a c.l.s.b of Lie algebra $\mathfrak{sp}(\infty)$ which can be expressed as $$Q=(L^\infty_1)^{\otimes v} \otimes (L_{1}^{x_{v+1}} L_{2}^{x_{v+2}} \dots
	L_{r}^{x_{v+r}}) E^m .$$ 
Then tensor product of c.l.s.b.
$Q\otimes R$ is a c.l.s.b.
\end{lemma}
\begin{proof}
One can show that the statement follows from Lemma \ref{BL}.
\end{proof}

One can easily deduce that

$$(L^\infty_v L_{v+1}^{x_{v+1}} L_{v+2}^{x_v+2} \dots L_{v+r}^{x_{v+r}}) E^m =
	(L^\infty_1)^{\otimes v} \otimes (L_{1}^{x_{v+1}} L_{2}^{x_{v+2}} \dots
	L_{r}^{x_{v+r}}) E^m,$$  
	$$(L^\infty_v L_{v+1}^{x_{v+1}} L_{v+2}^{x_{v+2}} \dots L_{v+r}^{x_{v+r}}) E^m  R =
	(L^\infty_1)^{\otimes v} \otimes (L_{1}^{x_{v+1}} L_{2}^{x_{v+2}} \dots
	L_{r}^{x_{v+r}}) E^m \otimes R.$$

For every c.l.s.b.  $Q=\{Q_n\}_{n \in \mathbb{Z}_{\geq 1}}$ we can define the ideal  
	$$I(Q):=\bigcup_m(\bigcap_{z\in Q_m}
	z)\subset U(\mathfrak{g(\infty)}).$$
	  We say that $I(Q)$ is the \emph{globalization} of $Q$.

Let $Z$ be a Young diagram with row lengths 
$$l_1\geq l_2 \geq \dots \geq l_s >0.$$ 
For each positive number $x$ we denote by $\{x\}$ the fractional part of $x$.
Let  $V(x,y,Z)(2n)$, for $x\in\mathbb{Z}_{\geq0}$, $y\in \frac{1}{2}\mathbb{Z}_{\geq 0}$, denote the simple $\mathfrak{g}(2n)$-module
with highest weight $$\sum^x_{i=1}(n+\{y\})\varepsilon_i + \sum^{x+s}_{i=x+1}(l_{i-x}+y)\varepsilon_i + \sum^{n}_{i=x+s+1}y\varepsilon_i $$
for $\mathfrak{g}(2n)=\mathfrak{o}(2n),$
$$\sum^x_{i=1}(n-\{y\})\varepsilon_i + \sum^{x+s}_{i=x+1}(l_{i-x}+y-2\{y\})\varepsilon_i + \sum^{n}_{i=x+s+1}(y-2\{y\})\varepsilon_i $$
for $\mathfrak{g}(2n)=\mathfrak{sp}(2n),$
where  $n\geq x+s$. It is clear that we can embed $V(x,y,Z)(2n)\hookrightarrow V(x,y,Z)(2n+2)$ as $\mathfrak{g}(2n)$ submodule. Now, we can  define $\mathfrak{g}(\infty)$-module $V(x,y,Z)$ as the direct limit $\varinjlim V(x,y,Z)(2n)$. Let $Q$ be a c.l.s.b of the form

\begin{itemize}
\item$(L^\infty_x L_{x+1}^{l_1} L_{x+2}^{l_2} \dots L_{x+s}^{l_s}) E^y $,
then $I(Q)=\operatorname{Ann}(V(x,y,Z))\subset U(\mathfrak{g}(\infty))$ for $y\in\mathbb{Z}_{\geq0}$,
\item $(L^\infty_x L_{x+1}^{l_1} L_{x+2}^{l_2} \dots L_{x+s}^{l_s}) E^{y-\frac{1}{2}}R$,
then $I(Q)=\operatorname{Ann}(V(x,y,Z))\subset U(\mathfrak{g}(\infty))$    for $y\in\mathbb{Z}_{\geq0}+\frac{1}{2}$ .
\end{itemize}

The following is the main result of this section.
\begin{theorem} \label{tyu}
\begin{enumerate}
\item[\textup{a)}]Any nonzero primitive ideal $I\subsetneq \operatorname{U}(\mathfrak{g}(\infty))$ is the annihilator $I(x,y,Z)$ of an  $\frak{sp}(\infty)$-module of the form

$$(\operatorname{ S}^\bullet(V))^{\otimes x} \otimes({ \Lambda}^\bullet(V))^{\otimes y}\otimes V_{Z} \textup{ for }y\in\mathbb{Z}_{\geq0},$$
$$(\operatorname{ S}^\bullet(V))^{\otimes x} \otimes({ \Lambda}^\bullet(V))^{\otimes y}\otimes V_{Z}\otimes R  \textup{ for }y\in\mathbb{Z}_{\geq0}+\frac{1}{2},$$
where $x\in\mathbb Z_{\ge0},$ and $Z$ is an arbitrary Young diagram. Moreover, $I(x_1,y_1,Z_1)=I(x_2,y_2,Z_2)$ if and only  if $x_1=x_2$, $y_1=y_2$ and $Z_1=Z_2$.
\item[\textup{b)}] The ideal $I(x,y,Z)$ is integrable if and only if $y\in \mathbb{Z}_{\geq0}$.

\end{enumerate}
\end{theorem}
Theorem \ref{tyu} follows form Proposition \ref{dsa}, Propositon \ref{TheBigOne} and  Corollary \ref{TheEnd} which  we prove below.
\begin{prop}\label{dsa} Every  primitive ideal $I\subset U(\mathfrak{sp}(\infty))$ is of the form $I(x,y,Z)$ for some $x\in \mathbb{Z}_{\geq0}$, some $y\in \mathbb{Z}_{\geq0}/2$, and some   Young diagram $Z$.
\end{prop}
\begin{proof} We claim that  $I=I(Q(I))$. Indeed, the p.l.s.b. $Q(I)$ consists of all bounded ideals $z\supset I\cap U(\mathfrak{sp}(2n))$. Since $I$ is weakly bounded, i.e., is such that every intersection $I\cap U(\mathfrak{sp}(2n))$ is an intersection of  bounded ideals, we have $\bigcap_{z\in Q(I)_n} z= I\cap U(\mathfrak{sp}(2n))$. Thus, $$I(Q(I))=\bigcup_n I\cap U(\mathfrak{sp}(2n))=I.$$ 

Recall, that for every p.l.s.b. $Q(I)$ we can find an equivalent c.l.s.b. $Q(I)'$. Therefore, there exists a bijection $\phi$ between the set of primitive ideals $I\subset U(\mathfrak{sp}(2n))$ and the set of  c.l.s.b. $Q$, such that if $\phi(I)=Q$ then $I(Q)=I$. Lemma \ref{lem: class} and the above classification of  c.l.s.b. in Proposition \ref{prop: c.l.s.b. class} complete the proof.

\end{proof}

For our next proposition we need some preliminary considerations.

Let  $Z(U(\mathfrak{sp}(2n)))$ be the center of $U(\mathfrak{sp}(2n))$, i.e., the set of all elements $g\in U(\mathfrak{sp}(2n))$ such that $gu=ug$ for each $u\in U(\mathfrak{sp}(2n))$.

\begin{defn} Let $L(\lambda)$ be a simple $\mathfrak{sp}(2n)$-module with highest weight $\lambda$,  a wnd let $v^+$ be a highest weight vector of $L(\lambda)$. For $z\in Z(U(\mathfrak{sp}(2n)))$ we have $z\cdot v^+=\chi(\lambda)(z)\cdot v^+$ for $\chi(\lambda)(z)\in \mathbb{C}$. Since $z$ is a central element, it acts as $\chi(\lambda)(z)$ on $L(\lambda)$. The map
$$\chi(\lambda)\colon Z(U(\mathfrak{sp}(2n)))\mapsto \mathbb{C}$$
is called the central character of weight $\lambda$.
\end{defn}


For a Young diagram $Z$,  we denote the length of the $i$th row of $Z$ by $Z_i$. Let $I(x,y,Z)$ be a primitive ideal of $U(\mathfrak{sp}(\infty))$.  One can check that the work of Zhiliskii \cite{Zh3}, together with Lemma \ref{prop:sp.bounded} and Proposition \ref{big} imply the following fact:      $Q(I(x,y,Z))_n$ consists of all weights $\lambda=\sum^n_{i=1}\lambda_i\varepsilon_i$ such that  $\lambda_i-\lambda_j\in \mathbb{Z}_{\geq_0}$  for $ i>j$, $\lambda_{n-1}+\lambda_{n}\geq -2$ for $\lambda_{n}\in \mathbb{Z}+\frac{1}{2}$, $\lambda_n\in \mathbb{Z}_{\geq 0}$ for $\lambda_{n}\in \mathbb{Z}$, and $y+Z_i- \lambda_{x+i}\in\mathbb{Z}_{\geq0}$ for all $i\in \mathbb{Z}_{\geq 1}$. We call the c.l.s.b. $Q(I)$ \emph{integral} whenever the entries of any $\lambda\in Q(I)$ are integers. If  the entries of any $\lambda\in Q(I)$ are half-integers, then  $Q(I)$ is \emph{half-integral}.

\begin{prop}\label{TheBigOne}Two primitive ideals $I_1=I(x_1,y_1,Z_1), I_2=I(x_2,y_2,Z_2)$ of $U(\mathfrak{sp}(\infty))$ are equal if and only if $x_1=x_2$, $y_1=y_2$ and $Z_1=Z_2$. 
\end{prop}
\begin{proof} Assume that the triplet $(x_1,y_1,Z_1)$ does not equal  $(x_2,y_2, Z_2)$. We will consider $Q(I)$ as a set of weights as in Subsections \ref{sub:c.l.s.b.} and \ref{sub:p.l.s.b.}. Then there exists such $n$ that $Q(I_1)_n\neq Q(I_2)_n$.  This is equivalent to the following. We may assume  without loss of generality that $x_1\geq x_2$ with $n>x_2$, and that there exist a weight $\lambda'\in Q_n(I_1)$ and an integer $x_2<k\leq n$, such that $\lambda'_i\geq\lambda_i$ for $x_2<i<k$ and $\lambda'_k>\lambda_k$ for any $\lambda\in Q(I_2)_n$.  Since the choice of $\lambda'$ is not unique, we choose one with minimal possible $k$. Also, without loss of generality, we assume that $k=n$  because of the fact that a $\mathfrak{sp}(2m)$-module $L(\beta)$ always has an $\mathfrak{sp}(2l)$-submodule isomorphic to $L(\gamma)$, for $m>l$ and $\gamma=\sum^l_{i=1}\beta_i\varepsilon_i$.

 Suppose that $I_1=I_2$.  Obviously, then $I_1\cap U(\mathfrak{sp}(2n))= I_2\cap U(\mathfrak{sp}(2n))$. This implies  
\begin{equation}\label{for:ann}
\bigcap_{\lambda\in Q(I_1)_n}\operatorname{Ann} L(\lambda)=\bigcap_{\lambda\in Q(I_2)_n}\operatorname{Ann} L(\lambda).
\end{equation}

Next, we consider the intersections $I_1\cap Z(U(\mathfrak{sp}(2n)))$ and $I_2\cap Z(U(\mathfrak{sp}(2n)))$. These intersections are equal because $I_1=I_2$.
By Harish-Chandra's Theorem, $Z(U(\mathfrak{sp}(2n)))$ is isomorphic to the polynomial algebra $\mathbb{C}[a_1,a_2,\ldots,a_n]$ in $n$ variables. A maximal ideal $F$ of $Z(U(\mathfrak{sp}(2n)))$ has the form 
$$
F=F(a)=\{f\in Z(U(\mathfrak{sp}(2n)))\mid f(a)=0\},
$$
for some $a\in \mathbb{C}^n$. Moreover, it is well known that  the intersection of the annihilator of a simple highest weight $\mathfrak{sp}(2n)$-module $L(\lambda)$ with
$Z(U(\mathfrak{sp}(2n)))$ equals to the maximal ideal $F(\chi(\lambda))$ of $Z(U(\mathfrak{sp}(2n)))$. This fact and formula (\ref{for:ann}) imply
$$\bigcap_{\lambda\in Q(I_1)_n}F(\chi(\lambda))=\bigcap_{\lambda\in Q(I_2)_n}F(\chi(\lambda)).$$ The latter  holds if and only if the Zariski closures of the sets $\{\chi(\lambda)\mid\lambda\in Q(I_1)_n\}$ and $\{\chi(\lambda)\mid \lambda\in Q(I_2)_n\}$ coincide. 

 We now  choose   the variables       $a_1, a_2, \ldots, a_n$  to equal the independent Casimir elements $G_s$ for $1\leq s\leq n$, which act on the module $L(\lambda)$ by the constants
$$g_s(\lambda)=(-1)^s\sum_{1\leq i_1 \leq i_2 \leq \ldots \leq i_s \leq n}\Pi^s_{j=1}((\lambda_{i_j}+n-i_j+1)^2-(i_j-j+1)^2).$$
The expressions $g_s(\lambda)$ are  symmetric polynomials in the variables\break$(\lambda_{i_j}+n-i_j+1)^2$, see \cite[Theorem 3.8]{IMR}. Therefore, $$\chi(\lambda)=(g_1(\lambda),g_2(\lambda),\ldots,g_n(\lambda)).$$

Next, we define the following equivalence relation on the set $Q(I_2)_n$:
$$ \lambda\sim_{x_2}\mu\Longleftrightarrow \lambda_i=\mu_i \textup{ for } i>x_2,$$ where $\lambda,\mu\in Q(I_2)_n$. In this way, we obtain finitely many equivalence classes. The class of  $\lambda$ is denoted  by $K[\lambda_{x_2+1},\lambda_{x_2+2},\ldots, \lambda_{n}].$

  Given $u,m,j \in\mathbb{Z}_{\geq 0}$, $m>x_2$,                   we set

\begin{equation}\label{form}d_{m,j}=((\lambda_{m}+n-m+1)^2-(m-j+1)^2),\end{equation}
 $$b_{u,j}=(u-j+1)^2.$$
 Furthermore, we 
consider the subset $S[\lambda_{x_2+1},\lambda_{x_2+2},\ldots, \lambda_{n}]$ of $\mathbb{C}^n$ defined as the set of all points $(a_1,a_2,\ldots, a_n)$ such that
$$a_s=f_s=(-1)^s\sum_{1\leq i_1 \leq i_2 \leq \ldots \leq i_s \leq n}(\prod_{i_j\leq x_2}(y_{i_j}-b_{i_j,j})\prod_{i_j>x_2}d_{i_j,j})$$
for all $1\leq s\leq n$, and for all $(y_1,y_2,\ldots,y_{x_2})\in \mathbb{C}^{x_2}$. It is clear that if  $\lambda\in K[\lambda_{x_2+1},\lambda_{x_2+2},\ldots, \lambda_{n}]$, then $\chi(\lambda)$ belongs to $S[\lambda_{x_2+1},\lambda_{x_2+2},\ldots, \lambda_{n}]$. 

Now, we will show that $S[\lambda_{x_2+1},\lambda_{x_2+2},\ldots, \lambda_{n}]$ is an affine subspace of $\mathbb{C}^n$. The fact that the polynomials $g_s(\lambda)$ are symmetric implies  that the polynomials $f_s$ are symmetric in the variables $y_1,y_2,\ldots, y_{x_2}$. Note that the  degree of  $f_s$  for $s\leq x_2$ is equal to $s$. In addition, the polynomial $f_s$ is linear in each $y_j$,  $1\leq j \leq x_2$. Thus, $f_s$ for $1\leq s\leq x_2$ are linearly independent.

 Hence, the polynomials $f_s$ for $1\leq s \leq n$ are linear combinations of the polynomials $f_t$ for $1\leq t\leq x_2$. Let  $s$ satisfy $x_2<s\leq n$. Denote by $l_s(a_1,a_2\ldots, a_{x_2})$  the affine-linear function such that $f_s=l_s(f_1,f_2\ldots, f_{x_2})$.
Then  $S[\lambda_{x_2+1},\lambda_{x_2+2},\ldots, \lambda_{n}]$ coincides with the affine subspace of $\mathbb{C}^n$ defined by the system of equations
\begin{equation}\label{27}\left\{ \begin{array}{l}
 a_{x_2+1}=l_{x_2+1}(a_{1},a_{2},\ldots,a_{x_2}), \\
 a_{x_2+2}=l_{x_2+2}(a_{1},a_{2},\ldots,a_{x_2}) ,\\
 \ldots,\\
a_{n}=l_{n}(a_{1},a_{2},\ldots,a_{x_2}).
  \end{array} \right.\end{equation}

Next, we will show that $S[\lambda_{x_2+1},\lambda_{x_2+2},\ldots, \lambda_{n}]$ is the Zariski closure of  $K[\lambda_{x_2+1},\lambda_{x_2+2},\ldots, \lambda_{n}]$. Indeed, assume that there is an equation\break $f(a_1,a_2,\dots, a_n)=0$ such that $f(\chi_{\lambda})=0$ for every $\lambda\in K[\lambda_{x_2+1},\lambda_{x_2+2},\ldots, \lambda_{n}]$. Then we consider the function $$\bar f(a_1,a_2\ldots, a_{x_2}) = (a_1, \dots,a_{x_2}, l_{x_2+1}(a_1, \ldots, a_{x_2}), \ldots,  l_{n}(a_1, \ldots, a_{x_2})),$$ and note that $\bar f(\chi(\lambda)_1, \chi(\lambda)_2,\ldots,\chi(\lambda)_{x_2} )=0$ for all $\lambda\in K[\lambda_{x_2+1},\lambda_{x_2+2},\ldots, \lambda_{n}]$. However,the Combinatorial Nullstellensatz \cite{A} claims that such polynomial is equal to $0$. Thus $S[\lambda_{x_2+1},\lambda_{x_2+2},\ldots, \lambda_{n}]$ is the Zariski closure of\break $K[\lambda_{x_2+1},\lambda_{x_2+2},\ldots, \lambda_{n}]$.

Now, we consider the functions
$$
u_t=\sum_{1\leq i_1 \leq i_2 \leq \ldots \leq i_s \leq x_2}(\prod^s_{j=1}(y_{i_j}-b_{i_j,j})
$$
for $1\leq t\leq x_2$.
Each $u_t$ can be expressed as linear combination of the functions $f_1, f_2,\ldots f_t$ with a nonzero coefficient of $f_t$. Hence, the functions $u_t$ for $1\leq t\leq x_2$ are linearly independent. Note that $d_{i_j,j}\neq 0$  for $\lambda_i\notin\mathbb{Z}$. This and formula (\ref{27}) implies that the functions $f_{t}$ for $n-x_2<t\leq n$ are linearly independent whenever $Q(I_2)$ is half-integral.

Consider the point $h^0$ of  $S[\lambda_{x_2+1},\lambda_{x_2+2},\ldots, \lambda_{n}]$ defined by the parameters $y_i=b_{i,1}$. One can check that $h^0=(\sum_{i>x_2}d_{i,1},  c_2, c_3,\ldots, c_{n-x_2},0,\ldots,0 )$ for some $c_1,c_2,\ldots,c_{n-x_2}\in \mathbb{C}$. The fact, that    $f_{s}$ for $n-x_2<s\leq n$ are linearly independent  whenever $Q(I_2)$ is half-integral,  implies that a point $h^1\in S[\lambda_{x_2+1},\lambda_{x_2+2},\ldots, \lambda_{n}]$ with $h^1_{i}=0$ for $i>n-x_2$ must be equal to $h^0$ whenever  $Q(I_2)$ is half-integral.  Denote ${d_{i,1}}$ by $d_i(\lambda)$ and $h^0$ by $h^0(\lambda)$. Clearly, $d_i(\lambda^1)>d_i(\lambda^2 )$ if $\lambda^1_i>\lambda^2_i$.

Recall that the proposition we are proving has already been proved by Zhilinskii in the  work  \cite{Zh3} under the assumption that the c.l.s.b. $Q(I_1)$ and $Q(I_2)$ are integral, i.e. $Q(I_1)$ and $Q(I_2)$ are c.l.s. Hence it remains to  prove our proposition for $I_1$ and $I_2$  such that at least one of the c.l.s.b. $Q(I_1)$ and $Q(I_2)$ is half-integral.

Denote by $Z(S)$ the Zariski closure   of  $S$ in $\mathbb{C}^n$ . Let $V$ be a set of $\mathfrak{sp}(2n)$-weights. We denote the respective set of central characters $\chi_{\lambda}$ for $\lambda\in V$ by $H(V)$.

In what follows we denote by  $\sim_{x_2}$  the  equivalence relation  on the set $Q(I_1)_n$, constructed exactly as   the  equivalence relation on $Q(I_2)_n$ denoted above by $\sim_{x_2}$. We abbreviate $$K[\lambda]:=K[\lambda_{n-x_2+1}, \lambda_{n-x_2+2},\ldots, \lambda_{n}]$$ and $$S[\lambda]:=S[\lambda_{n-x_2+1}, \lambda_{n-x_2+2},\ldots, \lambda_{n}].$$ 
Finally, we prove that $I_1\neq I_2$ by considering the following cases:
\begin{itemize}
\item $Q(I_1)$  is integral, $Q(I_2)$ is half-integral, and $x_1>x_2$. It is clear that $n=x_2+1$. This implies that $Z(H(Q(I_1)_n))= 0$, while the Zariski closure  $Z(H(Q(I_2)_n))$ is finite union of proper affine subspaces, and hence $I_1\neq I_2$. 

\item  $Q(I_1)$  is half-integral, $Q(I_2)$ is integral or half-integral, and $x_1>x_2$. We obtain that $n=x_2+1$ and that $Z(H(Q(I_1)_n))= 0$, $Z(H(Q(I_2)_n))\neq 0$ similarly to the case when $Q(I_1)$  is integral, $Q(I_2)$ is half-integral, and $x_1>x_2$.

\item  $Q(I_1)$  is half-integral, $Q(I_2)$ is integral or half-integral, and $x_1=x_2$. Then $Z(H(Q(I_1)_n))=\bigcap^y_{i=1}S(\lambda^i)$ for some $\lambda^i\in Q(I_1)_n$, and  denote $S_i:=S(\lambda^i)$. Also we obtain $Z(H(Q(I_2)_n))=\bigcap^u_{j=1}S(\lambda^j)$ for some $\lambda^j\in Q(I_2)_n$, and put $S'_j=S(\lambda^j)$. Therefore, $Z(H(Q(I_1)_n))=Z(H(Q(I_2)_n))$ if and only if the sets $\{S_1,S_2,\ldots, S_y \}$ and $\{S'_1,S'_2,\ldots, S'_u\}$ are equal. We show that  $S(\lambda')\neq S'_j$ for any $1\leq j\leq u$. There are two possibilities for $S'_j$: first, the coordinates with indices $n-x_2+1, n-x_2+2,\ldots, n$ of points in $S'$ are not linearly independent; second, the coordinates with indices $n-x_2+1, n-x_2+2,\ldots, n$ of points in $S'$ are  linearly independent. In the first case, we have $S(\lambda')\neq S_j$ because the last $x_2$ coordinates of points in $S(\lambda')$ are linearly independent. In the second case,  we note that $S(\lambda')$  contains the point $h^0(\lambda')=(\sum_{i>x_2}d_i(\lambda'),c'_2,c'_3,\ldots,c'_{n-x_2},0,\ldots,0)$. Since $\lambda'_{i}\geq\lambda_{i}$ for  $x_2 \leq i\leq n$, we have $\sum_{i>x_2}d_i(\lambda')>\sum_{i>x_2}d_i(\lambda)$ where $\lambda\in Q(I_2)_n$. Therefore  $h_0(\lambda^j)\neq h_0(\lambda')$ for $\lambda^j\in Q(I_2)_n$. As $h_0(\lambda^j)$ is the unique point with $\lambda^j_s=0$ for $n-x_2+1\leq s\leq n$, we conclude that $S(\lambda)$ does not coincide with $S'_j$ for all $j$, and hence $I_1\neq I_2$.

\item $Q(I_1)$  is integral, $Q(I_2)$ is half-integral, and $x_1=x_2$. Then it is  clear that $k=n=x_2+1$.   Note that,     if $d_{n,n}(\lambda')=0$ then $a_n=0$ for each $a\in S(\lambda')$.
On the other hand, formula (\ref{form}) implies $d_{n,n}(\lambda')$ is equal to zero if and only if $\lambda'_n=0$. This implies  $Q(I_1)=S(\lambda')$. Thus the Zariski closures   $Z(H(Q(I_2)_n))$ and $Z(H(Q(I_1)_n))$ do not coincide because there exists $a\in Z(H(Q(I_2)))$ such that $a_n\neq0$. For $d_{n,n}(\lambda')\neq0$ proof is the same as for the case when $Q(I_1)$  is half-integral, $Q(I_2)$ is integral or half-integral, and $x_1=x_2$.

 \end{itemize}
\end{proof}

\begin{corollary}\label{TheEnd}
 Every primitive ideal $I=I(x,y,Z)\subset U(\mathfrak{sp}(\infty))$ with $y\in\mathbb{Z}+\frac{1}{2}$,  is nonintegrable.
\end{corollary}
\begin{proof} Follows from the  Proposition \ref{TheBigOne} and  the classification of integrable ideals given by Zhilinskii in \cite{Zh1}, \cite{Zh2} and \cite{Zh3}.
\end{proof}

We conclude this thesis by the remark that we have now established that the primitive ideals of $U(\mathfrak{o}(\infty))$ and $U(\mathfrak{sp}(\infty))$ are described by the same triplets $(x,y,Z)$. This follows from a direct comparison of Proposition $4.8$ in \cite{PP3} and  Theorem \ref{tyu} above. The only difference between the two cases that the primitive ideals $I(x,y,Z)$ with $y\in \mathbb{Z}_{\geq0}+\frac{1}{2}$ are integrable in the case  of  $U(\mathfrak{o}(\infty))$, and nonintegrable in the case of $U(\mathfrak{sp}(\infty))$. This remark is a strong hint for the conjecture that the isomorphism of the lattices of ideals in $U(\mathfrak{o}(\infty))$ and $U(\mathfrak{sp}(\infty))$ constructed in \cite{PP3} preserves primitivity.



\newpage\begin{thebibliography}{XXXX}
\bibitem[A]{A} N. Alon, \emph{Combinatorial Nullstellensatz}, Combinatorics, Probability
and Computing \textbf{8} (1999),  7–29.


\bibitem[BB]{BB} A. Beilinson, J. Bernstein, \emph{Localisation de g-modules}, C. R. Acad. Sci. \textbf{292} (1981), 15–18.

\bibitem[BK]{BK} J. Brylinski, M. Kashiwarai, \emph{Kazhdan--Lusztig conjecture and holonomic systems}, Invent. Math. \textbf{64} (1981), 387–410.


\bibitem[BV]{BV} D. Barbasch, D. Vogan,\emph{ Primitive ideals and orbital integrals in complex classical groups}, Math. Ann. \textbf{259} (1982), 153–199.

\bibitem[BHL]{BHL} D. Britten, J. Hooper, F. Lemire, \emph{Simple $C_n$-modules with multiplicities $1$ and application}, Canadian Journal of Physics \textbf{72} (1994), 326--335.




\bibitem[D]{D} M. Duflo, \emph{Sur la classication des id\'eaux primitifs dans l'alg\`ebre enveloppante d'une alg\`ebre de Lie semisimple}, Ann. of Math. \textbf{105} (1977), 107–120.

\bibitem[DP]{DP} I. Dimitrov, I. Penkov, \emph{Weight modules of direct limit Lie algebras}, IMRN \textbf{5} (1999), 223-249.



\bibitem[F]{F}S. L. Fernando, Lie algebra modules with finite-dimensional weight spaces. I, Trans. Amer.
Math. Soc. \textbf{322} (1990), 757–781.

\bibitem[GP]{GP} D. Grantcharov, I. Penkov, \emph{Simple bounded weight modules of $\mathfrak{sl}(\infty)$, $\mathfrak{o}(\infty)$, $\mathfrak{sp}(\infty)$}, \texttt{arXiv}: 1807.01899, preprint, 2018.








\bibitem[H]{H} J. Humphreys, \emph{Representations of semisimple Lie algebras in the \textup{BGG} category $\mathcal{O}$}, Graduate Studies in Mathematics \textbf{94}, AMS, 1991.


\bibitem[IMR]{IMR} N. Iorgov, A. Molev, E. Ragoucy, \emph{Casimir elements from the Brauer–Schur–Weyl duality}, Journal of Algebra \textbf{387} (2013), 144-159.

\bibitem[J1]{J1} A. Joseph, \emph{On the associated variety of a primitive ideal}, Journal of Algebra \textbf{93} (1985), 509–523.


\bibitem[J2]{J2} A. Joseph, \emph{A characteristic variety for the primitive spectrum of the enveloping algebra of a semisimple Lie algebra}, In:  Non-Commutative Harmonic Analysis, Lecture Notes in Mathematics \textbf{587},  New York, Springer, 1978, 116--135.


\bibitem[K]{K} D. Knuth, \emph{The art of computer programming. Fundamental algorithms, Volume 3, Sorting and searching}, Addison--Wesley Series in Computer Science and Information Processing, Addison--Wesley, 1973.


\bibitem[KL]{KL} D. Kazhdan, G. Lusztig, \emph{Representations of Coxeter groups and Hecke algebras},
Invent. Math. \textbf{53} (1979), 165–184.


\bibitem[M]{M} O. Mathieu, \emph{Classification of irreducible weight modules}, Annales de l’institut Fourier \textbf{50} (2000),  537--592.

\bibitem[Mo]{Mo} A. Molev, \emph{On Gelfand--Tsetlin bases for representations of classical Lie algebras}. In: Formal Power Series and Algebraic Combinatorics, 12th International Conference, 2000,  300--308.

\bibitem[MR]{MR}  J. McConnell, J. Robson, \emph{Noncommutative N\"otherian rings}, Graduate Studies in Mathematics \textbf{30}, AMS, 1987.

\bibitem[PP1]{PP1} I. Penkov, A. Petukhov, \emph{On ideals in the enveloping algebra of a locally simple Lie algebra}, Int. Math. Res. Notices \textbf{2015} (2015), 5196--5228.

\bibitem[PP2]{PP2} I. Penkov, A. Petukhov, \emph{Annihilators of highest weight $sl(\infty)$-modules}, Transformation Groups \textbf{21} (2016), 821--849.

\bibitem[PP3]{PP3} I. Penkov, A. Petukhov,\emph{ On ideals in $U(\mathfrak{sl}(\infty))$, $U(\mathfrak{o}(\infty))$, $U(\mathfrak{sp}(\infty))$}. In: Representation theory --- current trends and perspectives, EMS Series of Congress Reports, EMS, 2016, 565--602.

\bibitem[PP4]{PP4} I. Penkov, A. Petukhov, \emph{Primitive ideals of $U(\mathfrak{sl}(\infty))$}, Bulletin LMS \textbf{50} (2018), 443--448.


\bibitem[PP5]{PP5} I. Penkov, A. Petukhov, \emph{Primitive ideals of $U(\mathfrak{sl}(\infty))$ and the Robinson--Schensted algorithm at infinity}, Representation of Lie Algebraic Systems and Nilpotent orbits, Progress in Mathematics \textbf{330}, Birkhauser, 471--499.




\bibitem[PS]{PS}I. Penkov, V. Serganova, \emph{On bounded generalized Harish-Chandra modules}, Annales de l'institut Fourier \textbf{62} (2012), 477--496.


\bibitem[Zh1]{Zh1} A. Zhilinskii, \emph{Coherent systems of representations of inductive families of simple
complex Lie algebras} (in Russian), preprint of Academy of Belarussian SSR, ser. 38
(438), Minsk, 1990.

\bibitem[Zh2]{Zh2}  A. Zhilinskii, \emph{Coherent finite-type systems of inductive families of non-diagonal
inclusions} (in Russian), Dokl. Acad. Nauk Belarusi \textbf{36} (1992), 9–13.

\bibitem[Zh3]{Zh3} A. Zhilinskii, \emph{On the lattice of ideals in the universal enveloping algebra of a diagonal Lie algebra}, preprint, Minsk, 2011.






\end{thebibliography}
\end{document}